\documentclass[11pt]{amsart}

\usepackage{graphicx}
\usepackage{multicol,multirow}
\usepackage{amsmath,amssymb,amsfonts}
\usepackage{mathrsfs}
\usepackage{rotating}
\usepackage{appendix}
\usepackage[numbers]{natbib}

\usepackage{amsmath, amsfonts, amssymb,color, appendix, epsfig,enumerate,mathrsfs, hyperref}
\usepackage[all,cmtip]{xy}

\usepackage[pagewise]{lineno}

\def\k{\kappa}

\newcommand{\mC}{\mathcal{C}}

\newcommand{\mE}{\mathcal{E}}

\newcommand{\mN}{\mathcal{N}}
\newcommand{\mO}{\mathcal{O}}

\newcommand{\ff}{\mathfrak{f}}

\newcommand{\fh}{\mathfrak{h}}

\newcommand{\fm}{\mathfrak{m}}

\newcommand{\bfC}{\mathbf{C}}

\newcommand{\bfF}{\mathbf{F}}

\newcommand{\bfQ}{\mathbf{Q}}
\newcommand{\bfR}{\mathbf{R}}

\newcommand{\bfT}{\mathbf{T}}

\newcommand{\bfZ}{\mathbf{Z}}

\newcommand{\Oo}{\mathcal{O}}

\newcommand{\ov}{\overline}

\newcommand{\be}{\begin{equation}}
\newcommand{\ee}{\end{equation}}
\newcommand{\bes}{\begin{equation*}}
\newcommand{\ees}{\end{equation*}}

\newcommand{\bs}{\begin{split}}
\newcommand{\es}{\end{split}}
\newcommand{\bss}{\begin{split*}}
\newcommand{\ess}{\end{split*}}

\newcommand{\bmat}{\left[ \begin{matrix}}
\newcommand{\emat}{\end{matrix} \right]}
\newcommand{\bsmat}{\left[ \begin{smallmatrix}}
\newcommand{\esmat}{\end{smallmatrix} \right]}

\newcommand{\bml}{\begin{multline}}
\newcommand{\eml}{\end{multline}}
\newcommand{\bmls}{\begin{multline*}}
\newcommand{\emls}{\end{multline*}}

\newcommand{\coh}{H}
\DeclareMathOperator{\ad}{ad}
\DeclareMathOperator{\alg}{alg}

\DeclareMathOperator{\Cl}{Cl}

\DeclareMathOperator{\cris}{cris}

\DeclareMathOperator{\diag}{diag}

\DeclareMathOperator{\End}{End}
\DeclareMathOperator{\ev}{ev}
\DeclareMathOperator{\Ext}{Ext}
\DeclareMathOperator{\Fil}{Fil}

\DeclareMathOperator{\Frob}{Frob}
\DeclareMathOperator{\Gal}{Gal}
\DeclareMathOperator{\GL}{GL}
\DeclareMathOperator{\GSp}{GSp}

\DeclareMathOperator{\Hom}{Hom}

\DeclareMathOperator{\LMFDB}{LMFDB}

\DeclareMathOperator{\Mat}{Mat}

\DeclareMathOperator{\semi}{ss}
\DeclareMathOperator{\SL}{SL}
\DeclareMathOperator{\Sp}{Sp}

\DeclareMathOperator{\spin}{spin}

\DeclareMathOperator{\Symm}{Sym^2}

\DeclareMathOperator{\Tr}{Tr}
\DeclareMathOperator{\tor}{tor}

\DeclareMathOperator{\un}{un}
\DeclareMathOperator{\val}{val}

\newcommand{\tr}{\textup{tr}\hspace{2pt}}

\newcommand{\invlim}{\mathop{\varprojlim}\limits}
\newcommand{\dirlim}{\mathop{\varinjlim}\limits}

\theoremstyle{oupplain}
\newtheorem{theorem}{Theorem}[section]
\newtheorem{lemma}[theorem]{Lemma}
\newtheorem{corollary}[theorem]{Corollary}
\newtheorem{prop}[theorem]{Proposition}
\newtheorem{assumption}[theorem]{Assumption}
\theoremstyle{oupdefinition}
\newtheorem{definition}[theorem]{Definition}
\theoremstyle{oupremark}
\newtheorem{rem}[theorem]{Remark}
\newtheorem{example}[theorem]{Example}
\theoremstyle{oupproof}

\numberwithin{equation}{section}

\usepackage[margin=1in]{geometry}

\begin{document}

\title{Klingen Eisenstein series congruences and modularity}

\author[Tobias Berger]{Tobias Berger}
\author[Jim Brown]{Jim Brown}
\author[Krzysztof Klosin]{Krzysztof Klosin}
\thanks{The third author was supported by the Travel Support for Mathematicians Gift (MP-TSM-00002098) from the Simons Foundation and by a
PSC-CUNY award jointly funded by the Professional Staff Congress and the City University of New York.}

\maketitle

\begin{abstract}
    We construct a mod $\ell$ congruence between a Klingen Eisenstein series (associated to a classical newform $\phi$ of weight $k$) and a Siegel cusp form $f$  with irreducible Galois representation. We use this congruence to show non-vanishing of the Bloch-Kato Selmer group $H^1_f(\bfQ, \ad^0\rho_{\phi}(2-k)\otimes \bfQ_{\ell}/\bfZ_{\ell})$ under certain assumptions and provide an example. We then prove an $R=dvr$ theorem for the Fontaine-Laffaille universal deformation ring of $\ov{\rho}_f$ under some assumptions, in particular,  that the residual Selmer group $H^1_f(\bfQ, \ad^0\ov{\rho}_{\phi}(k-2))$ is cyclic. For this we prove a result about extensions of Fontaine-Laffaille modules. We end by formulating conditions for when $H^1_f(\bfQ, \ad^0\ov{\rho}_{\phi}(k-2))$ is non-cyclic and the Eisenstein ideal is non-principal.
\end{abstract}

\section{Introduction}
The construction of Eisenstein congruences has a long and consequential history. Interesting in their own right, their significance is amplified by the existence of Galois representations attached to the congruent forms, as the ones attached to Eisenstein series are always reducible while the ones attached to cusp forms are often irreducible.
Using various generalizations of the result known as Ribet's Lemma, they lead to the construction of non-zero elements in Selmer groups. This direction was first explored by Ribet himself in the context of the group $\GL_2$ in \cite{RibetInvent76} and later used by many other authors in a variety of different settings e.g. \cite{Wiles90}, \cite{BrownCompMath07},  \cite{SkinnerUrbanMC}. %leading to such breakthroughs like the proof of various Iwasawa Main Conjectures or new cases of the Bloch-Kato Conjecture.

In a different direction, such congruences can play a crucial role in proving modularity of deformations  of  reducible residual Galois representations $\ov{\rho}$, see e.g. \cite{SkinnerWiles97}, \cite{BergerKlosin13}, \cite{BergerKlosin20}, \cite{BergerKlosin23}, \cite{Calegari06}, \cite{WWE20}, and \cite{Wake23}. %Indeed, given a Galois representation $\rho: G_F\to \GL_n(\ov{\bfQ}_{\ell})$ of the absolute Galois group $G_F$ of a number field $F$, unramified away from finitely many primes, one is interested in knowing if this representation arises from an automorphic form on an appropriate algebraic group (i.e., is \emph{modular}). This is usually achieved by constructing a Galois-stable lattice inside the space of $\rho$ and studying deformations of the mod $\ell$ reduction $\ov{\rho}: G_F\to \GL_n(\ov{\bfF}_{\ell})$ of $\rho$.
In \cite{Calegari06} Calegari introduced a method of proving modularity assuming $\ov{\rho}$ is unique up to isomorphism, which relies on proving the principality of the ideal of reducibility of the universal deformation ring $R$ of $\ov{\rho}$. This method was developed further by Berger and Klosin \cite{BergerKlosin11, BergerKlosin13, BergerKlosin20} and Wake and Wang-Erickson \cite{WWE20} and successfully applied in many contexts (see also \cite{Akers24, Huang24}). It relies heavily on the ideas of Bellaiche and Chenevier \cite{BellaicheChenevierbook} and their study of Generalized Matrix Algebras (GMAs).

In this paper we pursue both of these directions in the case of Klingen Eisenstein series of level one on the group $\Sp_4$. More precisely, let $k\geq 12$ be an even  integer and $\phi$ a classical weight $k$ Hecke eigenform of level $1$ (i.e., on the group $\GL_{2/\bfQ}$). Write $E_{\phi}^{2,1}$ for the (appropriately normalized) Klingen Eisenstein series on $\Sp_4$ induced from $\phi$. It is a Siegel modular form of weight $k$ and full level. Congruences between Klingen Eisenstein series and cusp forms have been studied previously by Kurokawa \cite{Kurokawa1,Kurokawa2}, Katsurada and Mizumoto \cite{KatsuradaMizumoto, MizumotoCongruences}, Takeda \cite{Takeda}, and Urban (unpublished).   Katsurada and Mizumoto obtain congruences as an application of the doubling method. In this paper, we produce congruences via a much shorter argument using results of Yamauchi \cite{Yamauchi}. The trade-off is that while our proof is much shorter, we obtain congruences only modulo a prime $\ell$ whereas Katsurada and Mizumoto obtain congruences modulo powers of $\ell$. However, the hypotheses required for our result are different and less restrictive than those needed in \cite{KatsuradaMizumoto}.  We show that under certain conditions $E_{\phi}^{2,1}$ is congruent to some cusp form $f$ of the same weight and level with irreducible Galois representation (Theorem \ref{cong 1}).  This is the first main result of the paper.  These congruences are governed by the numerator  of the (algebraic part) of the symmetric square $L$-function $L_{\rm alg}(2k-2, {\rm Sym}^2\phi)$ of $\phi$. We also exhibit a concrete example when the assumptions of Theorem \ref{cong 1} are satisfied (see Example \ref{ex 1}).

We then proceed to show that these congruences give rise (under some assumptions) to non-trivial elements in the Selmer group $H_{2-k}:=H^1_f(\bfQ, \ad \rho_{\phi}(2-k)\otimes \bfQ_{\ell}/\bfZ_{\ell})$. Here $\rho_{\phi}$ is the Galois representation attached to $\phi$ by  Deligne and we use the Fontaine-Laffaille condition at $\ell$. Assuming the Vandiver Conjecture for  $\ell$ we also deduce the non-triviality of the Selmer group $H^1_f(\bfQ, \ad^0 \rho_{\phi}(2-k)\otimes \bfQ_{\ell}/\bfZ_{\ell})$ (Corollary \ref{BK conj1} and Remark \ref{rem5.8}). This is our second main result and gives evidence for new cases of the Bloch-Kato conjecture. This conjecture was studied for other twists of $\ad \rho_\phi$ by \cite{DiamondFlachGuoAnnScEcole04} and \cite{KlosinAnnInstFourier2009}. In \cite{Urban2001} Urban assumed the existence of Klingen Eisenstein congruences to prove a result towards the main conjecture of Iwasawa theory for the adjoint $L$-function.

To properly analyze these Selmer groups we require some results on extensions of Fontaine-Laffaille modules whose proofs appear to be absent in the literature. In Section \ref{Fontaine-Laffaille} we carefully study certain aspects of Fontaine-Laffaille theory, in particular, prove the Hom-tensor adjunction formula and give a precise definition of Selmer groups with coefficient rings of  finite length.

Given the eigenvalue congruence $E^{2,1}_{\phi}\equiv f$ (mod $\ell$) we also study deformations of a {non-semi-simple}  Galois representation $\ov{\rho}: G_{\bfQ} \to \GL_4(\ov{\bfF}_{\ell})$ whose semi-simplification arises from the Klingen Eisenstein series. Such a representation is reducible with two 2-dimensional Jordan-Holder blocks and more precisely one has $$\ov{\rho} = \bmat \ov{\rho}_{\phi} & * \\ & \ov{\rho}_{\phi}(k-2)\emat. $$ Conjecturally such representations should arise as mod $\ell$  reductions of Galois representations attached to Siegel cusp forms  which are congruent to $E_{\phi}^{2,1}$ mod $\ell$. We assume that $\dim H_{2-k}[\ell]=1$, where $[\ell]$ indicates $\ell$-torsion. This can be seen as a  refinement of the uniqueness assumption of \cite{SkinnerWiles97} similar to the one in \cite{BergerKlosin13} and as in \cite{BergerKlosin13,Calegari06} we prove the principality of the reducibility ideal of the universal deformation. However, this principality cannot be achieved through the method of \cite{BergerKlosin13} because the representation in question fails to satisfy the strong self-duality property required for the method of [loc.cit.]. Instead we improve on a recent result of Akers \cite{Akers24} which replaces the self-duality condition with a one-dimensionality assumption on the Selmer group $H_{k-2}:=H^1_f(\bfQ, \ad \ov{\rho}_{\phi}(k-2))$ of the `opposite' Tate twist of $\ad \rho_{\phi}$. With these assumptions in place we are able to show that the universal deformation ring $R$ is a discrete valuation ring and prove a modularity result guaranteeing that the unique deformation of $\ov{\rho}$ indeed arises from a Siegel cusp form congruent to $E_{\phi}^{2,1}$ (Theorem \ref{main}). This is the third main result of the paper.

We then proceed to formulate conditions for non-cyclicity of the Selmer group $H_{k-2}$. While many results in the literature give bounds on the orders of Selmer groups (in particular, Corollary \ref{BK conj1} gives such a lower bound on $H_{2-k}$), the structure of these groups is notoriously mysterious. In this paper we prove  that if  the  (local) Klingen Eisenstein ideal $J_{\fm}$ is not principal then $H_{k-2}$ is not cyclic (Corollary \ref{not cyclic}). We further refine this result by providing a criterion for non-principality in terms of  the depth of congruences between cusp forms and $E^{2,1}_{\phi}$ (Corollary \ref{not cyclic 2}).  An intriguing feature of these results is that $H_{k-2}$ is non-critical, i.e. this Selmer group is not controlled by a critical $L$-value in the sense of Deligne.

The authors would like to thank Jeremy Booher and Neil Dummigan for helpful discussions.

\section{Background and notation}\label{sec:intro}
Given a field $F$ we denote by $G_F$ its absolute Galois group. Fix a rational prime $\ell>2$.
%\com{K: This notation may be ambigous if we don't say characteristic zero. Maybe: Given a field $F$ we denote by $G_F$ its absolute Galois group?\\J: good with this change.\\T: made the change}
If $M$ is a topological  $\bfZ_\ell[G_{F}]$-module we will write $M(n) = M \otimes \epsilon^{n}$ for the $n$-th Tate twist where $\epsilon$ denotes the $\ell$-adic cyclotomic character.

%\com{T: do we need the subscript $\ell$? \\J: I think we can drop it if we fix $\ell$ above as Kris is suggesting.}

For each prime $p$, we fix an embedding $\ov{\bfQ} \hookrightarrow \ov{\bfQ}_{p}$.
This is equivalent to choosing a prime $\ov{p}$ of $\ov{\bfQ}$ lying over $p$ and fixes an isomorphism $D_p \cong G_{\bfQ_p}$, where $D_p$ is the decomposition group of $\ov{p}$. We will denote by $I_p\subset D_p$ the corresponding inertia group. We also fix an isomorphism $\ov{\bfQ}_{\ell} \cong \bfC$.

Let $E$ denote a finite extension of $\bfQ_{\ell}$ with valuation ring $\Oo$, uniformizer $\lambda$ and residue field $\bfF$.
For a continuous homomorphism $\rho:G_F \to \GL_n(\Oo)$ we write $\ov{\rho}: G_F\to \GL_n(\bfF)$ for the mod $\lambda$ reduction of $\rho$.

% Let $\Mat_{n}$ denote the $n$ by $n$ matrices.
%\com{K: May be nicer to say:}
For $n\in \bfZ_+$, we denote by $\Mat_n$ (resp. $\GL_n$) the affine group scheme over $\bfZ$ of $n\times n$ (resp. invertible) matrices.
%\com{K: End suggested change. J: Removed the previous wording in favor of this. }
%For $\gamma \in \Mat_{n}$, we write $|\gamma|$ for the determinant of $\gamma$.
Given a matrix $\gamma \in \Mat_{2n}$, we will write it as
$\gamma = \bmat a_{\gamma} & b_{\gamma}\\ c_{\gamma} & d_{\gamma} \emat$ where the blocks are in $\Mat_{n}$.
%We will denote by $\GL_n$ the affine group scheme over $\bfZ$ of invertible $n\times n$ matrices. %\com{K: end add. \\J: good with this change.}
Set
$\GSp_{2n} = \left\{g \in \GL_{2n} : \, ^t\!g J_{n} g = \mu_{n}(g) J_{n} , \mu_{n}(g) \in \GL_1\right\},$
where $J_n=\bmat 0_n & -1_n \\ 1_n & 0_n \emat$ where $1_{n}$ is the $n$ by $n$ identity matrix, %\com{K: Should we define $1_n$? J: added}
and $\mu_{n}: \GL_{2n} \rightarrow \GL_{1}$ is the
%\com{K: the?J: changed}
homomorphism defined via the equation given in the definition.
Write $\GSp^{+}_{2n}(\bfR)$ for the subgroup of $\GSp_{2n}(\bfR)$ consisting of elements $g$ with $\mu_{n}(g) > 0$. We set $\Sp_{2n} = \ker (\mu_{n})$ and
\begin{equation*}
\Gamma_n = \Sp_{2n}(\bfZ) = \left\{g \in \GL_{2n}(\bfZ): \, ^t\!g J_{n} g = J_{n} \right\}.
\end{equation*}
Note that $\Sp_2 = \SL_2$,  the subgroup scheme of $\GL_2$ of matrices of determinant one. %\com{K: end add\\J: good with this}

The Siegel upper half-space is given by
\begin{equation*}
\fh_n =  \{z=x+iy \in \Mat_n(\bfC):  x,y\in \Mat_n(\bfR), \, ^t\!z=z, y>0\}
\end{equation*}
where we write $y>0$ to indicate that $y$ is positive definite.
The group $\GSp_{2n}^{+}(\bfR)$ acts on $\fh_{n}$ via
$\gamma z = (a_{\gamma}z + b_{\gamma})(c_{\gamma} z + d_{\gamma})^{-1}$.

For a function  $f: \fh_{n} \rightarrow \bfC$ set
$(f|_{\k} \gamma)(z) = \mu_{n}(\gamma)^{nk/2} j(\gamma,z)^{-k} f(\gamma z)$
for $\gamma \in \GSp_{2n}^{+}(\bfR)$ and $z \in \fh_{n}$ where $j(\gamma,z) = \det(c_{\gamma} z+ d_{\gamma})$.
A Siegel modular form of weight $k$ and level $\Gamma_{n}$ is a holomorphic function $f: \fh_{n} \rightarrow \bfC$ satisfying
$(f|_{k}\gamma)(z) = f(z)$
for all $\gamma \in \Gamma_{n}$.  If $n=1$, we also require the standard growth condition at the cusp.  We denote the
%\com{K: added $\bfC$, $k$ and level}
$\bfC$-vector space of Siegel modular forms of weight $k$ and level $\Gamma_n$ as $M_{k}(\Gamma_n)$.   Any $f \in M_{k}(\Gamma_{n})$ has a Fourier expansion of the form
\begin{equation*}
f(z)=\sum_{T\in\Lambda_n} a(T;f) e(\Tr(Tz))
\end{equation*}
where $\Lambda_n$ is defined to be the set of $n$ by $n$ half-integral (diagonal entries are in $\bfZ$,
%\com{K: integral may mean a lot of things. Should we say that they lie in $\bfZ$?\\J: changed it. Is that good?}
off diagonal are allowed to lie in $\frac{1}{2}\bfZ$) positive semi-definite symmetric matrices and $e(w):=e^{2\pi i w}$.  Given a ring $A \subset \bfC$,
%\com{K: I suggest using $A$ instead of $\Oo$. $\Oo$ later (and possibly in the Introduction denotes a concrete ring. Besides, should it be a subring of $\bfC$?\\J: I'm okay with this change, and probably subring of $\bfC$.}
we write $f \in M_{k}(\Gamma_{n};A)$ if $a(T;f) \in A$ for all $T \in \Lambda_{n}$. Define the subspace $S_k(\Gamma_n)=\ker \Phi\subset M_k(\Gamma_n)$ of \emph{cusp forms}, where
%We define the Siegel operator $\Phi: M_{k}(\Gamma_{n}) \rightarrow M_{k}(\Gamma_{n-1})$ by
$\Phi(f)(z) = \lim_{t \rightarrow \infty} f\left(\bmat z & 0 \\ 0 & it \emat \right).$
%We say $f \in M_{k}(\Gamma_n)$ is a cuspform if $\Phi(f) = 0$.  Set $S_{k}(\Gamma_n)= \ker(\Phi)$.
%\com{K: This may be standard, I don't know. But the notation ker is usually reserved for homomorpshisms. Is $\Phi$ a homomorphism?\\J: It is a linear map, so I think $\ker$ is appropriate.}

%\com{K: Maybe: We will now introduce certain Eisenstein series, which will play a prominent role in this paper?\\J: changed}
We will now introduce certain Eisenstein series, which will play a prominent role in this paper.  For $n \geq 1$ and $0 \leq r \leq n$ define the parabolic subgroup
\begin{equation*}
P_{n,r} = \left\{\bmat a_1 & 0 & b_1 & * \\ * & u & * & * \\ c_1 & 0 & d_1 & * \\ 0 & 0 & 0 & ^t\!u^{-1} \emat \in \Gamma_{n}: \bmat a_1 & b_1 \\ c_1 & d_1 \emat \in \Gamma_{r}, u \in \GL_{n-r}(\bfZ) \right\}.
\end{equation*}
We define projections
$\star: \fh_{n} \to \fh_{r}$, $z = \bmat z^{\star} & * \\ * &* \emat \mapsto z^{\star}$`
and $\star: P_{n,r} \to\Gamma_{r}$, $\gamma \mapsto \gamma^{\star} = \bmat a_1 & b_1 \\ c_1 & d_1 \emat.$

 %\com{K: Could we mumble something about the abuse of notation? We are using the same notation for two dfifferent maps. I am also not crazy about the superscript star on $z$. \\J: I'm fine with mentioning the abuse of notation.  I'm also fine with changing the notation, but the superscript star is standard notation for this in the literature.}
 %These allow us to define the Eisenstein series of interest.
 Let $\phi \in S_{k}(\Gamma_{1})$.  The Klingen Eisenstein series attached to $\phi$ is the series
 \begin{equation*}
E^{2,1}_{\phi}(z) = \sum_{\gamma \in P_{2,1} \backslash \Gamma_{2}} \phi((\gamma z)^{\star}) j(\gamma, z)^{-k}
\end{equation*}
where $z \in \fh_{2}$. The Eisenstein series converges for $k \geq 12$, see \cite{klingen} Theorem 1 page 67 for example.
%\com{K: Can we be more precise? This is a very broad statement. Convergence presumably requires the weight to be sufficiently high, which we have here. \\J: Updated.}
%\com{K: Maybe `one can check' below? Or is it really that obvious?\\J: I added a reference. I have no idea why it isn't compiling correctly though. ugh\\T: looks fine to me}
Note that \cite{klingen} Proposition 5 page 68 gives $\Phi(E_{\phi}^{2,1}) = \phi$.

Given two Siegel modular forms $f_1, f_2 \in M_{k}(\Gamma_n)$ with at least one a cusp form, set
\begin{equation*}
\langle f_1, f_2 \rangle = \int_{\Gamma_{n}\backslash\fh_n}f_1(z)\overline{f_2(z)}(\det y)^{k} d\mu z,
\end{equation*}
%\com{K: What is $\Gamma$ here? Do we need to divide by $[\Gamma_n:\Gamma]$ in front to make it independent of the choice of $\Gamma$? Or should it be $\Gamma_n$ everywhere?\\J: It can be $\Gamma_n$ since that is all we are working with.  Changed it. }
where $z=x+iy$ with $x = (x_{\alpha, \beta})$, $y = (y_{\alpha, \beta}) \in \Mat_{n}(\bfR)$,
$d\mu z = (\det y)^{-(n+1)} \prod_{\alpha \leq \beta} d x_{\alpha,\beta} \prod_{\alpha \leq \beta} dy_{\alpha,\beta}$ with $dx_{\alpha,\beta}$ and $dy_{\alpha,\beta}$ the usual Lebesgue measure on $\bfR$.
% Let $N_{k}(\Gamma_{n}):=S_{k}(\Gamma_{n})^{\perp}$ be the orthogonal complement to $S_k(\Gamma_{n})$ with respect to this inner product.

Given $\gamma \in \GSp_{2n}^{+}(\bfQ)$, we write $T(\gamma)$ to denote
 the double coset $\Gamma_{n} \gamma \Gamma_{n}$ and set
$T(\gamma) f = \sum_{i} f|_{k} \gamma_{i}$
where the $\gamma_i$ are given by the finite decomposition $\Gamma_{n} \gamma \Gamma_{n} = \coprod_{i} \Gamma_{n} \gamma_{i}$ and $f \in M_{k}(\Gamma_n)$.  Let $m >1$.  We define $T^{(n)}(m)$ via
\begin{equation*}
    T^{(n)}(m) = \sum_{\substack{d_1e_1 = \cdots = e_n d_n = m\\ d_1 \mid d_2 \mid \cdots \mid d_n \mid e_n \mid e_{n-1} \mid \cdots \mid e_1}} T(\diag(d_1, \dots, d_n, e_1, \dots, e_n)).
\end{equation*}
In particular, for $p$ a prime we have
\begin{equation*}
    T^{(n)}(p) = T(\diag(1_{n}, p1_{n})).
\end{equation*}
We also define
\begin{equation*}
T_{i}^{(n)}(p^2)  = T(\diag(1_{n-i}, p1_{i}, p^2 1_{n-i}, p1_{i})), \quad 1\leq i\leq n.
\end{equation*}
The spaces $M_{k}(\Gamma_{n})$ and $S_{k}(\Gamma_{n})$ are both stable under the action of $T^{(n)}(p)$ and $T_{i}^{(n)}(p^2)$ for $1 \leq i \leq n$ and all $p$. We say a nonzero $f \in M_{k}(\Gamma_{n})$ is an eigenform if it is an eigenvector of $T^{(n)}(p)$ and $T_{i}^{(n)}(p^2)$ for all $p$ and all $1 \leq i \leq n$.  As we will be focused on the case $n=2$, we specialize to that case. We let $\bfT'$ denote the $\bfZ$-subalgebra of $\End_{\bfC}(S_{k}(\Gamma_{2}))$ generated by the Hecke
%\com{K: I think we haven't called them `Hecke' so far\\J: added above}
operators $T^{(2)}(p)$ and $T^{(2)}_1(p^2)$ for all primes $p$.

%\com{K: Should it say ``for all primes $p$?}
% \com{K: SHould the second one be $T^{(2)}_1(p^2)$? Since we never use any other operators, could we simplify this notation to $T(p)$ and $T_1(p^2)$? Similarly could we write $\lambda_f^1(p^2)$ for $\lambda_p(p;T_1^{(2)}(p^2))$?\\
% K: Also, I think for the $\ell$-adic algebra I would like to exclude the operators at $\ell$ because Theorem \ref{Weissauer} doesn't give us anything for them. For here it makes sense to include all Hecke operators, but I think once we get to the modularity section I need to exclude these at $\ell$\\
%J: Yes for the first question. I adjusted it.  And sounds good for excluding at $\ell$.}

Recall that $E/\bfQ_{\ell}$ denotes a finite extension with valuation ring $\mO$ %\com{K: I think later the ring of integers of $E$ is denoted by $\Oo$.\\J: It is $\mO_{E}$ throughout the next section.  I changed it here and will do the next section as well.}
and uniformizer $\lambda$. Given eigenforms $f_1, f_2 \in M_{k}(\Gamma_n;\mO)$, following the notation in \cite{Yamauchi}
%\com{K: For this to make sense we need to first fix an embedding $\ov{\bfQ}_{\ell} \hookrightarrow \bfC$\\J:  Above we fixed such embeddings for each prime $p$. this is right before 2.2\\T: we only fix $\ov{\bfQ} \hookrightarrow \ov{\bfQ}_{p}$ for every prime $p$ there, but could add $\ov{\bfQ}_{\ell} \hookrightarrow \bfC$ at that point?}
we write $f_1 \equiv_{\ev} f_2 \pmod{\lambda}$ if $\lambda_{f_1}(T) \equiv \lambda_{f_2}(T) \pmod{\lambda}$ for all $T \in \bfT'$ where $Tf_{i} = \lambda_{f_i}(T)f_i$.

For an eigenform $\phi\in S_k(\Gamma_1)$ we set
\begin{equation*}
    \begin{split} L(s,\phi) := &\prod_p(1-\lambda_{\phi}(p)p^{-s}+p^{k-1-2s})^{-1},\\
    L(s, \Symm \phi)=&\prod_p \left[(1-\alpha_p^2 p^{-s})(1-\alpha_p \beta_p p^{-s}) (1-\beta_p^2 p^{-s}) \right]^{-1}\end{split}
\end{equation*}
%Note that in this case \begin{equation*}L_{p}(p^{-s},\phi;\spin) = (1 - \lambda_{\phi}(p)p^{-s} + p^{k-1-2s})\end{equation*}
where  $\lambda_{\phi}(p)$ is the eigenvalue of $T(p):=T^{(1)}(p)$ corresponding to $\phi$ and $\alpha_p, \beta_p$ denote the roots of $X^2-\lambda_{\phi}(p)X+p^{k-1}$. %\com{K: Same comment about the eigenvalues. I think it should be `for the eigenvalue of $T(p)$ corresponding to the eigenform $\phi$\\J: changed}
The symmetric square $L$-function converges in the right half-plane $\Re(s) > k$, satisfies a functional equation, and has analytic continuation to the entire complex plane.

For an eigenform $f\in S_k(\Gamma_2)$ we define
\begin{align*}
    L_{p}(X,f,\spin) = &(1-\lambda_{f}(p)X + (\lambda_f(p)^2 - \lambda_f(p^2) - p^{2k-4})X^2\\
    &-\lambda_{f}(p) p^{2k-3}X^3 + p^{4k-6}X^4)
\end{align*}
where we write $\lambda_f(p)$ is the eigenvalue of $T^{(2)}(p)$ corresponding to $f$ and $\lambda_{f}(p^2)$ for the eigenvalue $T^{(2)}(p^2)$ corresponding to $f$.
%\com{J: I changed the order because I'm not sure why the $n=2$ case was after the symmetric square. I adjusted the eigenvalue language.}

%In the $n=1$ case we will also be interested in the symmetric square $L$-function attached to $\phi$:\begin{equation*}L(s,\Symm \phi) = \prod_{p} L_{p}(p^{-s},\Symm \phi)^{-1}\end{equation*}where \com{K: let's find a classical definition of the sym square l-function}\begin{equation*}L_{p}(X,\Symm \phi) = (1 - \widetilde{\alpha_0}^2 X) (1- \widetilde{\alpha_0}^2 \alpha_1 X)(1- \widetilde{\alpha_0}^2 \alpha_1^2 X).\end{equation*}

%\com{K: same comment about eigenvalues and maybe again we should say something about convergence, etc.\\J: added}

%The following result of Laumon and Weissauer attaches an $\ell$-adic Galois representation to any eigenform $f\in S_k(\Gamma_2)$.
%\com{K: Maybe better: The following result of Laumon and Weissauer attaches an $\ell$-adic Galois representation to any eigenform $f\in S_k(\Gamma_4)$.\\J: changed\\T: corrected this to $\Gamma_2$}

\begin{theorem}[\cite{WeissauerAsterique05} Theorem 1] \label{Weissauer} Let $f \in S_{k}(\Gamma_2)$ be an eigenform.
%\com{K: This is pedantic, but we haven't called them Hecke eigenforms, just eigenforms.\\J: changed}
  For a sufficiently large finite extension $F/\bfQ_{\ell}$ one has $L_{p}(X,f, \spin) \in F[X]$ for all primes $p\neq \ell$ and there is a semisimple continuous representation %\com{Maybe `continuous homomorphism'?\\J: the wording I used was Weissauer's, but I can change it.  Is the suggestion replace `Galois representation' with `continuous homomorphism'?\\K: Yes.}
    $\rho_{f}: G_{\bfQ} \rightarrow \GL_4(F)$
which is unramified outside of $\ell$ so that for $p \neq \ell$ one has
    $L_p(X,f;\spin) = \det(1 - \rho_{f}(\Frob_p)X)$.
%\com{T: should there not be an inverse on the RHS as $L_q$ is a polynomial?\\J:  Removed it. }
% \com{K: COuld one replace $p^{-s}$ above by $X$? I guess this is saying that it is true for any $s$? So, the answer would be yes?\\ J: Yes, I think so.  This is just stated exactly how it appears in Weissauer.  Do you want to change it?\\
% K: Did that.}
%The eigenvalues of $\rho_{f}(\Frob_q)$ are algebraic integers for $q \neq \ell$.
\end{theorem}

\section{Congruence} \label{s3}

We keep the notation of Section \ref{sec:intro}. Throughout this section we fix an even weight $k \geq 12$ and an odd prime $\ell$ and make the following assumption.
\begin{assumption} \label{admis}
    Given an even weight $k \geq 12$ and prime $\ell$,  assume that $E/\bfQ_\ell$ is sufficiently large to contain the fields $F$ from Theorem \ref{Weissauer} for all forms $f \in S_k(\Gamma_2)$. We also assume that for every eigenform $\phi\in S_k(\Gamma_1)$ the field $E$ contains all the Hecke  eigenvalues of  $\phi$ as well as the value  $L_{\rm alg}(2k-2, \Symm\phi)$ (see \eqref{symalg} for the definition). In addition we suppose that $E$ contains a primitive cube root of unity.
%\com{T: and symmetric square L-value. \\J: but we don't want to fix a particular $\phi$ here and I don't think we need the symmetric square $L$-value for all $\phi$}
\end{assumption}

%\com{K: Are we placing any restrictions on $k$ here or just a positive integer?\\J: fixed above.}

Recall that we denote the valuation ring of $E$ by $\mO$.
Let $\phi \in S_{k}(\Gamma_1)$ be a normalized eigenform
%\com{K: Eigen? - seems to be needed for eigenvalue congruence\\cusp? - I think we only defined the Eisenstein series for cusp forms\\J: changed}
and consider the Klingen Eisenstein series $E_{\phi}^{2,1}$.  In this section we show under certain conditions that $E_{\phi}^{2,1}$ is eigenvalue-congruent to a cuspidal Siegel modular form with irreducible Galois representation.

Write
\begin{equation*}
    E^{2,1}_{\phi}(z) = \sum_{T \in \Lambda_2} a(T;E^{2,1}_{\phi}) e(\Tr(Tz)).
\end{equation*}
For $T$ that are singular, i.e., $\det T = 0$,  one has $T$ is unimodularly equivalent to $\bmat n & 0 \\ 0 & 0 \emat$ for some $n \in \bfZ_{\geq 0}$. For such $T$, one has $a(T;E^{2,1}_{\phi}) = a(n;\phi)$ where $\phi(z) = \sum_{n > 0} a(n;\phi) e(nz)$.

We use the following result to prove our congruence.

\begin{corollary}[\cite{Yamauchi} Corollary 2.3]\label{corl:yama} Assume $\ell \geq 7$.  Let $g$ be a Hecke eigenform in $M_{k}(\Gamma_2;\mO)$ with Fourier expansion $g(z) = \sum_{T \in \Lambda_2} a(T;g)e(\Tr(Tz))$.
%\com{K: Can we say instead `with Fourier expansion $g(z) = \sum_{T \geq 0} a(T;g)e(\Tr(Tz))$ such that $a(T;g)\in \Oo$' - is that what `defined over $\Oo$' means? Also, now that we fix $\Oo$ at the beginning of the section we shouldn't be redefining it here. Rather we should start the Corollary with ``Assume $E$ contains a primitive third root of unity.'' Or should we make this assumption at the beginning of the section?\\T: We now made this assumption at the end of section 2.\\J: I think the assumption is at the beginning of this section? Also, we have not mentioned $\mO$ in this section so far. In terms of the coefficients, we already defined that we write $g \in M_k(\Gamma_2;\mO)$ to mean that $a(T;g) \in \mO$ in Section 2.  I removed the defined over part. }
Assume that $\lambda \mid a(T;g)$ for all $T$ with $\det T =0$ and that there exists at least one $T > 0$ with $a(T;g) \in \mO^{\times}$.  Then there exists a %non-trivial \com{K: What does non-trivial mean here? $f$ is clearly non-zero because $a(T;f)\neq 0$, but is non-vanishing maybe what we are trying to emphasize here?\\J: It means that $f$ is not congruent to $0$ modulo $\lambda$.}
Hecke eigenform $f \in S_{k}(\Gamma_2;\mO)$ so that $g \equiv_{\ev} f \not \equiv_{\ev} 0 \pmod{\lambda} $.
\end{corollary}

For $T = \bmat m & r/2 \\ r/2 & n \emat$, we say $T$ is primitive if $\gcd(m,n,r) = 1$. We set $\det(2T) = \Delta(T) \ff^2$ for a positive integer $\ff$ and where $-\Delta(T)$ is the discriminant of the quadratic field $\bfQ(\sqrt{-\det(2T)})$.  We set $\chi_{T} = \left(\frac{-\Delta(T)}{\cdot}\right)$, the quadratic character associated to the field $\bfQ(\sqrt{-\det(2T)})$.

Define
$
    \vartheta_{T}(z) = \sum_{a,b \in \bfZ^2} e(z(ma^2 + r ab + nb^2))
    = \sum_{n \geq 0} b(n;\vartheta_{T}) e(nz).
$
Given $v \in \bfZ_{\geq 1}$, set
\begin{equation*}
    \vartheta_{T}^{(v)}(z) = \sum_{n \geq 0} b(v^2 n;\vartheta_{T}) e(nz).
\end{equation*}
One can check
%\com{K: Perhaps again `one can check'? I think it is hard for a reader to note that this is true? Also, has notation $M_k(\Gamma(N))$ been defined? \\J: I changed the wording to one can check and added a def of $\Gamma(N)$ and $M_k(\Gamma(N))$.  We can add more if desired, but I kind of feel that anyone reading this paper is going to know what modular forms of weight $k$ and level $\Gamma(N)$ are.}
that $\vartheta_{T}^{(v)} \in M_1(\Gamma(4\det T))$ where $\Gamma(N) = \ker\left(\SL_2(\bfZ) \rightarrow \SL_2(\bfZ/N\bfZ)\right)$ and $M_k(\Gamma(N))$ denotes the modular forms of weight $k$ and level $\Gamma(N)$. Set
\begin{equation*}
    D(s,\phi,\vartheta_{T}^{(v)}) = \sum_{n \geq 1} a(n;\phi)b(v^2 n; \vartheta_{T}) n^{-s}.
\end{equation*}
We have that $D(s,\phi,\vartheta_{T}^{(v)})$ converges in a right half-plane with meromorphic continuation to the entire complex plane (\cite{ShimuraCommPureAppliedMath1976}).
%\com{K: We should say something about the convergence of $D$.  \\ J: I added something, but it isn't exact. We are only using this in the context of the Fourier coefficient calculation we are citing and there it does not give the exact right half plane where it converges. Since $\phi$ is weight $k$ and $\vartheta$ is weight 1 and this is a convolution, it should be $\Re(s) > (k+3)/2 = (k+1)/2 + (1+1)/2$, but I can't find a reference where this is actually stated.}
Set
\begin{equation} \label{symalg}
    L_{\alg}(2k-2,\Symm \phi) := \frac{L(2k-2,\Symm \phi)}{\pi^{3k-3} \langle \phi, \phi \rangle},
\end{equation}
\begin{equation*}
    L_{\alg}(k-1,\chi_{T}) = \frac{\Delta(T)^{k-3/2} L(k-1,\chi_{T})}{\pi^{k-1}},
\end{equation*}
and
\begin{equation*}
    D_{\alg}(k-1,\phi,\vartheta_{T}^{(v)}) = \frac{D(k-1,\phi,\vartheta_{T}^{(v)})}{\pi^{k-1} \langle \phi, \phi \rangle}.
\end{equation*}
We have each of these terms is algebraic, see  (\cite{ShimuraCommPureAppliedMath1976}, \cite{SturmAJM1980}, \cite{ZagierSpringer77}). %\com{K: We should probably include a reference? Or say `It is well-known that...''?\\ J: added references}
Moreover, we have via \cite{ZagierSpringer77} Equation (22) that if $\ell > k-1$, then $L_{\alg}(k-1,\chi_{T})$ is $\ell$-integral.

\begin{theorem}\cite{MizumotoKodai84}  Let $\phi \in S_{k}(\Gamma_1)$ be a normalized eigenform with a Fourier expansion as above. %\com{K: Above $\phi$ was a cusp form. Do we still want it?\\J: The theorem is true without requiring $\phi$ to be a cusp form, but I changed it to flow better with what we're doing.}
Let $T> 0$ be primitive. %\com{K: Should we give dimensions for $T$? Or have we now fixed that $T$ is $2\times 2$?\\J: it is being used to index Fourier coefficients of a Siegel modular form of genus 2, so by how the Fourier coefficients were defined it should be clear from context I think.}
We have
\begin{align*}
    a(T;E^{2,1}_{\phi}) = &(-1)^{k/2} \frac{(k-1)!}{(2k-2)!} 2^{k-1} \frac{L_{\alg}(k-1,\chi_T)}{L_{\alg}(2k-2,\Symm \phi)} \\
    &\cdot \sum_{\substack{m \mid \ff\\ m > 0}} M_{T}(\ff m^{-1}) \sum_{\substack{t \mid m \\ t > 0}} \mu(t) D_{\alg}(k-1, \phi, \vartheta_{T}^{(m/t)})
\end{align*}
where
%\com{K: What is $\ff$ here?\\J: It is defined above right after Corollary 3.2:  $|2T| = \Delta(T)\ff^2$}
\begin{equation*}
    M_{T}(a) = \sum_{\substack{d \mid a \\ d >0}} \mu(d) \chi_{T}(d) d^{k-2} \sigma_{2k-3}(a d^{-1})
\text{ and }
    \sigma_{s}(d) = \sum_{\substack{g \mid d \\ g > 0}} g^{s}.
\end{equation*}
\end{theorem}

Note that while this theorem is only stated for Fourier coefficients indexed by primitive $T$, we have that Fourier coefficients indexed by non-primitive $T$ are an integral linear combination of Fourier coefficients indexed by primitive $T$ by \cite{MizumotoKodai84} Equation 1.3, so we only need to consider the primitive $T$ to guarantee the hypotheses of Corollary \ref{corl:yama} are satisfied.

\begin{lemma}\label{lem:galirred} Assume $\ell > 4k-7$. Let $f \in S_k(\Gamma_2;\mO)$ be an eigenform. If there exists a normalized eigenform $\phi \in S_k(\Gamma_1;\mO)$ so that  $f \equiv_{\ev} E_{\phi}^{2,1} \pmod{\lambda}$ and that $\ov{\rho}_\phi$ is irreducible, then $\rho_f$ is irreducible.
\end{lemma}

\begin{proof} We know via \cite{WeissauerAsterique05} that if $\rho_{f}$ is reducible, then the automorphic representation associated to $f$ is either CAP or a weak endoscopic lift.  Moreover, by \cite{PitaleSchmidtRamanujan} Corollary 4.5  since $f \in S_{k}(\Gamma_2)$ and $k >2$, the automorphic representation attached to $f$ can be CAP only with respect to the Siegel parabolic, i.e., $f$ is a classical Saito-Kurokawa lift. Suppose that $f$ is a Saito-Kurokawa lift of $\psi \in S_{2k-2}(\Gamma_1)$. Then we have $\ov{\rho}_{f}^{\semi} = \ov{\rho}_{\psi} \oplus \ov{\epsilon}^{k-1} \oplus \ov{\epsilon}^{k-2}$.  %\com{K: We need to either define $\ov{\rho}$ for an arbitrary Galois representation $\rho$, or define $\ov{\epsilon}$ somewhere. Also in other sections we use $\omega$ for that character\\In later sections there are some $\omega_\ell$ and one $\omega$ that need to be replaced by $\ov{\epsilon}$. \\J: I think I got all of them. }
%\com{T: we think it might be good to drop the subscript $\lambda$ for the 2-dimensional modular Galois representations $\rho_{\psi, \lambda}$ as well.\\J:  removed.}
Using the fact that $f \equiv_{\ev} E_{\phi}^{2,1} \pmod{\lambda}$ and that the eigenvalues of $E_{\phi}^{2,1}$ are given by $\lambda(p;E_{\phi}^{2,1}) = a(p;\phi) + p^{k-2}a(p;\phi)$,
%\com{T: Do we need $G$ at all in the proof or could we just assume $f \equiv_{\ev} E^{2,1}_\phi \mod{\lambda}$ in the statement?\\J: Nope, I adjusted it.}
the Brauer-Nesbitt  and Chebotarev Theorems give that $\ov{\rho}_{f}^{\semi} = \ov{\rho}_{\phi} \oplus \ov{\rho}_{\phi}(k-2)$, where recall that we write $\ov{\rho}_{\phi}(k-2)$ for $\ov{\rho}_{\phi} \otimes \ov{\epsilon}^{k-2}$.  This is a contradiction if $\ov{\rho}_{\phi}$ is irreducible. Thus, $f$ cannot be a Saito-Kurokawa lift.  It remains to show that the automorphic representation associated to $f$ is not a weak endoscopic lift.  The possible decompositions of $\rho_{f}$ are given in \cite{SkinnerUrbanJussieu06} Theorem 3.2.1 under the assumption that $\ell > 4k-7$. Of these, the only case remaining to check is Case B(v), which states if $\rho_{f} = \sigma \oplus \sigma'$ with $\sigma$ and $\sigma'$ both 2-dimensional, then $\det(\sigma) = \det(\sigma')$.  In our case, this would require $\det(\rho_{\phi}) = \det(\rho_{\phi}(k-2))$, i.e., $\ov{\epsilon}^{k-1} = \ov{\epsilon}^{2k-3}$, which is impossible by our assumption that $\ell>4k-7$. Thus, $\rho_{f}$ is irreducible.
\end{proof}

\begin{theorem} \label{cong 1} Assume that $\ell > 4k-7$. Let $\phi \in S_{k}(\Gamma_1; \mO)$ be a normalized eigenform.
%\com{K: Previously we required a normalized eigenform. For level one these are the same, so maybe normalized eigenform here?\\J: made the change.}
Suppose that
$\lambda \mid L_{\alg}(2k-2, \Symm \phi)$.  Furthermore, assume there exists $T_0 > 0$ so that \begin{equation*}
    \val_{\lambda}\left(L_{\alg}(2k-2, \Symm \phi)a(T_0,E^{2,1}_{\phi})\right) \leq 0.
\end{equation*}
Then there exists an eigenform $f \in S_{k}(\Gamma_2; \mO)$ so that
\begin{equation*}
    E^{2,1}_{\phi} \equiv_{\ev} f \pmod{\lambda}.
\end{equation*}
If in addition $\ov{\rho}_{\phi}$ is irreducible, then $\rho_{f}$ is irreducible.
\end{theorem}

\begin{proof}  Set $H_{\phi}^{2,1}(z) = L_{\alg}(2k-2,\Symm \phi) E^{2,1}_{\phi}(z)$.
%Observe that for $T>0$, we have
%\begin{align*}
%    a(T;H_{\phi}^{2,1}) &= (-1)^{k/2} \frac{(k-1)!}{(2k-2)!} 2^{k-1} L_{\alg}(k-1,\chi_T) \\
%        &\cdot \sum_{\substack{m \mid \ff\\ m > 0}} M_{T}(\ff m^{-1}) \sum_{\substack{t \mid m \\ t > 0}} \mu(t) D_{\alg}(k-1, \phi, \vartheta_{T}^{(m/t)}).
%\end{align*}
For $T \geq 0$, define
$c(T) = \val_{\lambda}(a(T;H_{\phi}^{2,1})).
$
%\com{K: I think we only defined $a(T; E^{2,1}_{\phi})$ - should we define Fourier coefficients for an arbitrary Siegel modular form?\\T: we have this definition in section 2,p. 4}
Let $c = \min_{T \geq 0} c(T)$.  Since $H_{\phi}^{2,1} \in M_k(\Gamma_2)$, the Fourier coefficients $a(T;H_{\phi}^{2,1})$ have bounded denominators so $c$ is well-defined (\cite{ShimuraFC}). Moreover, our assumption that there is a $T_0> 0$ with $\val_{\lambda}(a(T_0;H_{\phi}^{2,1})) = \val_{\lambda}\left(L_{\alg}(2k-2, \Symm \phi)a(T_0,E^{2,1}_{\phi})\right) \leq 0$ gives that $c \leq 0$.
Set
\begin{equation*}
    G^{2,1}_{\phi}(z) = \lambda^{-c} H_{\phi}^{2,1}(z).
\end{equation*}
We have $a(T;G_{\phi}^{2,1}) \in \mO$ for all $T \geq 0$ since $c(T) - c \geq 0$ for all $T \geq 0$.  Observe that for $T$ with $\det T=0$, we have $a(T;G^{2,1}_{\phi}) = \lambda^{-c}L_{\alg}(2k-2,\Symm \phi) a(n;\phi)$ for some $n \in \bfZ_{\geq 0}$.  Since $a(n;\phi) \in \mO$ by assumption and $-c\geq 0$, this gives $\lambda \mid a(T;G^{2,1}_{\phi})$ for all $T$ with $\det T = 0$, i.e., all the Fourier coefficients indexed by singular $T$ vanish modulo $\lambda$. Moreover, since $c = c(\widetilde{T})$ for some $\widetilde{T}$, we have $a(\widetilde{T};G_{\phi}^{2,1}) \in \mO^{\times}$ for some $\widetilde{T}$.  Since $c\leq 0$ and $\lambda \mid a(T;G_{\phi}^{2,1})$ for all singular $T$, we have $\widetilde{T}>0$.
% \com{T: how does this follow from the assumption $\val_{\lambda}\left(L_{\alg}(2k-2, \Symm \phi)a(T_0,E^{2,1}_{\phi})\right) \leq 0$? Or is this a different $T_0$?\\J: Sorry, it was a different $T_0$.  I cleaned the entire proof up a bit. Hopefully it is clearer now. \\T: yes, thanks for the revision.}
 Thus, Corollary \ref{corl:yama} and the fact that $G^{2,1}_{\phi}$ and $E^{2,1}_{\phi}$ have the same eigenvalues gives an eigenform $f \in S_{k}(\Gamma_2;\mO)$ so that
$
    E^{2,1}_{\phi} \equiv_{\ev} f \not \equiv 0\pmod{\lambda}.
$
By Lemma \ref{lem:galirred} we get that $\rho_{f}$ is irreducible.
\end{proof}

\begin{example} \label{ex 1} Consider the space $M_{26}(\Gamma_2)$.  This space has dimension seven and is spanned by $E^{2,0}$ (Siegel Eisenstein series), $E^{2,1}_{\phi}$ (Klingen Eisenstein series), three Saito-Kurokawa lifts, and two non-lift forms $\Upsilon_1$ and $\Upsilon_2$ where here $\phi \in S_{26}(\Gamma_1)$ is the unique newform given by
\begin{equation*}
    \phi(z) = e(z) - 48e(2z) -195804 e(3z) + \cdots.
\end{equation*}
We have via \cite{DummiganSymmSquare} that
\begin{align*}
    L_{\alg}&(50, \Symm \phi)\\ &= \frac{2^{41} \cdot 163\cdot 187273}{3^{26} \cdot 5^{10}\cdot 7^7 \cdot 11^4 \cdot 13^2 \cdot 17^2 \cdot 19 \cdot 23^2 \cdot 29 \cdot 31 \cdot 37 \cdot 41 \cdot 43 \cdot 47 \cdot 657931}
\end{align*}
We consider $\ell \in \{163, 187273\}$ and show that both primes produce an example for Theorem \ref{cong 1}.% We provide a different argument for each  in producing an example.

The Klingen Eisenstein series associated to $\phi$ is given in the beta version of LMFDB.  By considering the Fourier coefficients indexed by $\bmat 1 & 0 \\ 0 & 0 \emat$ and $\bmat 2 & 0 \\ 0 & 0 \emat$ one can see that the Klingen Eisenstein series given there, say $E^{\LMFDB}_{\phi}$, is given by
\begin{equation*}
    E^{2,1}_{\phi}(z) = -\frac{E^{\LMFDB}_{\phi}(z)}{2^6 \cdot 3^3 \cdot 11 \cdot 19 \cdot 163 \cdot 187273}.
\end{equation*}
We have from LMFDB that
\begin{equation*}
    a\left(\bmat 1 & 1/2 \\ 1/2 & 1 \emat ; E^{2,1}_{\phi}\right) = \frac{2^2\cdot 5 \cdot 43}{ 11 \cdot 19 \cdot 163 \cdot 187273}
\end{equation*}
Consider $G^{2,1}_{\phi}(z) = L_{\alg}(50,\Symm \phi) E^{2,1}_{\phi}(z)$. We have for $\ell$ as above that $\ell \mid a(T;G^{2,1}_{\phi})$ for all $T$ with $\det T= 0$ and $a\left(\bmat 1 & 1/2 \\ 1/2 & 1 \emat ; G^{2,1}_{\phi}\right) \not \equiv 0 \pmod{\ell}$.
%\com{T: do we already need $\ov{\rho}_\phi$ irreducible for the following statement?\\J: I don't think so?  I think that only comes in to show that we can find $f$ with $\rho_f$ irreducible.}
Thus by Theorem \ref{cong 1} there exists a non-trivial Hecke eigenform $f \in S_{k}(\Gamma_2; \bfZ_{\ell})$ with $E^{2,1}_{\phi} \equiv_{\ev} f \pmod{\ell}$.

Consider first the prime $\ell = 163$ and suppose that $\ov{\rho}^{\rm ss}_{\phi,163}=\psi_1 \oplus \psi_2$ for some characters $\psi_1, \psi_2$.  Since  $\ov{\rho}_\phi$ is unramified for all $p \neq \ell$ we see that $\psi_1$ and $\psi_2$ are each an integer power of $\ov{\epsilon}$ (see the proof of Lemma \ref{Lemma 5.3}). As $163 \nmid a(163;\phi)$  we know $\phi$ is ordinary at $163$ and we get $\ov{\rho}^{\rm ss}_{\phi,163}=\ov{\epsilon}^{25} \oplus 1$. By \cite{RibetInvent76} Proposition 2.1 we can find a lattice such that
\begin{equation*}
     \ov{\rho}_{\phi,163} = \bmat 1 & * \\ 0 & \ov{\epsilon}^{25}\emat \not \cong 1 \oplus \ov{\epsilon}^{25}.
 \end{equation*}

One can use ordinarity of $\phi$ to show that  $*$ gives an unramified $163$-extension of $\bfQ(\zeta_{163})$ (see e.g. the proof of Theorem 4.28 in \cite{BergerKlosin23}). %\com{J: Are we using Ribet again here?  If we are using Ribet, isn't it more than just because it is unramified as he uses a separate argument for this prime?\\K: You're right, it's not obvious. We now added a reference above}
By Herbrand's Theorem this implies that $163 \mid B_{26}$.  However, one can check this is not true, so we must have that $\ov{\rho}_{\phi,163}$ is irreducible and so $E^{2,1}_{\phi}$ must be congruent (modulo 163) to a cusp form $f$ that is not a Saito-Kurokawa lift, i.e. $\rho_f$ is irreducible by Theorem \ref{cong 1}.  One uses LMFDB to check that $f= \Upsilon_2 $.

% As $163 \nmid a(163;\phi)$, we have $\phi$ is ordinary at $163$ so we can write
% \begin{equation*}
%     \ov{\rho}_{\phi,163} = \bmat \omega_{163}^{-25} & * \\ 0 & 1 \emat
% \end{equation*}
% with $* \neq 0$.

% One has that $*$ gives an unramified $163$-extension of $\bfQ(\zeta_{163})$.  By Herbrand's Theorem this implies that $163 \mid B_{26}$.  However, one can check this is not true, so we must have that $\ov{\rho}_{\phi,163}$ is irreducible and so $E^{2,1}_{\phi}$ must be congruent to a cuspform that is not a Saito-Kurokawa lift modulo $163$.  One uses LMFDB to check that $E^{2,1}_{\phi} \equiv_{\ev} \Upsilon_2 \pmod{163}$.

Now consider the case that $\ell = 187273$.  In this case it is less practical to calculate $a(187273; \phi)$, so we directly eliminate the possibility that  $E^{2,1}_{\phi}$ is congruent to a Saito-Kurokawa lift modulo $187273$.  The space to consider is $S_{50}(\Gamma_1)$.  This space has one Galois conjugacy class of newforms consisting of three newforms, call them $\psi_1, \psi_2$, and $\psi_3$.  Each newform has a field of definition $K_{\psi_{i}}$ generated by a root $\alpha_{i}$ of
\begin{equation*}
    c(x) = x^3 + 24225168x^2 - 566746931810304x -13634883228742736412672.
\end{equation*}
One has that $\lambda(2, E^{2,1}_{\phi}) = -805306416$ and that $\lambda(2,\psi_{i}) = 2^{49}+2^{48} + \alpha_{i}$. One uses SAGE to check that $\lambda(2,E^{2,1}_{\phi}) \not \equiv \lambda(2,\psi_{i}) \pmod{187273}$, so $E^{2,1}_{\phi}$ must be congruent to a cusp form that is not a Saito-Kurokawa lift.  One uses LMFDB to see that $E^{2,1}_{\phi}\equiv_{\ev} \Upsilon_1 \pmod{187273}$.
\end{example}

%\com{J: technically, I could omit the arguments showing the Galois reps are irreducible in the previous example just by showing the forms are not congruent to the SK-lifts directly..  I think it is fine how it is though. }

\section{Extensions of  Fontaine-Laffaille modules} \label{Fontaine-Laffaille}

In this section we gather  various facts (in particular Proposition \ref{Homtensor} and Proposition \ref{lem4.18}) about extensions of Fontaine-Laffaille modules, which we use in this article but which to the best of our knowledge have not been published elsewhere.

\subsection{Definitions} \label{s4.1}
We keep our assumption that $\ell$ is an odd prime. %\com{K: I suggest to fix $a,b$ so we don't have to repeat what inequality they satisfy. For example:\\J: This is fine with me.}
We fix integers $a,b$ such that $0 \leq b-a \leq \ell -2$.  In this section let $E$ be an arbitrary finite extension of $\bfQ_{\ell}$ with ring of integers $\Oo$, uniformizer $\lambda$ and residue field $\bfF$. Write ${\rm LCA}_{\Oo}$ (respectively ${\rm LCN}_{\Oo}$) for the category of local complete Artinian (respectively Noetherian) $\Oo$-algebras with residue field $\bfF$.
For a category $\mC$ we will write $X \in \mC$ to mean that $X$ is an object of $\mC$.
\begin{definition}[\cite{Kalloniatis19} Definition 2.3/\cite{Booher19} Definition 4.1] \label{defn2.1}
{~}
\begin{enumerate}
%\com{T: added (1), could use this to shorten (2), i.e. to define $MF^{f, [a,b]}_{?,\bfZ_\ell}$ as full subcategories of $MF^{f}_{?,\bfZ_\ell}$ with $k=a$ and $l=b$.}
\item A Fontaine-Laffaille module is a finitely generated $\bfZ_\ell$-module $M$ together with a decreasing filtration by $\bfZ_\ell$-module direct summands $M^i$ for $i \in \bfZ$ such that  there exists $k \leq l$ with $M^i=M$ for $i \leq k$ and $M^{i+1}=0$ for $i \geq l$, and a collection of $\bfZ_\ell$-linear maps $\phi^i_M: M^i \to M$ such that $\phi^i_M|_{M^{i+1}}=\ell \phi^{i+1}_M$ for all $i$ and $M=\sum_i \phi^i_M(M^i)$.  The  category of all Fontaine-Laffaille modules %\com{K: I think this is unnecessary: `(as we vary the weights)'}
is denoted $MF^{f}_{\bfZ_\ell}$.  Morphisms in this category are $\bfZ_\ell$-linear maps $f: M \to N$ satisfying $f(M^i) \subset N^i$ and $f \circ \phi^i_M=\phi^i_N \circ f|_{M^i}$
%\com{K: Should the last map be $f|_{M^i}$? \\T: yes made the change}
for all $i$. We will write $MF^{f}_{\tor, \bfZ_\ell}$ for  the full subcategory whose objects are of finite length as a $\bfZ_\ell$-modules.

    \item
 For a fixed interval $[k,l]$ we denote the full subcategory of $MF^{f}_{?,\bfZ_\ell}$ whose objects $M$ have a filtration  satisfying $M^k=M$ and $M^{l+1}=0$ by $MF^{f, [k,l]}_{?,\bfZ_\ell}$ for $?\in \{\emptyset, \textup{tor}\}$.
  %  A Fontaine-Laffaille module with weights in $[a,b]$ is a Fontaine-Laffaille module for which $k=a$

   % finitely generated $\bfZ_\ell$-module $M$ together with a decreasing filtration by $\bfZ_\ell$-module direct summands $M^i$ for $i \in \bfZ$ such that  $M^a=M$ and $M^{b+1}=0$, and a collection of $\bfZ_\ell$-linear maps $\phi^i_M: M^i \to M$ such that $\phi^i_M|_{M^{i+1}}=\ell \phi^{i+1}_M$ for all $i$ and $M=\sum_i \phi^i_M(M^i)$. The corresponding category is denoted $MF^{f, [a,b]}_{\bfZ_\ell}$. Morphisms in this category are $\bfZ_\ell$-linear maps $f: M \to N$ satisfying $f(M^i) \subset N^i$ and $f \circ \phi^i_M=\phi^i_N \circ f$ for all $i$. We will write $MF^{f, [a,b]}_{tor, \bfZ_\ell}$ for  the full subcategory whose objects are of finite length as a $\bfZ_\ell$-modules.

    \item For any $A \in {\rm LCA}_\Oo$, a Fontaine-Laffaille module over $A$ consists of an object $M \in MF^{f, [a,b]}_{\tor, \bfZ_\ell}$ together with a map $\theta: A \to {\rm End}_{MF^{f, [a,b]}_{\tor, \bfZ_\ell}}(M)$ that makes $M$ into a free finitely generated module over $A$
 in such a way that
 %\com{K: replaced: `the filtered pieces of $M$ are' by}
 $M^i$ is an $A$-direct summand of $M$ for each $i$.
 %\com{K: end changes}
 A morphism between two such objects is required to additionally preserve the $A$-structure. We will denote this category of Fontaine-Laffaille modules over $A$ as $MF^{f, [a,b]}_{\tor, \bfZ_\ell} \otimes_{\bfZ_\ell} A$.
\item For $M \in MF^{f, [a,b]}_{\tor, \bfZ_\ell} \otimes_{\bfZ_\ell} A$ any integer $i$ for which $M^i/M^{i+1} \neq 0$ is called a Fontaine-Laffaille weight for $M$.
  The set of Fontaine-Laffaille weights for $M$ will be denoted by ${\rm FL}(M)$.

 \end{enumerate}
 \end{definition}

\begin{rem}
    %\com{K: deleted: Note that}
    We impose the stronger restriction on the length of the filtration as in \cite{BlochKato} Section 4 and \cite{ClozelHarrisTaylor08} Section 2.4.1, compared to that in  Section 1.1.2
of \cite{DiamondFlachGuoAnnScEcole04} or \cite{Kalloniatis19} Definition 2.3 (which allow the length to be $\ell-1$).
\end{rem}

%\com{I Made the following into a separate remark, because it is tangential to the main body of the section}

\begin{definition} \label{def4.3} We introduce the following examples of Fontaine-Laffaille modules:
\begin{enumerate}[(i)]
    \item If $0 \in [a,b]$ we write $\textbf{1} \in MF^{f, [a,b]}_{\bfZ_\ell}$ for the Fontaine-Laffaille module defined by $\textbf{1}^i=\bfZ_\ell$ for $i \leq 0$ and $\textbf{1}^i=0$ for $i >0$. We set $\phi^i: \textbf{1}^i \to \textbf{1}$ to be given by $x \mapsto \ell^{-i} x$ for $i \leq 0$.
%    \item For $n \in [a,b] \cap \bfZ$ define $M_n\in MF^{f, [a,b]}_{\tor, \bfZ_\ell}$  to be the 1-dimensional $\bfF_\ell$-vector space  with filtration $M_n^i=M_n=\bfF_\ell$ for $i \leq n$, $M_n^{n+1}=0$ and $\phi^i: M_n^i \to M_n$ the $0$-map for all $i \neq n$, and the identity map for $i=n$.

\item For any $A \in {\rm LCA}_\Oo$ we define $M_{n, A} \in MF^{f, [a,b]}_{\tor, \bfZ_\ell} \otimes_{\bfZ_\ell} A$ to be the free rank one $A$-module equipped with the  filtration $M_{n, A}^i=A$ for $i \leq n$, $M_{n, A}^{n+1}=0$ and $\phi^i: M_{n, A}^i \to M_{n, A}$ given by $x \mapsto \ell^{n-i} x$ for $i \leq n$. We put $\mathbf{1}_A=M_{0, A}$.
\end{enumerate}
\end{definition}

\begin{definition}[\cite{Booher19} Definition 4.9] \label{defn2.5}
    For $M \in  MF^{f, [a,b]}_{\tor, \bfZ_\ell}$ and $s \in \bfZ$ define $M(s)$ to be the same underlying $\bfZ_{\ell}$-module, but change the filtration to $M(s)^i=M^{i-s}$ for any $i \in \bfZ$. This means that $M(s) \in MF^{f, [a+s,b+s]}_{\tor, \bfZ_\ell}$.
\end{definition}

\subsection{Extensions}

To ease notation  in the rest of this section we put $\mathcal{C}_A^I=MF_{\tor, \bfZ_\ell}^{f, I} \otimes_{\bfZ_\ell} A$ for $A \in {\rm LCA}_\Oo$. Here $I=[a,b]$.

\begin{definition}[Definition/Lemma] \label{def4.6}
Given $M, N \in \mC_A^I$  define a filtration on the $A$-module $\Hom_{A}(M,N)$ by $$\Hom_{A}(M,N)^i=\{f\in\Hom_{A}(M,N)\mid f(M^j)\subset N^{j+i} \text{ for all } j \in \bfZ\}$$ and   $\bfZ_\ell$-linear maps $\phi^i: \Hom_{A}(M,N)^i \to \Hom_{A}(M,N)$  by $$\phi^i(f)(\phi^j_M(m))=\phi^{i+j}_N(f(m))$$ (note that  $M=\sum \phi_M^j(M^j)$) for $ f \in \Hom_{A}(M,N)^i$ and all $m \in M^j$ and $j \in \bfZ$. We claim this defines a Fontaine-Laffaille structure and that $\Hom_{A}(M,N)\in MF_{\tor, \bfZ_\ell}^{f, [a-b, b-a]} \otimes_{\bfZ_\ell} A.$
\end{definition}

\begin{proof}
   First note that there exists a canonical $A$-module homomorphism $\psi: M^\vee \otimes_A N \to \Hom_{A}(M,N)$, where $M^\vee=\Hom_A(M,A)$.
    Definition 4.19  in \cite{Booher19}  defines a Fontaine-Laffaille structure on $M^\vee$ (and Lemma 4.20 and 4.21 prove that this structure  is well-defined and so we get an object in $MF_{\tor, \bfZ_\ell}^{f, [-b, -a]} \otimes_{\bfZ_\ell} A$). Definition 4.17  in \cite{Booher19}  then gives us the Fontaine-Laffaille structure on $M^\vee \otimes_A N$.

    We claim that transferring  this structure on $M^\vee \otimes_A N$ via  $\psi$ to $\Hom_{A}(M,N)$ matches our definition.
    Recall from \cite{Booher19} that $(M^\vee)^i=\{f\in \Hom_{A}(M,A)| f(M^k) \subset \mathbf{1}_A^{i+k} \text{ for all } k \in \bfZ\}$ and $(M^\vee \otimes N)^n=\sum_{i+j=n} (M^\vee)^i \otimes_A N^j.$
    We will first show that $\psi((M^\vee \otimes N)^n) \subset \Hom_{A}(M,N)^n$.
    Let $f_i \otimes n_j \in (M^\vee)^i \otimes_A N^j$. Then $\psi(f_i \otimes n_j): m \in M^k \mapsto f_i(m)n_j \in N^j$. In fact, the image lies in $N^{n+k}$. This is clear for $j \geq n+k$. If $j<n+k$ (and hence $0<i+k$) it follows since $f_i(m) \in \mathbf{1}_A^{i+k}=0$.
%    \com{T: this shows $\psi((M^\vee \otimes N)^n) \subset \Hom_{A}(M,N)^n$. Still need to do the reverse inclusion.\\T: the problem is that I don't know how to express $\psi^{-1}$, but does the following work?\\J: I think it looks good. One thing to note though is that for a general module, the module is not necessarily isomorphic to its double dual.  It should be fine in our case because ours are finitely generated and free over $A$, but we should probably state that when it is used. \\T: We don't think this uses double duality. Are we missing something?\\J: I can't find my notes for this and I don't see anymore where I thought it was being used, so my comment is withdrawn. :-)}
To show the reverse inclusion $\psi((M^\vee \otimes N)^n) \supset \Hom_{A}(M,N)^n$ consider  $f \in \Hom_{A}(M,N)^n$ and let $j$ be maximal among integers $l$ such that $f(M) \subset N^l$. To satisfy $f(M^k) \subset N^{k+n}$ for all integers $k$ we need $f(M^k)=0$ for $k+n>j$ by maximality of $j$. This means  that we need $f$ to factor through $M/M^{1-i}$ for  $i:=n-j$.
%Since $M^{1-i}$ is an $A$-module direct summand of $M$ we have $$\{f:M \to N^j|f(M^{1-i})=0 \}=\{f:M/M^{1-i} \to N^j\}.$$
By \cite{Booher19} Lemma 4.20 we have $(M^\vee)^i=\Hom_A(M/M^{1-i}, A)$ so we get $$(M^\vee)^i \otimes N^j=\Hom_A(M/M^{1-i}, A) \otimes N^j\overset{\psi}{\cong}\Hom_A(M/M^{1-i}, N^j).$$ We conclude that $f \in \psi^{-1}((M^\vee)^i \otimes N^j) \subset \psi^{-1}((M^\vee \otimes N)^n)$.

    Now we check the $\bfZ_\ell$-linear maps: Recall from \cite{Booher19} that for $f \in M^\vee$ we have $\phi^i_{M^\vee}(f)(\phi^j_M(m))=\phi^{i+j}(f(m))$ for all $m \in M^j$ and $j \in \bfZ$. We also have $\phi^n_{M^\vee \otimes_A N}=\sum_{i+j=n} \phi^i_{M^\vee} \otimes \phi^j_N.$ We claim that $\phi^n_{\Hom_A(M,N)} \circ \psi=\psi \circ \phi^n_{M^\vee \otimes_A N}: (M^\vee \otimes N)^n \to \Hom_A(M,N).$ For this one calculates that both sides map $f \otimes n \in (M^\vee)^i \otimes N^{n-i}$ to the homomorphism, for which  $$\phi^k_M(m) \mapsto \begin{cases}
        0& \text{ if } i+k \geq 0\\
        \phi^{n+k}_N(f(m)x)& \text{ if } i+k \leq 0
    \end{cases}$$ for any $m \in M^k$ (for $\psi \circ \phi^n_{M^\vee \otimes_A N}$ this uses $\phi^{n+k}_N|_{N^{n-i}}=\ell^{k+i} \phi^{n-i}_N$ for $i+k \leq 0$). This claim, combined with the results in \cite{Booher19} shows that the definition of $\phi^n_{\Hom_A(M,N)}$ is well-defined and satisfies the requirements for $\Hom_A(M,N)$ to be a Fontaine-Laffaille module in $ MF_{\tor, \bfZ_\ell}^{f, [a-b, b-a]} \otimes_{\bfZ_\ell} A$.
\end{proof}

% \com{J: In Booher he makes a comment about ``if the $\phi^{i}$ exist...''   Our definition looks like we are stating we have these maps and they always work fine.  Should we include a comment along the lines of if they exist? \\T: Kris and I worked on this today and proved the well-definedness of the maps and that this satisfies the conditions for being a FoLa module for the case of $N=M(n)$ rank 2- which we need later for ${\rm ad} \rho_f$. We'll see if we can generalize this or just include this extension of Booher's result.\\T: Further update: we think we can do this in full generality by combining what Booher did for the dual and the tensor product. Will add these proofs soon.}

For $M, N\in \mC_A^I$
%We can then endow  $\Hom_{A}(M,N)$ with the filtration given by $$\Hom_{A}(M,N)^i=\{f\in\Hom_{A}(M,N)\mid f(M^j)\subset N^{j+i}\}.$$
consider the map $\phi-1:  \Hom_{A}(M,N)^0 \to  \Hom_{A}(M,N)$ which takes  $f$ to the homomorphism that sends $m=\sum_j \phi^j_{M}(m_j)$ to $$\sum_j\phi^j_{N}(f(m_j))-f(m)=\sum_j \left(\phi^j_{N}(f(m_j))-f(\phi^j_{M}(m_{j}))\right).$$ Note that $\ker (\phi-1)=\Hom_{\mC_A^I}(M,N).$
% \com{J: I think we're missing a bunch of subscripts here.  They aren't so hard to tell from context, but since we just used $\phi^{j}$ for an entirely different thing right above (unless I'm mistaken), we should include them I think.  My reading is it should say: ``For $M, N\in \mC_A^I$.
% %We can then endow  $\Hom_{A}(M,N)$ with the filtration given by $$\Hom_{A}(M,N)^i=\{f\in\Hom_{A}(M,N)\mid f(M^j)\subset N^{j+i}\}.$$
% consider the map $\phi-1:  \Hom_{A}(M,N)^0 \to  \Hom_{A}(M,N)$ which takes  $f$ to the homomorphism that sends $m=\sum_j \phi^j_{M}(m_j)$ to $$\sum_j\phi^j_{N}(f(m_j))-f(m)=\sum_j \left(\phi^j_{N}(f(m_j))-f(\phi^j_{M}(m_{j}))\right).$$ Note that $\ker (\phi-1)=\Hom_{\mC_A^I}(M,N).$''\\T: you're correct and I made the change you suggested.}

\begin{prop}[\cite{ClozelHarrisTaylor08} Lemma 2.4.2, \cite{Kalloniatis19} Proposition 2.17] \label{prop2.17}
    Given $M,N \in \mC_A^I$ we have an exact sequence of $A$-modules (note that  $\Hom_{\Fil, A}(M,N)$ in \cite{Kalloniatis19} equals $\Hom_{A}(M,N)^0$)
    $$0 \to \Hom_{\mC_A^I}(M,N)  \to \Hom_{A}(M,N)^0 \overset{\phi-1}{\to} \Hom_{A}(M,N) \to \Ext^1_{\mC_A^I}(M,N) \to 0.$$
\end{prop}

Given $M,N \in \mC_A^I$ we write ${\rm FL}(M)>{\rm FL}(N)$ if there is an integer $j$ such that all elements of ${\rm FL}(M)$ are greater than or equal to $j$, and all elements of ${\rm FL}(N)$ are strictly less than $j$.
%\com{K: moved this sentence `From Proposition \ref{prop2.17} one can calculate the following:' into the proof below}

\begin{prop} \label{Kp238}   The extension group $\Ext^1_{\mC_A^I}(M,N)$ is a  finitely generated $A$-module. Furthermore one has %\com{K: Do we know that the extension group is finitely generated, or does it only follow after we know (i) and (ii) below?\\T: I've added this now to the statement of the Proposition. \\T: and now added something to the proof.\\T: On 21 August we didn't understand Kallionatis proof of freeness. We only need freeness though for $A=\bfF$, for which it's fine.\\T: I have now changed the statement to remove the claim that $\Ext^1_{\mC_A^I}(M,N)$ is free in general, but capturing the two extreme cases we want and added more detail to the proof.\\K: The new statement and proof look good to me. Thank you for fixing it.}

    \begin{enumerate}[(i)]
        \item If ${\rm FL}(M)>{\rm FL}(N)$ then $\Ext^1_{\mC_A^I}(M,N) \cong \Hom_A(M,N)$, in particular it is a free $A$-module and   ${\rm rk}_A(\Ext^1_{\mC_A^I}(M,N))={\rm rk}_A(M) {\rm rk}_A(N)$.
        \item If ${\rm FL}(M)<{\rm FL}(N)$ then $\Ext^1_{\mC_A^I}(M,N)=0$.
    \end{enumerate}
\end{prop}
\begin{proof}
    This follows from Proposition \ref{prop2.17}. In particular, $\Ext^1_{\mC_A^I}(M,N)$ is a quotient of the finitely generated $A$-module $\Hom_A(M,N)$.  The calculation on \cite{Kalloniatis19} p. 238 (``two notable cases")  is carried out for $MF^{f, [0,\ell-1]}_{\tor, \bfZ_\ell} \otimes_{\bfZ_\ell} A$, but applies verbatim to $\mC_A^I$. %\com{T: removed: "Note that \cite{Kalloniatis19} proves the freeness of the extension groups on page 238."}
    If ${\rm FL}(M)>{\rm FL}(N)$ then this calculation shows that $\Hom_A(M,N)^0=0$, while if ${\rm FL}(M)<{\rm FL}(N)$ then one gets $\Hom_A(M,N)^0=\Hom_A(M,N)$.
\end{proof}

\begin{prop}[Hom-tensor adjunction] \label{Homtensor} Let $M, N \in \mC_A^I$. Assume that $\Hom_{A}(M,N)$ equipped with the filtration as in Definition \ref{def4.6} is an object in $\mC_A^I$ and that $0 \in I$.
Then there exists a canonical isomorphism of $A$-modules:
$$\Ext^1_{\mC_A^I}(M,N)\cong \Ext^1_{\mC_A^I}(\mathbf{1}_A, \Hom_{A}(M,N)).$$
\end{prop}

\begin{proof}
The statement follows from the existence of the following commutative diagram with exact columns:
\begin{equation} \label{eqn4.2} \xymatrix{0\ar[d]&0\ar[d]\\
\Hom_{\mC_A^I}(M,N) \ar[d]&\Hom_{\mC_A^I}(\mathbf{1}_A,\Hom_{A}(M.N))\ar[d]\\
 \Hom_{A}(M,N)^0 \ar[r]^-{\psi'}\ar[d]_{\phi-1} &\Hom_{A}(\mathbf{1}_A, \Hom_{A}(M,N))^0\ar[d]^{\phi-1}\\
 \Hom_{A}(M,N) \ar[r]^-{\psi}\ar[d]_{\alpha} &  \Hom_{A}(A, \Hom_{A}(M,N))\ar[d]\\
 \Ext^1_{\mC_A^I}(M,N)\ar[r]^-{\tilde \psi} \ar[d]& \Ext^1_{\mC_A^I}(\mathbf{1}_A,\Hom_{A}(M,N))\ar[d]\\
0&0}\end{equation}

The exactness of both columns follows from Proposition \ref{prop2.17}.  %Analogous definitions to the ones before the statement of the proposition apply to the second column.
The second horizontal arrow is the usual isomorphism $\psi$ of $A$-modules given by $f\mapsto (a\mapsto af)$ (recall that the underlying module of the object $\mathbf{1}_A$  is $A$) with the inverse map sending $g$ to $g(1)$, where $1$ is the multiplicative identity of $A$.
The map $\tilde \psi$ is defined by lifting an element of $ \Ext^1_{\mC_A^I}(M,N)$ to $\Hom_{A}(M,N)$ and using $\psi$. The exactness of the first column ensures that such a map is well-defined.

The first horizontal arrow is the restriction $\psi'$ of $\psi$ to $\Hom_{A}(M,N)^0$ (note that $\Hom_{A}(M,N)^0$ is a subgroup of $\Hom_{A}(M,N)$ even though $\phi-1$ is not necessarily injective). We need to check that $\psi'$ lands in $\Hom_{ A}(\mathbf{1}_A, \Hom_{A}(M,N))^0$.  By its definition we need to check if $f(\mathbf{1}_A^j) \subset \Hom_{A}(M,N)^j$.
%\com{could change to $j$ for consistency to match the choice of free variable in the definition of the MF structure on Hom\\K: I've done this now.}
If $j>0$ there is nothing to check as then $\mathbf{1}_A^j=0$, so assume that $j\leq 0$. Then $\mathbf{1}_A^j=A$ and $\Hom_{A}(M,N)^j\supset \Hom_{A}(M,N)^0$. So, it is enough to show that if $f\in \Hom_{A}(M,N)^0$ %(i.e., if $f(M^j)\subset N^j$ for every $j$),
%\com{K: Do you suggest to change from $j$ to something else in this parenthetical remark?\\T: maybe, or we could remove the parenthetical remark as it doesn't seem to be used?}
then $\psi'(f)(A)\subset  \Hom_{A}(M,N)^0$. Let $a\in A$. Then $\psi'(f)(a)=af$ which clearly lies in $\Hom_{A}(M,N)^0$ as $\Hom_{A}(M,N)^0$ is an $A$-module.

Now let $g\in \Hom_{ A}(\mathbf{1}_A, \Hom_{A}(M,N))^0$.  We need  to show that $\psi^{-1}(g)$ lands in $\Hom_{A}(M,N)^0$. Again we need to consider $\psi^{-1}(g)(\mathbf{1}_A^j)$. If $j>0$, then $g=0$, hence we are done. Assume that $j\leq 0$. Then $\mathbf{1}_A^j=A$ and  $\psi^{-1}(g)=g(1)$. As $1\in \mathbf{1}_A^0$ and $g \in \Hom_{ A}(\mathbf{1}_A, \Hom_{A}(M,N))^0$ we must have that $g(1)\in  \Hom_{A}(M,N)^0$. So, we are done again.

This shows that $\psi'$ is a bijection, hence an isomorphism.
Hence by the second Four Lemma $\tilde \psi$ is injective, and since it is clearly surjective, it is an isomorphism.

\end{proof}

\subsection{Fontaine-Laffaille Galois representations}

Fix %\com{T: maybe change to ``(Torsion) crystalline Galois representations? FoLa Galois representations we currently only define as Defn 6.2}
an interval $I=[a,b]$ with $a, b \in \bfZ$ and $b-a \leq \ell-2$.
In this subsection we introduce certain categories of $G_{\bfQ_\ell}$-representations and define
  a covariant version $V_I$
of the functor in \cite{FoLa} from the categories of Fontaine-Laffaille modules defined in section \ref{s4.1} to these categories of Galois representations.

Let $A_{\cris}$ and $B_{\cris}$ denote the usual  Fontaine's $\ell$-adic period rings (see Definition 7.3 and 7.7 in \cite{Fontainebook} and \cite{FontaineAnnMath82}).  %\com{J: Let's insert the reference Tobias provides below here so people can easily look up these definitions.. unless you both think it is too basic to bother referencing.}
We recall that a $\bfQ_\ell[G_{\bfQ_\ell}]$-module $V$ is called crystalline if $\dim_{\bfQ_{\ell}}V = \dim_{\bfQ_{\ell}} \coh^0(\bfQ_{\ell}, V \otimes_{\bfQ_{\ell}} B_{\cris})$.  Our convention is that the Hodge-Tate weight of the cyclotomic character is $+1$.
%\com{K: How about replacing the  definition below (currently 4.8) with the following two\\T: looks good to me, I've removed the old definition.\\J: Okay to remove this comment.}
\begin{definition} \label{defn4.5} Let %$I$ be an interval \com{J: $I$ is fixed above already.} and
$A\in {\rm LCA}_\Oo$. We introduce the following categories:
\begin{enumerate}[(i)]
    \item  ${\rm Rep}_{\bfZ_\ell}^{f}(G_{\bfQ_{\ell}})$,  the category of $\bfZ_\ell[G_{\bfQ_{\ell}}]$-modules that are finitely generated as $\bfZ_\ell$-modules.
%\com{T: should they be free?\\ No, I think everybody seems to define it this way, except that most talk of ``finite type".}
    \item  ${\rm Rep}_{\tor, \bfZ_\ell}^{f}(G_{\bfQ_{\ell}})$,  the full subcategory of ${\rm Rep}_{\bfZ_\ell}^{f}(G_{\bfQ_{\ell}})$ whose objects are required to be of finite length as $\bfZ_\ell[G_{\bfQ_{\ell}}]$-modules.
%    \com{T: split    "${\rm Rep}_{\bfZ_\ell}^{{\rm cris}, I}(G_{\bfQ_\ell})$ (respectively ${\rm Rep}_{tors, \bfZ_\ell}^{{\rm cris}, I}(G_{\bfQ_\ell})$), the full subcategory of ${\rm Rep}_{\bfZ_\ell}^{f}(G_{\bfQ_{\ell}})$ (respectively ${\rm Rep}_{tor, \bfZ_\ell}^{f}(G_{\bfQ_{\ell}})$) whose objects are isomorphic to $T/T'$, where $T$ and $T'$ are $G_{\bfQ_{\ell}}$-stable lattices in a crystalline $\bfQ_\ell$-representation with Hodge-Tate weights in $I$. We refer to the objects in ${\rm Rep}_{\bfZ_\ell}^{{\rm cris}, I}(G_{\bfQ_\ell})$ as \emph{torsion crystalline representations}." into two definitions}
\item    ${\rm Rep}_{\bfZ_\ell}^{{\rm cris}, I}(G_{\bfQ_\ell})$, the full subcategory of ${\rm Rep}_{\bfZ_\ell}^{f}(G_{\bfQ_{\ell}})$ whose objects are isomorphic to $T/T'$, where $T$ and $T'$ are $G_{\bfQ_{\ell}}$-stable finitely generated submodules of a crystalline $\bfQ_\ell$-representation with Hodge-Tate weights in $I$.
    \item   ${\rm Rep}_{\tor, \bfZ_\ell}^{{\rm cris}, I}(G_{\bfQ_\ell})$,  the full subcategory of ${\rm Rep}_{\tor, \bfZ_\ell}^{f}(G_{\bfQ_{\ell}})$ whose objects are isomorphic to $T/T'$, where $T$ and $T'$ are $G_{\bfQ_{\ell}}$-stable lattices in a crystalline $\bfQ_\ell$-representation with Hodge-Tate weights in $I$. %We refer to the objects in ${\rm Rep}_{\bfZ_\ell}^{{\rm cris}, I}(G_{\bfQ_\ell})$ as \emph{torsion crystalline representations (with weights in $I$)}. %\com{K: Shouldn't it be ${\rm Rep}_{tors, \bfZ_\ell}^{{\rm cris}, I}(G_{\bfQ_\ell})$? But Theorem 4.11 (iii) below suggests not. Why torsion then?\\T: This is Booher's terminology (start of section 4.1. I think I've seen it elsewhere as well. But I agree it's misleading. Maybe we could also call them `Fontaine-Laffaille"?}
    \item  ${\rm Rep}_{{\rm free}, A}^{{\rm cris}, I}(G_{\bfQ_\ell})$, the category of free finite rank $A$-modules $M$ with an $A$-linear $G_{\bfQ_\ell}$-action, for which there exists a crystalline representation of $G_{\bfQ_\ell}$ defined over $E$ with Hodge-Tate weights in $I$ containing $G_{\bfQ_\ell}$-stable $\Oo$-lattices $T' \subset T$, and an $\Oo$-algebra map $A \to {\rm End}_\Oo(T/T')$ such that $M$ is isomorphic as an $A[G_{\bfQ_\ell}]$-module to $T/T'$. We will call objects of this category \emph{Fontaine-Laffaille $A$-representations (with weights in $I$).} \end{enumerate}\end{definition}
    \begin{rem}
        Definition \ref{defn4.5}(v) matches  Definition 2.1 in \cite{Kalloniatis19} .
    \end{rem}

\begin{definition} [\cite{BlochKato} p. 363, \cite{Booher19} Definition 4.7+4.9]
Similar to  \cite{Booher19} we define the following two functors.
\begin{enumerate}[(i)]
    \item A covariant functor $T_{\rm cris}: MF^{f, [2-\ell,0]}_{\bfZ_\ell} \to {\rm Rep}_{\bfZ_\ell}^{f}(G_{\bfQ_\ell})$ defined via
$$T_{\rm cris}(M):=\ker\left(1-\phi^0_{A_{\rm cris} \otimes_{\bfZ_\ell} M}: {\rm Fil}^0(A_{\rm cris} \otimes_{\bfZ_\ell} M) \to A_{\rm cris} \otimes_{\bfZ_\ell} M \right).$$
%\com{J: Same ring of $\ell$-adic periods as above?  Same reference is given but different notation from above. \\T: Fontaine defined several period rings. We now introduce them together a bit earlier. Not sure if [Fo82] is the best reference, \url{https://www.imo.universite-paris-saclay.fr/~fontaine/galoisrep.pdf} Defnition 6.1.and 6.7 defines $A_{\rm cris}$ and $B_{\rm cris}$.\\J: If everyone is okay with, let's either switch to this reference or at least include it as another reference. [Fo82] is a long paper so having to rummage through it to find the correct defs isn't ideal to me.}
%\item \com{K: If we fix $a$ and $b$ at the beginning as I suggest, then we can delete `For any $I=[a,b]$ with $b-a \leq \ell-2$' and so make part (ii) look like (i)}
\item A covariant functor   $V_I: MF^{f, [a,b]}_{\bfZ_\ell} \to {\rm Rep}_{\bfZ_\ell}^{f}(G_{\bfQ_\ell}),$
 defined via \be \label{doubletwist}V_I(M)=T_{\rm cris}(M(-b))(- b).\ee Recall that $M(-b)$ was defined in Definition \ref{defn2.5}, while $(-b)$ on the outside denotes the Tate twist as defined in Section \ref{sec:intro}.
\end{enumerate}
%\com{K: Somewhere, maybe here we should make an observation that for $?\in \{\emptyset, tor\}$ the category $MF_{?, \bfZ_{\ell}}^{f,[a,b]}$ is a full (important because we use it with maps between objects in these categories, in particular extensions!) subcategory of $MF_{?, \bfZ_{\ell}}^{f,[a,a+\ell-2]}$, otherwise $T_{\rm cris}$ would not apply in the above. Or do we want to simply assume throughout that $b-a=\ell-2$? For now I added the following remark.\\T: I agree with the remark. The point is that $MF_{?, \bfZ_{\ell}}^{f,[a,b]}$ are full subcategories of a category $MF_{?, \bfZ_{\ell}}^{f}$ of filtered Fontaine modules with no specified filtration length. Should we also introduce this or is this remark sufficient?\\T: I added this definition to Defn 4.1.\\J: Okay to remove this comment?}
\begin{rem}\label{full}

    We note that for $?\in \{\emptyset, \tor\}$ the category $MF_{?, \bfZ_{\ell}}^{f,[a,b]}$ is a full  subcategory of $MF_{?, \bfZ_{\ell}}^{f,[a,a+\ell-2]}$, since they are both full subcategories of $MF_{?, \bfZ_{\ell}}^{f}$ (cf. Definition \ref{defn2.1}), so in particular \eqref{doubletwist} makes sense.
\end{rem}

\end{definition}

\begin{rem}

    Note that $V_I$ extends $T_{\rm cris}$ to general $I$ (in particular $V_{[2-\ell,0]}=T_{\rm cris}$).
    Also observe that for  $M \in  MF^{f, [a,b]}_{\tor, \bfZ_\ell}$ we have $M(-b) \in MF^{f, [2-\ell, 0]}_{\tor, \bfZ_\ell}$ since  $M(-b)^{1}=M^{b+1}=0$ and $M(-b)^{2-\ell}=M^{2-\ell+b}=M$ as $b+2-\ell \leq a$. In particular, the definition of $V_I$ makes sense.

Compared to \cite{Booher19} we work with the more restrictive interval $[2-\ell,0]$ for $T_{\rm cris}$ and correct a sign error in the Galois twist in  \cite{Booher19} Definition 4.9.

\end{rem}

%\com{J:  In Definition 4.1, it is only stated that $M^{a} = M$ and it is a decreasing filtration. Why would we know $M^{2-\ell+b} = M$ just because $b + 2 - \ell \leq a$? It seems in our definition nothing is said in this situation if $b +2 - \ell < a$.  \\T: I added the phrase `` for $i \in \bfZ$ " to Defnition 4.1 to make clearer that there are $M^i$ for all integers $i$. Since this filtration is decreasing, specifying $M^a=M$ means that $M^i=M$ for all $i \leq a$.\\J: Thanks!}
%\com{T: corrected  Booher's twist to $-b$. \\J:  I'm still trying to learn this stuff.  Is it obvious that the $b$ is a typo? Is there an easy explanation I would possibly understand?\\T: I'll forward you  the email exchange with Jeremy I had this week where he confirms his mistake.\\J: Perfect, thanks!}

%\com{K: Should we say what the morphisms are in that category? Also is it the category of such $\rho$s or rather the category of $A[G_{\bfQ_{\ell}}]$-modules isomorphic to $T/T'$ for some $T,T'$?}

\begin{theorem}[\cite{BlochKato} Theorem 4.3, \cite{Niziol93} Section 2, \cite{DiamondFlachGuoAnnScEcole04} Section 1.1.2, \cite{Hattori19} Section 2.2, \cite{Booher19} Fact 4.10, \cite{Kalloniatis19} Theorem 2.10] \label{ThmFoLa}
We have:
\begin{enumerate}[$(i)$]
    \item The covariant functor $V_{[a,b]}:MF^{f, [a,b]}_{\bfZ_\ell} \to {\rm Rep}^f_{\bfZ_\ell}(G_{\bfQ_\ell})$ is well-defined, exact and fully faithful.
    \item For $M \in MF^{f, [a,b]}_{\bfZ_\ell}$ one has $V_{[a,b]}(M)=\invlim_n V_{[a,b]}(M/\ell^n)$.
    \item The essential image of $V_{[a,b]}$
    %\com{K: This `on $MF^{f, [a,b]}_{\bfZ_\ell} $' is unnecessary}
    is closed under formation of sub-objects,  quotients
    %\com{T: added- reference \cite{DiamondFlachGuoAnnScEcole04}}
    and finite direct sums. It is given by the subcategory  %\com{T: removed "torsion". The reference for this claim is \cite{DiamondFlachGuoAnnScEcole04} and maybe Harris, but he says "subquotients of crystalline $\bfQ_\ell$-representations"}
    ${\rm Rep}_{\bfZ_\ell}^{{\rm cris}, [-b,-a]}(G_{\bfQ_\ell})$.
%\com{T: do we need the functor on non-torsion objects? If we keep it then should follow \cite{DiamondFlachGuoAnnScEcole04} in having submodules instead of lattices in the definition of ${\rm Rep}_{\bfZ_\ell}^{{\rm cris}, [-b,-a]}(G_{\bfQ_\ell})$. And ``torsion-crystalline" should only apply to functor on torsion objects. But then on $MF^{f, [a,b]}_{tor, \bfZ_\ell}$ should have category with sublattices as \cite{Booher19} and \cite{Kalloniatis19}.\\T: I made this change in Definition 4.11.\\J:  okay to remove this comment?}
    %\com{T: reference for this is Harris and DFG, others just focus on torsion case or $\ell$-torsion-free case.\\ The calculation of HT weights for general $a,b$ is not in the references, but follows from the definition via the twists and Booher's Remark 4.11.}
    For $M \in  MF^{f, [a,b]}_{\tor, \bfZ_\ell}$ the lengths  of $M$ and $V_I(M)$ as $\bfZ_\ell$-modules agree; in particular the essential image of $ MF^{f, [a,b]}_{\tor, \bfZ_\ell}$ under $V_{[a,b]}$ is  ${\rm Rep}_{\tor, \bfZ_\ell}^{{\rm cris}, [-b,-a]}(G_{\bfQ_\ell})$.
    % \com{T: a reference for this claim is [RennesHodge] Cor 2.13.\\J:  should we add the reference in?\\T: I've added this to the list of references for the theorem.}
    \item  For $A \in {\rm LCA}_\Oo$, the functor $V_{[a,b]}$ induces a functor from $MF^{f, [a,b]}_{\tor, \bfZ_\ell} \otimes_{\bfZ_\ell} A$ to the category of free finite rank $A$-modules with an $A$-linear $G_{\bfQ_\ell}$-action, which we will also denote by $V_{[a,b]}$.
    %For $A \in {\rm LCA}_\Oo$, this induces $V_{[a,b]}: MF^{f, [a,b]}_{tor, \bfZ_\ell} \otimes_{\bfZ_\ell} A \to {\rm Rep}_A(G_{\bfQ_\ell})$.
    %\com{K: I think we should reword this as what is "this" in the sentence above? We mean $V_{[a,b]}$, but then denote this induced functor the same way. I kind of see a rationale for actually including $A$ in the notation, for example $V_{[a,b]}^A$, but here is a suggested change to the first sentence without notational modifications: For $A \in {\rm LCA}_\Oo$, the functor $V_{[a,b]}$ induces a functor from $MF^{f, [a,b]}_{tor, \bfZ_\ell} \otimes_{\bfZ_\ell} A$ to ${\rm Rep}_A(G_{\bfQ_\ell})$, which we will also denote by $V_{[a,b]}$.\\K: However, has ${\rm Rep}_A(G_{\bfQ_\ell})$ been defined?\\T: No, I added the definition, but removed the notation since this is not used elsewhere?}
    Its essential image is given by ${\rm Rep}^{{\rm cris}, [-b,-a]}_{{\rm free}, A}(G_{\bfQ_\ell})$. In fact,  $V_{[a,b]}$ gives an equivalence of  categories between $ MF^{f, [a,b]}_{\tor, \bfZ_\ell} \otimes_{\bfZ_\ell} A $ and ${\rm Rep}^{{\rm cris}, [-b,-a]}_{{\rm free}, A}(G_{\bfQ_\ell})$.
\end{enumerate}
%\com{T: need to check again the weights- but I think this is consistent with $V(M_n)=\bfF(\omega^{-n})$ and Booher's Remark 4.8. Kalloniatis and Booher both say that the Fontaine-Laffaille weights and Hodge-Tate weights agree under their contra/covariant functors, so they  have a different convention on Hodge-Tate weights (Tristan says cyclo character has HT weight +1, Booher gives it -1). We usually declare the cyclotomic character to have HT weight $+1$, so that's why I get $[-b,-a]$ as the interval for the HT weights.}

\end{theorem}

\begin{rem} {~}
\begin{enumerate}
    \item Note that for $M \in  MF^{f, [a,b]}_{\tor, \bfZ_\ell}$ we have $V_{[a+s,b+s]}(M(s))=V_{[a,b]}(M)(-s)$.
    \item For $I=[a,b]=[0, \ell-2]$ the functor $V_I$ agrees with that of the functor $\mathbb{V}$ in \cite{DiamondFlachGuoAnnScEcole04} p. 670 by \cite{Breuil1998} Proposition 3.2.1.7.
    \item For $M \in MF^{f, [a,b]}_{\tor, \bfZ_\ell} \otimes_{\bfZ_\ell} A$ the Hodge-Tate weights of $V_I(M)$ (in the sense of Definition \ref{defn4.5}(3)) equal the negatives of the Fontaine-Laffaille weights of $M$, defined in Definition \ref{defn2.1}(3), due to our convention that the Hodge-Tate weight of the cyclotomic character is $+1$.
\end{enumerate}

\end{rem}

%\begin{example}   For $n \in [2-\ell,0]$ and $A \in {\rm LCA}_\Oo$ we have that $T_{\rm cris}(M_{n, A})=A(-n)$, where $(-n)$ denotes the Tate twist as defined in Section \ref{s2.1}.
    % \com{K: This notation is fine, but in later sections we do actually write $A(-n)$ for this. Does it make sense to state this somewhere?\\J:  yes, I think adding it here would be a good spot.\\It was defined in section \ref{s2.1}.} %\com{T: Note the minus sign! Booher doesn't state this explicitly, but implicitly in Remark 4.11.}     In particular, for any $n \in \bfZ\cap [a,b]$ we get $V_{[a,b]}(M_{n, \bfF})=\bfF({-n})=\bfF({-n+1-\ell})$, but $M_{n, \bfF} \in MF^{f,[a,b]}_{\tor, \bfZ_\ell} \otimes_{\bfZ_\ell} \bfF$, whereas $M_{n, \bfF}(\ell-1)=M_{n+\ell-1, \bfF}\in MF^{f,[a+\ell-1,b+\ell-1]}_{\tor, \bfZ_\ell} \otimes_{\bfZ_\ell} \bfF$. \end{example}

%\com{J: This is addressing one of the issues that Ravi originally brought up, correct?\\T: this came from discussions with Ravi, Remark 4.20 and 4.21 discuss what this means for the Selmer groups. I since found that \cite{Niziol93} reached similar conclusions to us, and also used a notation for Selmer groups including the specific interval for HT weights, and added a reference to her paper in Remark 4.20.}

As an immediate consequence of the equivalence of categories in Theorem \ref{ThmFoLa}(iv) we obtain the following corollary.

\begin{corollary}
    For any $M, N \in MF_{\tor, \bfZ_\ell}^{f, I} \otimes_{\bfZ_\ell} A$  there is an  isomorphism of $A$-modules \be \label{Ext} {\rm Ext}^1_{MF_{\tor, \bfZ_\ell}^{f, I} \otimes_{\bfZ_\ell} A}(M,N)\cong{\rm Ext}^1_{ {\rm Rep}_{A}^{{\rm cris}, -I}(G_{\bfQ_{\ell}})}(V_I(M),V_I(N)).\ee
\end{corollary}

\subsection{Local Selmer groups}
 Let $I=[a,b]$ be an interval as in the previous section (so $0\leq b-a\leq \ell-2$) but we now also require that $0 \in I$ (so that $\mathbf{1} \in  MF^{f, I}_{\bfZ_\ell}$, see Definition \ref{def4.3}).

    For an extension between two objects $M,N$ in ${\rm Rep}_{A}(G_{\bfQ_\ell})$ $0 \to M \to E \to N \to 0$ we define the $n$-th Tate twist of the extension to be the extension $0 \to M(n) \to E(n) \to N(n) \to 0.$ For a subgroup $G$ of ${\rm Ext}^1_{{\rm Rep}_{A}(G_{\bfQ_{\ell}})}(M,N)$  we define $G(n)$ to consist of extensions which are the $n$-th Tate twists of the elements of $G$.

%\com{T: revised the following discussion}

Given an extension $\mE \in {\rm Ext}^1_{MF^{f, I}_{\tor, \bfZ_\ell} \otimes_{\bfZ_\ell} A}(M_3,M_1)$ represented by an exact sequence  $$0 \to M_1 \to M_2 \to M_3 \to 0$$ we will write $V_I(\mE)$ for the extension in ${\rm Ext}^1_{{\rm Rep}^{{\rm cris}, -I}_{{\rm free}, A}(G_{\bfQ_{\ell}})}$$(V_I(M_3),V_I(M_1))$ represented by $$0 \to V_I(M_1) \to V_I(M_2) \to V_I(M_3) \to 0.$$ This uses the exactness of the functor $V_I$ (cf. Theorem \ref{ThmFoLa}(i)).
Since we defined  $V_I(M)=T_{\rm cris}(M(-b))(-b)$ (see Equation \eqref{doubletwist}) we conclude the following lemma:

\begin{lemma} \label{l4.21}
For
%$I=[a,b]$  \com{J: We already fixed $I$.}and
$A \in {\rm LCA}_\Oo$ and $M \in MF^{f, I}_{\tor, \bfZ_\ell} \otimes_{\bfZ_\ell} A$ we have
\begin{align*}
    V_I({\rm Ext}^1_{MF^{f, I}_{\tor, \bfZ_\ell} \otimes_{\bfZ_\ell} A}(\mathbf{1}_A,M))&={\rm Ext}^1_{{\rm Rep}^{{\rm cris}, -I}_{{\rm free}, A}(G_{\bfQ_\ell})}(T_{\rm cris}(M_{-b, A})(-b),T_{\rm cris}(M(-b))(-b))\\ &\cong {\rm Ext}^1_{{\rm Rep}^{{\rm cris}, [0,\ell-2]}_{{\rm free}, A}(G_{\bfQ_\ell})}(A(b), T_{\rm cris}(M(-b)))(- b).
\end{align*}Note that the latter is naturally isomorphic to ${\rm Ext}^1_{{\rm Rep}^{{\rm cris}, [0,\ell-2]}_{{\rm free}, A}(G_{\bfQ_{\ell}})}(A(b), T_{\rm cris}(M(-b)))$ and they give rise to the same subgroup of $H^1(\bfQ_\ell, V_I(M))$, see Definition \ref{def3.5}.
\end{lemma}

\begin{definition} \label{def3.5}
For $M \in MF^{f, I}_{\tor, \bfZ_\ell} \otimes_{\bfZ_\ell} A$ let $H^1_{f, I}(\bfQ_\ell, V_I(M))=V_I({\rm Ext}^1_{MF^{f, I}_{\tor, \bfZ_\ell}\otimes_{\bfZ_\ell} A}(\mathbf{1}_A,M)) \subset  %H^1(\bfQ_\ell,\Hom_{G_{{\bfQ_\ell}}}(1, V_I(M)))=
H^1(\bfQ_\ell, V_I(M))$.

\end{definition}

%\begin{rem}Lemma \ref{l4.21} shows that all the extensions in $H^1_{f, I}(\bfQ_\ell, V_I(M))$ arise from objects in ${\rm Ext}^1_{MF^{f, [2-\ell,0]}_{\tor, \bfZ_\ell}}$, in particular, extensions between representations of Hodge-Tate weights in the interval $[0,\ell-2]$. \end{rem}

\begin{rem}
 This is a more precise version of the definition made in \cite{BergerKlosin13} Section 5.2.1 (where the prime $\ell$ was denoted by $p$).
In \cite{BergerKlosin13} we worked (implicitly) with $I=[0,p-2]$, but the results in \cite{BergerKlosin13} Section 5 (in particular Corollary 5.4 and Proposition 5.8 restated below) carry over to $H^1_{f, I}$ defined here for general $I$.

    T.B. and K.K. would like to clarify how certain definitions and results in some of our papers fit in with this more precise description of the groups $H^1_{f, I}$: In \cite{BergerKlosin19} the relevant interval $I$ is $I=[1-k, k-1]$ for Section 5, and $p$ should satisfy $p-1>2k-2$. The examples in Section 6 of [loc. cit] satisfy this stronger condition. Similarly in \cite{BergerKlosin20} one has $I=[3-2k, 2k-3]$ ($p-1>4k-6$). In \cite{BergerKlosin13} Section 6  the suitable interval $I$ is such that  $\Hom_\Oo(\tilde \rho_2, \tilde \rho_1)$ has Hodge-Tate weights in  $I$.   For $i, j \in \{1, 2\}$ the local condition at $v \mid p$ for the Selmer groups  $H^1_\Sigma(F, \Hom_\bfF(\rho_i,  \rho_j))$ is $H^1_{f, I}(F_v, \Hom_\bfF(\tilde \rho_i, \tilde \rho_j))$. In [loc.cit.] Section 9 one has $I=[-1,1]$ ($p-1>2$), in Section 10  $I=[1-k,k-1]$ ($p-1>2k-2$). In \cite{BergerKlosin15} Sections 7 and 8 the same comment applies as for \cite{BergerKlosin13} Section 9.

    In J.B.'s paper \cite{BrownCompMath07} the argument in Sections 8 and 9 to show the splitting at $\ell$ of $\begin{pmatrix}
        \ov{\epsilon}^{k-2}&*\\0&\ov{\epsilon}^{k-1}
    \end{pmatrix}$ by relating it to $H^1_f(\bfQ_\ell, \bfF(-1))=0$ requires an interval $I$ containing $-1$ and $2k-3$, so would need $p-1>2k-2$. However, one could instead not twist and invoke Proposition \ref{Kp238}.

    Similar comments apply to other results in the literature, e.g. in \cite{DiamondFlachGuoAnnScEcole04} Corollary 2.3 the expression $H^1_f(\bfQ_\ell, {\rm ad}^0_{\kappa} \ov{L})$ is only indirectly defined by $H^1_f(\bfQ_\ell, {\rm ad}_{\kappa} \ov{L})=H^1_f(\bfQ_\ell, {\rm ad}^0_{\kappa} \ov{L}) \oplus H^1_f(\bfQ_\ell, \kappa)$. To define the Selmer group for the trace zero endomorphisms and prove this identity requires ${\rm ad}^0_{\kappa}$ to lie in the essential image of the Fontaine-Laffaille functor, and therefore $I=[1-k, k-1]$ should be specified, rather than $I=[0, \ell-2]$ as in \cite{DiamondFlachGuoAnnScEcole04} Section 1.1.2.

\end{rem}

If $M,N \in {\rm Rep}_{{\rm free}, A}^{{\rm cris}, I}(G_{\bfQ_{\ell}})$, then $M\oplus N\in {\rm Rep}_{{\rm free}, A}^{{\rm cris}, I}(G_{\bfQ_{\ell}})$ and it is clear that \be \label{eqn4.3} H^1_{f,I}(\bfQ_{\ell}, M\oplus N)=H^1_{f, I}(\bfQ_{\ell}, M)\oplus H^1_{f,I}(\bfQ_{\ell},  N)\ee because the extension groups as well as the functor $V_I$ commute with direct sums.

\begin{prop} \label{lem4.18}
For any $n \in [2-\ell, \ell-2]$  such that $0, -n \in I$ the group $H^1_{f, I}(\bfQ_\ell, V_I(M_{-n, \bfF}))$ is independent of $I$. In fact we have
$$H^1_{f, I}(\bfQ_\ell, \bfF(n))=\begin{cases} 0 & n<0\\
H^1_{\un}(\bfQ_\ell, \bfF) & n=0\\
H^1_{\rm fl}(\bfQ_\ell, \mu_\ell)& n=1\\
H^1(\bfQ_\ell, \bfF(n))& n>1,\end{cases}$$ where $$H^1_{\un}(\bfQ_\ell, \bfF):=\ker(H^1(\bfQ_\ell, \bfF) \to H^1(I_\ell, \bfF)) \cong \Hom(G_{\bfQ_\ell}/I_\ell, \bfF)$$ and $H^1_{\rm fl}(\bfQ_\ell, \mu_\ell)$ denotes the peu ramifi\'ee classes, namely those classes corresponding to $\bfZ_\ell^\times/(\bfZ_\ell^\times)^\ell \subset \bfQ_\ell^\times/(\bfQ_\ell^\times)^\ell \cong H^1(\bfQ_\ell, \bfF(1))$.
For $n \geq 0$ we note that $\dim_\bfF H^1_{f, I}(\bfQ_\ell, \bfF(n)) =1$.
\end{prop}

\begin{rem}
\begin{enumerate}
    \item Proposition \ref{lem4.18} justifies writing $H^1_{\Sigma}(\bfQ_\ell, V_I(M_n))$ as we did in \cite{BergerKlosin19}, without specifying the interval $I$, as long as $I$ contains $-n$. Under the conditions of Proposition \ref{DFG2.2} (see comment after Proposition \ref{prop5.8}), once we have fixed a suitable interval $I$ we will also drop the subscript $I$ in this paper.
\item Note that the definition of $H^1_{f, I}(\bfQ_\ell, V_I(M_n))$ depends on $n \in \bfZ$,
even though the coefficients $V_I(M_n)=\bfF(n)$ only depend on $n \mod{\ell-1}$.
\item \cite{Niziol93} Section 9.3 states a version
of this result for the local crystalline cohomology of unramified extensions of $\bfQ_\ell$ and with $\bfZ_\ell/\ell^m(n)$ coefficients for $m \in \bfZ_{>0}$.
\end{enumerate}

\end{rem}

\begin{proof}
We first note that $H^1(\bfQ_\ell, \bfF(n))$ is $1$-dimensional for $n \neq 0,1$, which follows from local Tate duality and the Euler characteristic formula, see e.g.  \cite{WashingtonBoston} Theorem 1 and Proposition 3.

For $n=0$ we refer the reader to \cite{ClozelHarrisTaylor08} Corollary 2.4.4 for identifying $H^1_{f, I}(\bfQ_\ell, \bfF(n))$ with $H^1_{\un}(\bfQ_\ell, \bfF)$. That $H^1_{\un}(\bfQ_\ell, \bfF)$ is $1$-dimensional follows since $\#H^1(G_{\bfQ_\ell}/I_\ell, \bfF)=\#H^0(\bfQ_\ell, \bfF)$.
Recall that $$H^1_{f, I}(\bfQ_\ell, \bfF(n))=H^1_{f, I}(\bfQ_\ell, V_I(M_{-n, \bfF}))=V_I({\rm Ext}^1_{MF^{f, I}_{\tor, \bfZ_\ell} \otimes_{\bfZ_\ell} \bfF}(M_{0, \bfF},M_{-n, \bfF})).$$
%Since $0 \in I$ we know $b \geq 0$. If $I=[a,b]$ and
If $n<0$
%(in particular $-n \not \in [2-\ell, 0]$)
then
 by Proposition \ref{Kp238}(ii)
${\rm Ext}^1_{MF^{f, I}_{\tor, \bfZ_\ell}\otimes_{\bfZ_\ell} \bfF}(M_{0, \bfF},M_{-n,\bfF})=0$ since the Fontaine-Laffaille weights satisfy the inequality $-n>0$.

On the other hand, if $n> 0$ then $H^1_{f, I}(\bfQ_\ell, V_I(M_{-n}))$ is $1$-dimensional by Proposition \ref{Kp238}(i).
%since $-n < 0$.
For $n>1$ this equals $H^1(\bfQ_\ell, \bfF(n))$ by our observation at the start of the proof.

For $n=1$ we have $H^1(\bfQ_\ell, \bfF(1)) \cong \bfQ_\ell^\times/(\bfQ_\ell^\times)^\ell$ is 2-dimensional, and one can identify the Fontaine-Laffaille extensions with the peu ramifi\'ee classes (see e.g. \cite{BreuilBarcelona} Lemma 8.1.3).
%\com{T: Since Ravi's Lemma 4 doesn't link it to Fontaine-Laffaille theory (and has several typos) I found this alternative reference. An alternative argument goes as follows: Fontaine-Laffaille section 9 show that $MF^{f, [0, 1]}_{tor, \bfZ_\ell}$ is equivalent to finite flat group schemes, and Serre's Duke 1998 article Proposition 3 and 4 show that this corresponds to peu ramifi\'ee extensions. I think this applies to us since ${\rm Ext}^1_{MF^{f, [2-\ell,0]}_{tor, \bfZ_\ell}}(M_{-b},M_{-1-b})\cong {\rm Ext}^1_{MF^{f, [0,1]}_{tor, \bfZ_\ell}}(M_{0},M_{-1})$}
\end{proof}

\begin{rem}

Note that $[2-\ell, 0]$ contains both $0$ and $2-\ell$ (and is the only interval of this length that contains both).  Then since $\bfF(-1)=\bfF(\ell-2)=V_{[2-\ell, 0]}(M_{2-\ell})$ we get
\begin{align*}
    H^1_{f, [2-\ell, 0]}(\bfQ_\ell, \bfF(-1))&=H^1_{f, [2-\ell, 0]}(\bfQ_\ell, \bfF(\ell-2))\\
    &=H^1_{f, [2-\ell, 0]}(\bfQ_\ell, V_{[2-\ell, 0]}(M_{2-\ell, \bfF}))\\ &\neq 0,
\end{align*}
corresponding to the crystalline non-split extension $\begin{pmatrix} \ov{\epsilon}^{\ell-2}&*\\0&1 \end{pmatrix}$.  Note that $1 \notin [2-\ell, 0]$. %, so this does not contradict the result in Proposition \ref{lem4.18} that $H^1_{f,I}(\bfQ_\ell, \bfF(-1))=0$ for $I$ with $1 \in I$.

However for all other intervals $I \subset [2-\ell, \ell-2]$ of length $\ell-2$ we have $1 \in I$ and so
\begin{align*}
H^1_{f, I}(\bfQ_\ell, \bfF(-1))&=V_I({\rm Ext}^1_{ MF^{f, [a,b]}_{\tor, \bfZ_\ell}}(M_{0, \bfF},M_{1, \bfF}))\\
&=T_{\rm cris}({\rm Ext}^1_{MF^{f, [2-\ell,0]}_{\tor, \bfZ_\ell}}(M_{-b, \bfF},M_{1-b, \bfF}))(-b)\\
&=0
\end{align*}
by Proposition \ref{lem4.18}.
This demonstrates that  $H^1_{f, I}(\bfQ_\ell, \bfF(n))$ is only independent of $I$ for $I$ containing $-n$.

%\com{T: could maybe mention that Venkatesh arxiv 16.08 avoids this issue by restricting to the interval $I=[-\frac{p-3}{2},\frac{p-3}{2}]$. At least that's why I think he's choosing this interval. But I think he could work with slightly larger intervals as we do.}

\end{rem}

 Following \cite{BlochKato}   for a $\bfQ_\ell[G_{\bfQ_\ell}]$-module $V$ define
    $\coh_{f}^1(\bfQ_\ell,V)= \ker\left(\coh^1(\bfQ_\ell,V) \rightarrow \coh^1(\bfQ_\ell, V \otimes_{\bfQ_{\ell}} B_{\cris})\right).$ Let $V$ be a finite-dimensional $E$-vector space and $T \subset V$ a $G_{\bfQ_\ell}$-stable $\Oo$-lattice, i.e., $T$ is a free $\Oo$-submodule of $V$ that spans $V$ as a vector space over $E$.  We set $W = V/T$ and $W[\lambda^{m}] = \{w \in W: \lambda^{m}w = 0\} \cong T/\lambda^{m} T$ for any $m \in \bfZ_{>0}$. Note that $W[\lambda^{m}]$ lies in ${\rm Rep}^{{\rm cris}, -I}_{\Oo/\lambda^m}(G_{\bfQ_\ell})$ if $V$ is crystalline with Hodge-Tate weights in $-I$. We let $\coh^1_{f}(\bfQ_\ell,W)$ be the image of $\coh^1_{f}(\bfQ_\ell,V)$ under the natural map $\coh^1(\bfQ_\ell,V) \rightarrow \coh^1(\bfQ_\ell,W)$.

\begin{prop}[\cite{DiamondFlachGuoAnnScEcole04} Proposition 2.2] \label{DFG2.2}
Assume $V$ is a crystalline $E[G_{\bfQ_\ell}]$-module as above with Hodge-Tate weights in $-I=[-b,-a]$ (and $0 \in I$). For $T \subset V$ and $W=V/T$ as above we then have $H^1_{f}(\bfQ_\ell, W)= \dirlim_m H^1_{f,I}(\bfQ_\ell, W[\lambda^m])$.
\end{prop}

%\com{T: repeat definition of $W$ in statement of proposition and corollary?\\J: I don't think we need it in both places.}
\begin{proof}
    We note that the proof of \cite{DiamondFlachGuoAnnScEcole04} Proposition 2.2 carries over from $[0, \ell-2]$ to general $I$ (in particular one has Proposition \ref{prop2.17}) and apply the argument with (in their notation) $V_1$ the trivial $G_{\bfQ_\ell}$-representation and $V_2=V$.
\end{proof}

\begin{corollary}[\cite{DiamondFlachGuoAnnScEcole04} (33), \cite{BergerKlosin13} Corollary 5.4] \label{Cor5.4}
    For every $m \in \bfZ_{>0}$ we have an exact sequence of $\Oo$-modules
    $$0 \to H^0(\bfQ_\ell, W)/\lambda^m \to H^1_{f, I}(\bfQ_\ell, W[\lambda^m]) \to H^1_f(\bfQ_\ell, W)[\lambda^m] \to 0.$$
\end{corollary}

\begin{corollary}
For $n \in \bfZ$ with $0, n \in I \subset[2-\ell, \ell-2]$ and $n \neq 0$ we have $$H^1_{f, I}(\bfQ_\ell, V_I(M_{-n, \bfF}))=%H^1_f(\bfQ_\ell, \Hom_{G_{{\bfQ_\ell}}}(1, E/\Oo(n)))[\lambda]=
H^1_f(\bfQ_\ell, E/\Oo(n))[\lambda].$$
\end{corollary}

\begin{proof}
Note that $H^0(\bfQ_\ell, E/\Oo(n)[\lambda])=0$ since $n \not \equiv 0 \mod{\ell-1}$.
    This implies $H^0(\bfQ_\ell, E/\Oo(n))=0$, hence we are done by Corollary \ref{Cor5.4}.
\end{proof}

\section{Selmer Groups} \label{Selmer group}
\subsection{Definitions} \label{s5.1}

For $M$ a topological $\bfZ_\ell[G_\bfQ]$-module set
\begin{equation*}
    \coh^1_{\un}(\bfQ_{p},M):= \ker\left(\coh^1(\bfQ_{p},M) \rightarrow \coh^1(I_{p},M)\right)
    \end{equation*}
for every prime $p$.
%\com{J: in the previous section we used unr instead of un.}
%\com{K: Jim, could you change $\Oo_E$'s to $\Oo$'s?\\J: done}
Let $E/\bfQ_{\ell}$ be a finite extension with valuation ring $\mO$, uniformizer $\lambda$, and residue field $\bfF$.
Let $V$ be a finite dimensional $E$-vector space on which one has a continuous $E$-linear $G_{\bfQ}$ action. %This gives rise to a continuous homomorphism $\rho: G_{\bfQ} \rightarrow \GL_E(V)$.
For finite primes $p$ with $p \neq \ell$, we set
\begin{equation*}
    \coh^1_{f}(\bfQ_p,V) = \coh^1_{\un}(\bfQ_p,V).
\end{equation*}

For $p=\ell$, we recall from Section \ref{Fontaine-Laffaille} that
\begin{equation*}
    \coh_{f}^1(\bfQ_\ell,V)= \ker\left(\coh^1(\bfQ_\ell,V) \rightarrow \coh^1(\bfQ_\ell, V \otimes_{\bfQ_{\ell}} B_{\cris})\right).
\end{equation*}

 Let $T \subset V$ be a $G_{\bfQ}$-stable $\mO$-lattice. We set $W = V/T$ and $W[\lambda^{n}] = \{w \in W: \lambda^{n}w = 0\} \cong T/\lambda^{n} T$. For every $p$ we let $\coh^1_{f}(\bfQ_p,W)$ be the image of $\coh^1_{f}(\bfQ_p,V)$ under the natural map $\coh^1(\bfQ_p,V) \rightarrow \coh^1(\bfQ_p,W)$. We have $\coh^1_{f}(\bfQ_p,W) = \coh^1_{\un}(\bfQ_p, W)$ for all $p \neq \ell$, as long as $V$ is unramified at $p$, which for us will always be the case.

We define the global Selmer group of $W$ as
 \begin{equation*}
     \coh^1_{f}(\bfQ,W) = \ker \left\{\coh^1(\bfQ,W) \rightarrow \bigoplus_{p} \frac{\coh^1(\bfQ_{p},W)}{\coh^1_{f}(\bfQ_{p}, W)}\right\}.
 \end{equation*}

  We note that as $H^1_{f}(\bfQ_\ell, W)$ commutes with direct sums and so clearly does $H^1_{\un}(\bfQ_\ell, W)$, we get that $H^1_f(\bfQ, W)$ does as well.

 Let $I=[a,b]$ with $a, b \in \bfZ$ and $b-a \leq \ell-2$ and assume that $0 \in I$. If $V$ is crystalline with Hodge-Tate weights in $-I$ we define
 \begin{align*}
     \coh^1_{f, I}&(\bfQ,W[\lambda^n])\\ &= \ker \left\{\coh^1(\bfQ,W[\lambda^n]) \rightarrow \bigoplus_{p \neq \ell} \frac{\coh^1(\bfQ_{p},W[\lambda^n])}{\coh^1_{\rm un}(\bfQ_{p}, W[\lambda^n])}\oplus  \frac{\coh^1(\bfQ_{\ell},W[\lambda^n])}{\coh^1_{f, I}(\bfQ_{\ell}, W[\lambda^n])}\right\}.
 \end{align*}

  As noted in \eqref{eqn4.3}  $H^1_{f}(\bfQ_\ell, W[\lambda^n])$ also commutes with direct sums and so we get that $H^1_{f, I}(\bfQ, W[\lambda^n])$ does as well.

 % We record a slight strengthening of \cite{BergerKlosin13} Proposition 5.8:
  \begin{prop} \label{prop5.8}
   Assume that the interval $I=[a,b]$ contains $0$ and $V$ is $E[G_{\bfQ}]$-module which is  finite-dimensional as an $E$-vector space and a crystalline $G_{\bfQ_{\ell}}$-module with Hodge-Tate weights in $-I$.
 If $H^0(\bfQ, W[\lambda])=0$ then we have
 $$H^1_f(\bfQ, W)[\lambda^n] \cong H^1_{f,I}(\bfQ, W[\lambda^n]).$$

  \end{prop}

\begin{proof}
 \cite{BergerKlosin13} Proposition 5.8 proves the claim under the assumption $H^0(\bfQ, W)=0$.

Suppose we have $\alpha \in \coh^0(\bfQ, W)$.  We know every element of $W$ is annihilated by some power of $\lambda$, so if $\alpha \neq 0$ there is an integer $m$ so that $\lambda^{m} \alpha =0$ but $\lambda^{n} \alpha \neq 0$ for all $0 < n < m$.  However, this gives $\lambda^{m-1} \alpha \in \coh^0(\bfQ,  W[\lambda]) =0$, so it must be that $\alpha = 0$.  Thus, $\coh^0(\bfQ, W) = 0$ as desired.
\end{proof}
 %If $H^0(\bfQ, W)=0$ then by Proposition \ref{DFG2.2} (see e.g. \cite{BergerKlosin13} Proposition 5.8) we have $$H^1_f(\bfQ, W)[\lambda^n] \cong H^1_{f,I}(\bfQ, W[\lambda^n]).$$
 After a suitable interval $I$ has been fixed we will therefore also drop the subscript $I$ and write $H^1_{f}(\bfQ, W[\lambda^n])$.

 Let $G$ be a group, $R$ a commutative ring with identity, and $M_{i}$ finitely generated free $R$-modules with $R$-linear action given by $\rho_{i}: G \rightarrow \GL_R(M_i)$ for $i=1,2$.  The action of $G$ on $\Hom_{R}(\rho_2,\rho_1)$  is given by
$(g\cdot \varphi)(v) = \rho_1(v) \varphi(\rho_2(g^{-1})v).$
In particular, if $\rho_1 = \rho_2 = \rho$, we define the adjoint representation of $\rho$ to be the $R[G]$-module $\ad \rho = \Hom_R(\rho,\rho)$.  We write $\ad^0\rho$ for the $R[G]$-submodule of $\ad\rho$ consisting of endomorphisms of trace zero.

If $\rho$ is of rank $n$ and $2n \in R^\times$ then we have an isomorphism of $R[G]$-modules
\be \label{eq5.2} \ad \rho \cong\ad^0 \rho \oplus R.
\ee

\subsection{Non-vanishing of a Selmer group} \label{s5.2}

In this section we explain how the congruence of a Siegel cusp form to the Klingen Eisenstein series in Section \ref{s3} leads to a non-zero element of $H^1_f(\bfQ, {\rm ad}^0(\rho_{\phi, \lambda})(2-k) \otimes E/\Oo)$.

%\com{J: I reordered the beginning a bit here..}

From now on,  we fix the weight $k \geq 12$ even and the prime $\ell$ satisfying $\ell>4k-5$ and impose Assumption \ref{admis} on the field $E/\bfQ_\ell$.
Let $\phi \in S_{k}(\Gamma_1)$ be a normalized eigenform. %\com{J: changed from ``newform'' to ``normalized eigenform'' to match what we have earlier in the paper.}
Let $\rho_{\phi}$ be the $\lambda$-adic Galois representation associated to $\phi$ and assume $\ov{\rho}_{\phi}$ is irreducible.  Let $f \in S_{k}(\Gamma_2)$ be an eigenform with irreducible Galois representation $\rho_{f}$ so that $f$ is eigenvalue congruent to $E_{\phi}^{2,1}$ modulo $\lambda$.

The following result shows we can choose a lattice so that the residual Galois representation  gives rise to a non-split extension.

%\com{K: Still need to clean up the notation in the lemma below\\T: should $g$ be called $f$ for consistency with later discussion? Also should repeat assumption that $\ov{\rho}_\phi$ is (absolutely?) irreducible. \\J: I think we need to move the sentence ``Let $\phi \in S_{k}(\Gamma_1)$ be as in Theorem \ref{cong 1}.  Let $\rho_{\phi,\lambda}$ be the $\lambda$-adic Galois representation associated to $\phi$.'' to before the lemma and add the absolutely irreducible. It has been a while since $\phi$ was used.  Also, $G$ is for a general group above but now the Galois group so we should specify what $G$ is here.\\T: replaced $G$'s by $G_\bfQ$.}
\begin{lemma}\label{Ribet1}
% Let $g\in S_k(\Gamma_2)$ be any eigenform such that $E_{\phi}^{1,2}\equiv_{\rm ev}g$ (mod $\lambda$).
There exists a $G_\bfQ$-stable lattice  in the space of $\rho_f$ such that with respect to this lattice $$\ov{\rho}_{f} = \bmat \ov{\rho}_{\phi} &* \\ &\ov{\rho}_{\phi}(k-2)\emat\not\cong \ov{\rho}_{\phi}\oplus\ov{\rho}_{\phi}(k-2).$$ \end{lemma}

\begin{proof} Using the compactness of $G_\bfQ$ one can show that there exists a $G_\bfQ$-stable lattice $\Lambda'$ in the space of $\rho_f$. %In other words, there exists a $\GL_4(E)$-conjugate $\rho_{f, \Lambda'}$ of $\rho_f$ which has image in $\GL_4(\Oo)$. As $$\tr \rho_f(\Frob_p)\equiv \tr\rho_{\phi}(\Frob_p)+\tr \rho_{\phi}(k-2)(\Frob_p)\pmod{\lambda}$$ for all primes $p\neq \ell$
%\com{J: what does the $\equiv$ mean here?\\K: Fixed it.}
One uses Brauer-Nesbitt Theorem together with the Chebotarev Density Theorem to conclude that
$\ov{\rho}_{f, \Lambda'}^{\rm ss}=\ov{\rho}_{\phi}\oplus \ov{\rho}_{\phi}(k-2)$. Now the existence of the desired lattice which gives the non-split extension follows from Theorem 4.1 in \cite{BergerKlosin20}.
\end{proof}

From now on, whenever we write $\rho_f$, we assume we have made a choice of lattice as in Lemma \ref{Ribet1}, so we consider $\rho_f$ as a map from $G_\bfQ$ to $\GL_4(\Oo)$.

%\com{T: added}
We now choose the interval $I=[3-2k,2k-3]$ so that it contains all the Hodge-Tate weights of $\rho_f$, $\rho_\phi$, $\rho_\phi(k-2)$, $\ad \rho_\phi(2-k)$, and $\ad \rho_\phi(k-2)$. Note that $-I=I$. We assume that $\ell-2\geq 4k-6$. When we write $H^1_f$ from now on this refers to $H^1_{f, I}$ as defined in Section \ref{s5.1}.
%\com{end add}

Let $\rho$ be any of the representations above and write $V$ for the representation space of $\rho$. We choose a $G_{\bfQ}$-stable lattice $T\subset V$ and recall that the isomorphism class of the semi-simplification of the $\bfF[G_{\bfQ}]$-representation $T/\lambda T$ is independent of the choice of $T$.  It is well-known that if $T/\lambda T$ is irreducible then the $\Oo$-length of  $H^1_f(\bfQ, W)$ is independent of $T$, where as before $W=V/T$. By Proposition \ref{prop5.8} we then conclude that also the $\Oo$-length of $H^1_f(\bfQ,W[\lambda^n])$ is independent of the choice of $T$ as long as $H^0(\bfQ,W)=0$.

\begin{lemma} \label{Lemma 5.3}
    Under our assumptions (in particular, $\ov{\rho}_{\phi}$  irreducible and $\ell >4k-5$) the modulo $\lambda$ reduction of $\ad^0\rho_{\phi}$ is irreducible.
\end{lemma}

\begin{proof}
    Assume the three-dimensional representation $\ad^0\ov{\rho}_{\phi}$ is reducible. Then it either has a one-dimensional $G_{\bfQ}$-stable subspace or quotient. Since $\ad\rho_{\phi}$ and $\mathbf{1}$ are self-dual, so is $\ad^0 \ov{\rho}_{\phi}$. Hence we can assume without loss of generality that $\ad^0\ov{\rho}_{\phi}$ has a $G_{\bfQ}$-stable line. Write $\psi$ for the character by which $G_{\bfQ}$ acts on the line.

    As $\ov{\rho}_{\phi}$ is unramified away from $\ell$ and the order of $\psi$ is prime to $\ell$, %the splitting field of $\psi$ must be a subfield of $\bfQ(\zeta_{\ell^{\infty}})$. As the order of $\psi$ is prime to $\ell$, this splitting field must be a subfield of $\bfQ(\mu_{\ell})$, so
    we have $\psi=\ov{\epsilon}^a$ for some  integer $a \in I$.
This would require $H^0(\bfQ, \ad^0 \ov{\rho}_{\phi}(-a)) \neq 0$. Note that $H^0(\bfQ, \ad \ov{\rho}_{\phi}(-a))=\Hom_{G_\bfQ}(\ov{\rho}_\phi(a), \ov{\rho}_\phi).$
    If $a\equiv 0$ (mod ($\ell-1$)), then this space is one-dimensional by Schur's Lemma since $\ov{\rho}_{\phi}$ is irreducible. So, $H^0(\bfQ, \ad^0\ov{\rho}_{\phi})=0$, contradiction.

    If $a\not\equiv 0$ (mod ($\ell-1$)), then $H^0(\bfQ, \ad \ov{\rho}_{\phi}(-a))=H^0(\bfQ, \ad^0 \ov{\rho}_{\phi}(-a))\neq 0.$
    This means that $\ov{\rho}_{\phi}$ is isomorphic to $\ov{\rho}_{\phi}(a)$. Considering the determinant, $\ov{\epsilon}^a$ must be  the trivial character or the quadratic character  $\ov{\epsilon}^{(\ell-1)/2}$. Both are ruled out since $a \in I=[3-2k, 2k-3]$ by our assumption that $\ell >4k-5$.
\end{proof}

\begin{rem} \label{remlattice} From Lemma \ref{Lemma 5.3} we conclude that when $\rho\in \{\rho_{\phi}, \rho_{\phi}(k-2), \ad^0\rho_{\phi}(2-k), \ad^0\rho_{\phi}(k-2)\}$, the $\Oo$-length of $H^1_f(\bfQ, W)$ and $H^1_f(\bfQ, W[\lambda^n])$ are independent of the choice of $T$. As we will ever only be interested in the order of these groups, the choice of $T$ is immaterial and we will simply assume that such a choice was made. So, for example we will use the notation $H^1_f(\bfQ, \ad^0\rho_{\phi, \lambda}(k-2)\otimes E/\Oo)$, thus assuming that when we write  $\ad^0\rho_{\phi, \lambda}(k-2)$, we have made a choice of a lattice for this representation.
Likewise any one-dimensional representation $\rho$ is irreducible, so the $\Oo$-length of  $H^1_f(\bfQ, \rho\otimes E/\Oo)$ is independent of the choice of $T$.

For the representation $\ad\rho(m)$, $m\in \{k-2, 2-k\}$ (which is reducible) we choose a lattice which is a direct sum of a lattice inside $\ad^0\rho(m)$ and a lattice inside $E(m)$. So, from now on whenever we write $\ad\rho(m)$ we mean such a lattice.
Since the formation of Selmer groups commutes with direct sums we then get
\be \label{dirsum 1}H^1_f(\bfQ, \ad\rho_{\phi}(m)\otimes E/\Oo)= H^1_f(\bfQ, \ad^0\rho_{\phi}(m)\otimes E/\Oo)\oplus H^1_f(\bfQ, E/\Oo(m))\ee
for $m\in \{k-2,2-k\}$.
Note that  the $\Oo$-length (and in particular, the non-triviality) of $H^1_f(\bfQ, \ad \rho(m)\otimes E/\Oo)$  is independent of the choice of a lattice inside $\ad\rho_{\phi}(m)$ as long as it is the direct sum of lattices in $\ad^0\rho_{\phi}(m)$ and $E(m)$.

\end{rem}

\begin{theorem} \label{Selmerbound} With the set-up as above we have $H^1_f(\bfQ,  \ad\rho_{\phi}(2-k)\otimes E/\Oo)\neq 0$.
\end{theorem}

\begin{proof}
We have via Lemma \ref{Ribet1} that there is a lattice $T_{f} \subset V_{f}$ so that the residual representation $\ov{\rho}_{f}: G_{\bfQ} \rightarrow \GL_4(\bfF)$ has the form
\begin{equation} \label{eqn5.2}
    \ov{\rho}_{f} = \bmat \ov{\rho}_\phi & \psi \\ 0 & \ov{\rho}_\phi(k-2) \emat
\end{equation}
and is not semisimple.
The fact that $\psi$ as in  \eqref{eqn5.2}  gives a non-trivial class $[\psi]$ in
%\com{T: the macro  backslash coh doesn't italicize H, whereas in the rest of the paper the cohomology groups are italicized. \\J: fixed this}
$\coh^1(\bfQ,\Hom_{\bfF}(\ov{\rho}_2, \ov{\rho}_1))=
\coh^1(\bfQ, \ad\rho_{\phi}(2-k)\otimes E/\Oo[\lambda])$ is clear.   We need to show that $[\psi]$ lies in $\coh^1_f(\bfQ, \ad\rho_{\phi}(2-k)\otimes E/\Oo[\lambda])$ and that the latter group injects into $\coh_f^1(\bfQ, \ad\rho_{\phi}(2-k)\otimes E/\Oo)$.

We first show that $[\psi]$ satisfies the conditions to be in $\coh^1_{f}(\bfQ, \ad\rho_{\phi}(2-k)\otimes E/\Oo[\lambda])$. We have that ${\rho}_f$ is unramified at all primes $p \neq \ell$, so the local conditions are satisfied for all primes $p \neq \ell$.

Since $f$ has level one and weight $k$,  $\rho_{f}|_{D_{\ell}}$ is crystalline with Hodge-Tate weights in $[0, 2k-3] \subset I=-I$. Hence $\ov{\rho}_f$ (considered as a  $G_{\bfQ_{\ell}}$-module)  belongs to ${\rm Rep}_{{\rm free}, \bfF}^{{\rm cris}, I}(G_{\bfQ_{\ell}})$ and gives rise to an element of ${\rm Ext}^1_{ {{\rm Rep}_{{\rm free}, \bfF}^{{\rm cris}, I}(G_{\bfQ_{\ell}})}}(\ov{\rho}_\phi(k-2),\ov{\rho}_\phi)\subset \Ext^1_{\bfF[G_{\bfQ_{\ell}}]}(\rho_\phi(k-2) \otimes E/\Oo[\lambda], \rho_\phi \otimes E/\Oo[\lambda]).$
By our choice of $I$ we can use \eqref{Ext} and Proposition \ref{Homtensor} to get a non-zero element in $${\rm Ext}^1_{ {{\rm Rep}_{{\rm free}, \bfF}^{{\rm cris}, I}(G_{\bfQ_{\ell}})}}(\bfF, \ad\rho_{\phi}(2-k)\otimes E/\Oo[\lambda])\subset \Ext^1_{\bfF[G_{\bfQ_{\ell}}]}(\bfF,  \ad\rho_{\phi}(2-k)\otimes E/\Oo[\lambda]).$$

As this extension maps to  $[\psi|_{G_{\bfQ_\ell}}]$ in $H^1(\bfQ_{\ell},  \ad\rho_{\phi}(2-k)\otimes E/\Oo[\lambda])$ under the canonical isomorphism $\Ext^1_{\bfF[G_{\bfQ_{\ell}}]}(\bfF,  \ad\rho_{\phi}(2-k)\otimes E/\Oo[\lambda])\cong H^1(\bfQ_{\ell},  \ad\rho_{\phi}(2-k)\otimes E/\Oo[\lambda])$, we conclude that $$[\psi|_{G_{\bfQ_\ell}}] \in H^1_{f}(\bfQ_{\ell},  \ad\rho_{\phi}(2-k)\otimes E/\Oo[\lambda])\subset H^1(\bfQ_{\ell},  \ad\rho_{\phi}(2-k)\otimes E/\Oo[\lambda]).$$
Therefore we have established that
$[\psi] \in H^1_{f}(\bfQ,  \ad\rho_{\phi}(2-k)\otimes E/\Oo[\lambda]).$  By Proposition \ref{prop5.8} this group is isomorphic to $H^1_{f}(\bfQ,  \ad\rho_{\phi}(2-k)\otimes E/\Oo)[\lambda]$ if $\coh^0(\bfQ, \ad\rho_{\phi}(2-k)\otimes E/\Oo[\lambda])=0$.
The latter holds since
\be \label{trivial invariant} \ad\rho_{\phi}(2-k)\otimes E/\Oo[\lambda]^{G_\bfQ}=\Hom_{G_\bfQ}(\ov{\rho}_\phi(k-2), \ov{\rho}_\phi)=0\ee as $\ov{\rho}_\phi$ and $\ov{\rho}_\phi(k-2)$ are  absolutely irreducible (by assumption) and non-isomorphic since $k-2 \not \equiv 0, \frac{\ell-1}{2} \pmod{\ell-1}$
%\com{T: added $\frac{\ell-1}{2}$ as it's a priori possible for $\ov{\rho}_\phi$ to be isomorphic to its twist by the quadratic character $\ov{\epsilon}^{(\ell-1)/2}$. These are the only two options, since $\Hom_{G_\bfQ}(\ov{\rho}_\phi(k-2), \ov{\rho}_\phi)\neq 0$ means $\ov{\rho}_\phi \cong \ov{\rho}_\phi(k-2)$, and then taking determinants requires $(\ov{\epsilon}^{k-2})^2 \equiv 1$ }
as $\ell>4k-5$ and $k \neq 2$ (cf. the proof of Lemma \ref{Lemma 5.3}).
%(note that $k\neq 2$ since $S_2(\Gamma_2)=0$).
\end{proof}

%\com{T: new version to allow application both to the Corollary and to Lemma 6.3}

\begin{lemma} \label{vanishing 1}
     Let $n$ be an even integer satisfying $3-2k<n \leq 0$. Assuming $\ell \nmid \#\Cl_{\bfQ(\zeta_{\ell})^{+}}^{\ov{\epsilon}^{n}}$, one has $H^1_f(\bfQ, \bfF(n))=0$ and, if additionally $n \neq 0$, $H^1_f(\bfQ, E/\Oo(n))=0$.

\end{lemma}

\begin{proof}
    We see from Proposition \ref{lem4.18} that any cohomology class in $H^1_f(\bfQ, \bfF(n))$ must vanish when restricted to $I_{\ell}$. As all classes in $H^1_f(\bfQ, \bfF(n))$ are unramified away from $\ell$, we  get that they are unramified everywhere.  Using inflation-restriction sequence where $H=\Gal(\bfQ(\zeta_{\ell})^{+}/\bfQ)$ we see that $$H^1(\bfQ, \bfF(n))\cong H^1(\bfQ(\zeta_{\ell})^{+}, \bfF(n))^H=\Hom_H(G_{\bfQ(\zeta_{\ell})^{+}}, \bfF(n)).$$ Note that everywhere unramified classes map to homomorphisms that kill all the inertia groups. Hence the image of $H^1_f(\bfQ, \bfF(n))$ lands inside $\Hom\left(\Cl_{\bfQ(\zeta_{\ell})^{+}}^{\ov{\epsilon}^{n}}, \bfF\right)=0$.

Note that  a torsion $\Oo$-module $M$ is zero if and only if $M[\lambda]=0$.
%Since every element of $H^1_f(\bfQ, E/\Oo(n))$ is torsion
     Therefore the vanishing of $H^1_f(\bfQ, E/\Oo(n))$ follows from Proposition \ref{prop5.8}, which tells us that $H^1_f(\bfQ, E/\Oo(n))[\lambda]=H^1_f(\bfQ, \bfF(n))$ if $H^0(\bfQ,E/\Oo(n))=0$.
      We know that $H^0(\bfQ_\ell,E/\Oo(n)[\lambda])=H^0(\bfQ,\bfF(n))=0$ for $n \neq 0$ since $n \not \equiv 0 \pmod{\ell-1}$ under our assumption $\ell>4k-5$. %and note that any non-zero element of $H^0(\bfQ, E/\Oo(2-k))$ could be multiplied by an appropriate power of $\lambda$ to give  a non-zero element of  $H^0(\bfQ_\ell,E/\Oo(2-k)[\lambda])=H^0(\bfQ_\ell,\bfF(2-k))$.
%\com{T: Not sure if $H^1_f(\bfQ, E/\Oo)=0$ since Proposition \ref{prop5.8} doesn't apply.\\ The residual vanishing for  $n=0$ is used in Lemma 6.3.}
\end{proof}

\begin{corollary}\label{BK conj1} Let $\phi \in S_{k}(\Gamma_1)$ be as in Theorem \ref{cong 1} and assume the hypotheses of Theorem \ref{cong 1} are satisfied.
      Assuming $\ell \nmid \#\Cl_{\bfQ(\zeta_{\ell})^{+}}^{\ov{\epsilon}^{2-k}}$ one has $H^1_f(\bfQ, \ad^0\rho_{\phi}(2-k)\otimes E/\Oo)\neq 0$.
\end{corollary}

\begin{proof}
    This follows from Theorem \ref{Selmerbound}, Lemma \ref{vanishing 1} and isomorphism \eqref{dirsum 1}.
\end{proof}

\begin{rem} \label{rem5.8} If we assume Vandiver's conjecture for the prime $\ell$, this gives that
$\ell \nmid \#\Cl_{\bfQ(\zeta_{\ell})^{+}}^{\ov{\epsilon}^{2-k}}$.
%\com{J: changed $\omega$ to $\ov{\epsilon}$.}
\end{rem}

%\begin{cor} \label{BK conj1} Let $\phi \in S_{k}(\Gamma_1)$ be as in Theorem \ref{cong 1} and assume the hypotheses of Theorem \ref{cong 1} are satisfied.  Furthermore, assume Vandiver's conjecture is true for the prime $\ell$.  Then we have $$\ell \mid \#\coh_{f}^1(\bfQ,\ad^0 \ov{\rho}_{\phi}(2-k)).$$ \end{cor}

\section{Modularity}\label{Modularity}

We begin with the following commutative algebra result that will be useful in this section.

\begin{lemma}\label{commalg1}
If $J$ is an ideal of $\bfF[[X_1,\dots, X_n]]$  that is strictly contained in the maximal ideal, then $\bfF[[X_1,\dots, X_n]]/J$ admits an $\bfF$-algebra surjection to $\bfF[T]/T^2$.
\end{lemma}

\begin{proof}
For a positive integer $k$ let $I_k$ be the ideal of $\bfF[[X_1,\dots, X_k]]$ generated by all the monomials of degree at least 2. Set $S_k:=\bfF[[X_1,\dots, X_k]]/I_k$ and write $\phi_k: \bfF[[X_1,\dots, X_k]]\to S_k$ for the canonical $\bfF$-algebra surjection. If $\phi_n(J)=0$, then composing $\phi_n$ with the map $S_n\to \bfF[[T]]/T^2$ sending $X_1$ to $T$ and $X_i$ for $i>1$ to zero gives the desired surjection.

%Then $S:=\bfF[[X_1,\dots, X_n]]/I$ is an $\bfF$-vector space of dimension $n+1$ and we have an $\bfF$-algebra surjection $\phi: \bfF[[X_1,\dots, X_n]]\to S$. The ideal $\phi(J)$ (note that it is indeed an ideal as $\phi$ is onto) is a vector subspace of $\bfF[[X_1,\dots, X_n]]/I$ of dimension, say, $r$. If $\phi(J)=0$, then the map $S\to \bfF[[T]]/T^2$ sending $X_1$ to $T$ and $X_i$ for $i>1$ to zero gives the desired surjection.

Now suppose $\phi_n(J)\neq 0$. %\com{T: What are the $S_k$ for? Do we just need $S_n$?\\Needed to talk about $\phi_k$ later}
Without loss of generality (renumbering the variables if necessary) we may assume then that  $J$ contains an element of  the form $u:=X_n+f(X_1,\dots, X_{n-1})+g(X_1,\dots, X_n)$, where $f$ is homogeneous of degree one and all the terms in $g$ have degree at least 2. Note that we can assume without loss of generality that some power of $X_n$ appears in $g$. (Indeed, if $g$ contains no $X_n$  then we replace $u$ by $u+u^2 \in J$.)
%\com{T: do we also need to assume that $g$ depends on $X_n$ for the later statement "Thus for any power series $h$   where the smallest total degree of any term containing $X_n$ is $s$ we have $$h\equiv h' \pmod{J}$$ for some power series $h'$ with the smallest total  degree of any term containing $X_n$  strictly greater than $s$. "?\\ Maybe one can achieve this by relabelling $X_n$ plus the constant terms of $f$ and $g$ as $X_n$ and then adding the square of $X_n+f+g$?}
By Theorem 7.16(a) in \cite{Eisenbud} there is a unique $\bfF$-algebra map from $\bfF[[X_1,\dots, X_n]]$ to itself sending $X_n$ to $-f-g$ and $X_i$ to itself for $i<n$. In other words, for any power series $h(X_1,\dots, X_n)$, the element $h(X_1,\dots, X_{n-1}, -f-g)$ also lives in $\bfF[[X_1, \dots, X_n]]$ and we denote it by $h'(X_1, \dots, X_n)$. Clearly $h-h'\in J$.

Thus for any power series $h$   where the smallest total degree of any term containing $X_n$ is $s$ we have $$h\equiv h' \pmod{J}$$ for some power series $h'$ with the smallest total  degree of any term containing $X_n$ equal to $s'>s$. By the same process we get an $h''$ such that $h' \equiv h'' \mod{J}$ and the smallest total degree of any term $X_n$ in $h''$ is strictly greater than $s'$. This way we can construct a sequence of power series $h_s$ where
%\com{J: I don't understand this part.  Where is the sequence coming from? My understanding is $h$ is a fixed power series, so the $s$ value is fixed based upon $h$ so I'm unclear how one is getting a different $h_{s}$ for each $s$.  For instance, if the smallest non-zero degree in $X_{n}$ is 5, I see how one would get $h_{5}$ but I don't see where $h_6, h_7$, etc. are coming from.}
for every $s$ we have the smallest total degree of any term containing $X_n$ being greater than or equal to $s$ and such that $h- h_s\in J$  for every $s$. We note that $h_s$ is a Cauchy sequence with respect to the $(X_1,\dots,X_n)$-adic topology (indeed, for $t,u>s$ we see that $h_t-h_u$ lies in $(X_1,\dots,X_n)^s$).  Set $h_0=\lim_{s\to \infty}h_s$. As $J$ is a closed ideal, we get that $h_0-h\in J$. For every $s$ we have $$h_0 \equiv h_s \equiv w_s \mod{X_n^s},$$ for some $w_s \in \bfF[[X_1, \dots, X_{n-1}]]$. Note that the $w_s$ also form a Cauchy sequence since $h_s$ does. Set $w:= \lim_{s \to \infty} w_s \in \bfF[[X_1, \dots, X_{n-1}]]$. Thus $h_0 \equiv w$  modulo $\bigcap_s (X_n^s) \subset \bigcap_s (X_1, \ldots, X_n)^s=0$, so $h_0 \in \bfF[[X_1, \dots, X_{n-1}]]$.

Hence the natural $\bfF$-algebra map $\psi_{n-1}:\bfF[[X_1,\dots, X_{n-1}]]\to \bfF[[X_1, \dots, X_n]]/J$ given by $h_0 \mapsto h_0+J$ is surjective. Thus we get an $\bfF$-algebra isomorphism $\bfF[[X_1,\dots, X_n]]/J \to \bfF[[X_1, \dots, X_{n-1}]]/J_{n-1}$, where $J_{n-1}=\ker \psi_{n-1}$.

%So, we get an isomorphism $\psi:\bfF[[X_1,\dots, X_n]]/J\cong \bfF[[X_2,\dots, X_n]]/J_2$ for some ideal $J2$.
If $\phi_{n-1}(J_{n-1})\neq 0$, continue this way obtaining a sequence of ideals $J_{n-2},J_{n-3},...$. If at any stage ($1\leq r\leq n-2$) we get $\phi_{n-r}(J_{n-r})=0$, then we are done. Otherwise we can eliminate all but one variable and get $\bfF[[X_1,\dots, X_n]]/J\cong \bfF[[X_1]]/J_1$ and now we must have $\phi_1(J_1)=0$ as otherwise $J_1$ and hence $J$ is maximal.
%Note that $r<n+1$, as all of the power series in $J$ have zero constant terms, since $J$ cannot contain units. Observe that any $\bfF$-linear automorphism of $\bfF[[X_1,\dots, X_n]]/I$ is also an $\bfF$-algebra automorphism. So, if the codimension of $\phi(J)$ in $\bfF[[X_1,\dots, X_n]]/I$ is at least 2, we can apply a change of basis automorphism $\psi$, so that $\bfF[[X_1,\dots, X_n]]/I\xrightarrow{\psi}_{\cong} \bfF[[T_1, \dots, T_n]]/(I',\psi(\phi(J))$ where $I'$ is the ideal generated by all the monomials of degree at least 2 and no element in $\psi(\phi(J))$ contains the variable $T_1$. Then the map sending $T_1$ to $T$ and the rest of the variables $T_2, \dots, T_n$ to 0, gives the desired surjection onto $\bfF[T]/T^2$.   If $r=n$ then $\phi(J)$ is a maximal ideal.
\end{proof}

Recall that in the earlier sections we fixed the weight $k \geq 12$ even and prime $\ell > 4k-5$ and imposed Assumption \ref{admis} on the field $E/\bfQ_\ell$. We also fixed the Fontaine-Laffaille interval $I=[3-2k, 2k-3]$.
Let $\phi\in S_k(\Gamma_1)$ be a newform such that $\ov{\rho}_{\phi}$ is irreducible.
The goal of this section is to prove a modularity theorem under the following  assumption:
% \com{J: this reads a bit weird to me.  It says we are strengthening above assumptions, but above assumptions give us such an $f$.  So we are really weakening some as presumably sometimes such an $f$ could exist outside the assumptions of Theorem \ref{cong 1}, but strengthening the assumptions to assume that we get $\#H^1_f(\bfQ, W)=\#\Oo/\lambda$. Maybe we should remove that discussion in the previous paragraph and then move it to right below (where we already kind of state it again.)\\TK: We have revised this as suggested.}

\begin{assumption} \label{assumption section 6}
    For $k$, and $\phi$ as above we assume that

\begin{itemize}
    \item[(i)] there exists $f \in S_k(\Gamma_2)$ such that $f\equiv_{\rm ev}E^{2,1}_{\phi}$ (mod $\lambda$), and
    \item[(ii)] $\#H^1_f(\bfQ, {\rm ad}^0\rho_{\phi}(2-k)\otimes_{\Oo}E/\Oo)=\#\Oo/\lambda$ (recall that the left hand side is independent of the choice of lattice, see Remark \ref{remlattice}), and
    \item[(iii)] $H^1_f(\bfQ, \ad^0\ov{\rho}_{\phi})=0.$
\end{itemize}
\end{assumption}

\begin{rem}
Assumption \ref{assumption section 6} (i) is satisfied under the assumptions of Theorem \ref{cong 1}, and so is one inequality in Assumption \ref{assumption section 6} (ii) under the assumptions of Corollary \ref{BK conj1}. \end{rem}

% \com{J: Since we are assuming the result $\val_{\lambda}(L_{\rm alg}(2k-2, \Symm\phi))<0$, this additional assumption gives $\#H^1_f(\bfQ,W) = \#\Oo/\lambda$ correct?  Should we also just fix the assumption $\#H^1_f(\bfQ,W)\leq\#\Oo/\lambda$ throughout the section since it seems to be used for everything rather than restating it in each result?\\
% K: Yes, we can fix this assumption in this section. When I was writing it, I didn't yet envision having a follow up section and in that follow-up section we would not make that assumption - the folloow up section is currently the section "R=T"}

% \com{J: since we are assuming $\#H^1_f(\bfQ,W) = \# \Oo/\lambda$, some cases are forced to be cyclic based on how $\ell$ ramifies in $\Oo$ correct?\\
% K: Yes. This is related to my previous comment. You only get non-cyclicity if $\#H^1_f(\bfQ,W) > \# \Oo/\lambda$}

%More precisely, From now on we assume that $\#H^1_f(\bfQ, W)=\#\Oo/\lambda$ and that $\ov{\rho}_{\phi}$ is irreducible. By Theorem \ref{cong 1} there exists $f\in S_k(\Gamma_2)$, not a Saito-Kurokawa lift such that $E_{\phi}^{1,2}\equiv_{\rm ev}f$ (mod $\lambda$). We will fix such $f$ in what follows.
% For any eigenform $g\in S_k(\Gamma_2)$ as before we denote by $\rho_g: G_{\bfQ}\to \GL_4(E)$ the $\ell$-adic Galois representation attached to $g$ (see Theorem \ref{Weissauer}).  In particular, such $\rho_g$ is unramified away from $\ell$. Lemma \ref{lem:galirred} gives that $\rho_{g}$ is irreducible.
We impose Assumption \ref{assumption section 6} and fix $f$ as in Assumption \ref{assumption section 6} in what follows. We will write $G_{\{\ell\}}$ for the Galois group of the maximal Galois extension of $\bfQ$ unramified away from $\ell$.
Let $\rho_f: G_{\{\ell\}}\to \GL_4(E)$  be as in Theorem \ref{Weissauer}.   Lemma \ref{lem:galirred} gives that $\rho_{f}$ is irreducible.
We will use  Mazur's deformation theory  and refer the reader to standard references such as \cite{MazurDefTheory,Ramakrishna}
%\com{J: Which ramakrishna paper should we cite here? This one?\\K: Yes, that's the one.}
for the definitions and basic properties.

\begin{definition}
        For $B \in {\rm LCN}_\Oo$ we say that a representation $\rho:G_{\bfQ_\ell} \to \GL_n(B)$ is \emph{Fontaine-Laffaille} (with Hodge-Tate weights in $-I$) if $\rho \otimes_B A$ lies in ${\rm Rep}_{{\rm free}, A}^{{\rm cris}, -I}(G_{\bfQ_\ell})$ (see Definition \ref{defn4.5}(v)) for every Artinian quotient $A$ of $B$. By Theorem \ref{ThmFoLa}(iv) this is equivalent to requiring  $\rho \otimes_B A$ to lie in the essential  image of the Fontaine-Laffaille functor.
    \end{definition}

\begin{rem} \label{rem6.4}
We know that any choice of $\Oo$-lattice $\rho_L$ in $\rho_\phi$ or $\rho_f$ is Fontaine-Laffaille in this sense, since their restrictions to $G_{\bfQ_\ell}$ lie in ${\rm Rep}_{\bfZ_\ell}^{{\rm cris}, -I}(G_{\bfQ_{\ell}})$ and therefore in the essential image of the Fontaine-Laffaille functor by Theorem \ref{ThmFoLa}(iii). Since they are also free $\Oo$-modules this implies by Theorem \ref{ThmFoLa} (iii) and (iv) that $\rho_L \otimes B$ lies in ${\rm Rep}_{{\rm free}, A}^{{\rm cris}, -I}(G_{\bfQ_\ell})$ for every Artinian quotient $B$ of $\Oo$.

\end{rem}

For any local complete Noetherian $\Oo$-algebra $A$ with residue field $\bfF$ by a deformation of a residual Galois representation $\tau: G_{\{\ell\}} \to \GL_n(\bfF)$ we will mean a strict equivalence class of lifts $\tilde{\tau}:G_{\{\ell\}}\to \GL_n(A)$ of $\tau$ that are Fontaine-Laffaille at $\ell$. This deformation condition is introduced in \cite{BergerKlosin13} Section 5.3 and \cite{ClozelHarrisTaylor08} p.35.

%\com{T: should use notation from Section 4.1, e.g. ${\rm Rep}^{f, -I}_{A}(G_{\bfQ_\ell})$ and take out the next sentence ``Our notion of a Fontaine-Laffaille representation agrees with the notion of a `short crystalline' representation in \cite{DiamondFlachGuoAnnScEcole04}", since they define this notion only for $\bfQ_\ell$ representations (and  we replace their interval $[0, \ell-2]$ of allowed Hodge-Tate weights by $[3-2k,2k-3]$, so that e.g. also $\rho_\phi^\vee$ is in our category).\\But in Section 4.1 we only allowed Artinian $A$. But as in \cite{BergerKlosin13} end of section 5 (taken from \cite{ClozelHarrisTaylor08} p. 35) we can define a deformation $\rho$ to $A$ FoLa if $\rho \otimes_A A'$ is in ${\rm Rep}^{f, -I}_{A}(G_{\bfQ_\ell})$ for each Artinian quotient $A'$ of $A$.}
As is customary, we will denote a strict equivalence class of deformations by any of its members. If $\tau$ has scalar centralizer then this deformation problem is representable by a local complete Noetherian $\Oo$-algebra which we will denote by $R_{\tau}$ \cite{Ramakrishna02}. %\com{T: added citation- is this the correct Ramakrishna paper? I haven't been able to view it}
In particular, the identity map in $\Hom_{\Oo-{\rm alg}}(R_{\tau},R_{\tau})$ furnishes what is called the universal deformation $\tau^{\rm univ}: G_{\{\ell\}}\to \GL_n(R_{\tau})$.

\begin{lemma} \label{diag1} %$L^{\rm alg}(k-1,{\rm Sym}^2\phi)\in \Oo^{\times}$.
One has $R_{\ov{\rho}_{\phi}}\cong R_{\ov{\rho}_{\phi}(k-2)}\cong \Oo$.
Furthermore, $\rho_{\phi}$ (resp. $\rho_{\phi}(k-2)$) is the unique deformation of $\ov{\rho}_{\phi}$ (resp. $\ov{\rho}_{\phi}(k-2)$) to $\GL_2(\Oo)$.

  %Let $R_{\ov{\rho}_{\phi}}$ (resp. $R_{\ov{\rho}_{\phi}(k-2)}$) denote the universal deformation ring parametrizing deformations of $\ov{\rho}_{\phi}$ (resp. $\ov{\rho}_{\phi}(k-2)$) that are Fontaine-Laffaille at $p$. If there is no newform other than $\phi$ in $S_k(\Gamma_1)$ that is congruent to $\phi$ modulo $\lambda$, then $R_{\rho_{\phi}}\cong R_{\rho_{\phi}(k-2)}\cong \Oo$ with $\rho_{\phi}$ (resp. $\rho_{\phi}(k-2)$) being the respective universal deformations.
 \end{lemma}

\begin{proof}
We have \be \label{eqntangent} \#\Hom_{\Oo-{\rm alg}}(R_{\ov{\rho}_{\phi}}, \bfF[X]/X^2)=\#H^1_f(\bfQ, \ad\ov{\rho}_{\phi})=0,\ee where the first equality follows from the fact that our deformation condition is the property of being Fontaine-Laffaille (see e.g., Section 2.4.1 \cite{ClozelHarrisTaylor08}), and the second one holds since
we have $H^1_f(\bfQ, \ad\ov{\rho}_{\phi})=H^1_f(\bfQ, \ad^0\ov{\rho}_{\phi}) \oplus H^1_f(\bfQ, \bfF)=0$ and $H^1_f(\bfQ, \bfF)=0$ by Lemma \ref{vanishing 1} as we have imposed Assumption \ref{assumption section 6}(iii).

By Theorem 7.16 in \cite{Eisenbud} we know that any local complete Noetherian $\Oo$-algebra with residue field $\bfF$ is a quotient of $\Oo[[X_1, \dots, X_n]]$ for some positive integer $n$.
Hence $S:=R_{\ov{\rho}_{\phi}}/(\lambda R_{\ov{\rho}_{\phi}})\cong\bfF[[X_1, \dots, X_n]]/J$ for some ideal $J$. Suppose first that $J$ is not maximal. Then by Lemma  \ref{commalg1} we know that $S$ admits a surjection $\varphi$ to $\bfF[T]/T^2$. This contradicts \eqref{eqntangent}, hence $S=\bfF$.
We now use the complete version of Nakayama's Lemma to conclude that the structure map $\Oo\to R_{\ov{\rho}_{\phi}}$ is surjective (cf. \cite{Eisenbud}, Exercise 7.2 or \cite{Matsumura} Theorem 8.4). Let us briefly explain why this version applies here. As $R_{\ov{\rho}_{\phi}}\otimes_{\Oo}\bfF\neq 0$, we see that $\lambda \in \fm$, where $\fm$ is the maximal ideal of $R_{\ov{\rho}_{\phi}}$. Hence \be \label{sep1} \bigcap_n \lambda^n R_{\ov{\rho}_{\phi}} \subset \bigcap_n\fm^n.\ee The latter intersection is zero, since $R_{\ov{\rho}_{\phi}}$ is complete, so separated with respect to $\fm$. Hence \eqref{sep1} implies that $R_{\ov{\rho}_{\phi}}$ is separated with respect to $\lambda R_{\ov{\rho}_{\phi}}$ allowing for the application of the complete version of Nakayama's Lemma.

As $\rho_{\phi}$ is a deformation to $\Oo$, we conclude that $R_{\ov{\rho}_{\phi}}=\Oo$.
%[Alternatively, we could have argued as in the proof of Corollary \ref{Osurj} using  the triviality of $H^1_f(\bfQ, \ad\ov{\rho}_{\phi})$ and the second equality in \eqref{eqntangent} in lieu of Lemma \ref{infi}.]
This implies that if $\rho: G_{\{\ell\}}\to \GL_2(\Oo)$ is  any deformation of $\ov{\rho}_\phi$, one has $\rho \cong \rho_{\phi}$. Similarly, if $\rho: G_{\{\ell\}} \to \GL_2(\Oo)$ is a deformation of $\ov{\rho}_{\phi}(k-2)$ then $\rho(2-k)$ is a deformation of $\ov{\rho}_{\phi}$. Note that our choice of $I=[3-2k, 2k-3]$ means that this twisting stays inside our category of Fontaine-Laffaille representations. Hence we get that $\rho(2-k)\cong \rho_{\phi}$, and so we are done.\end{proof}

\begin{rem}
Note that the determinant of our deformations is automatically fixed  as $H^1_f(\bfQ, \ad\ov{\rho}_{\phi})=H^1_f(\bfQ, \ad^0\ov{\rho}_{\phi})$ under our assumptions. This  means that all  deformations $\rho$ of  $\ov{\rho}_{\phi}$ (respectively $\ov{\rho}_{\phi}(k-2)$)  satisfy $\det \rho=\epsilon^{k-1}$ (respectively  $\det \rho=\epsilon^{2k-3}$).

\end{rem}

\begin{rem} \label{rem 6.5}
    Regarding Assumption \ref{assumption section 6}(iii) we note that if one additionally assumes that $\ov{\rho}_\phi$ is absolutely irreducible when restricted to $\Gal(\ov{\bfQ}/\bfQ(\sqrt{(-1)^{(\ell-1)/2}\ell})$ then \cite{DiamondFlachGuoAnnScEcole04} Theorem 3.7 (see also \cite{HidaMFG} Theorem 5.20) relates $H^1_f(\bfQ, {\rm ad}^0 \rho_\phi \otimes E/\Oo)$ (via an $R_{\ov{\rho}_{\phi}}=\bfT$ theorem) to a  congruence ideal $\eta_\phi^\emptyset$. One can use Proposition \ref{prop5.8} to see that $H^1_f(\bfQ, \ad^0\ov{\rho}_{\phi})=H^1_f(\bfQ, {\rm ad}^0 \rho_\phi \otimes E/\Oo)[\lambda]=0$ if $\eta_\phi^\emptyset$ is coprime to $\ell$.

    %\com{T: $\eta_\phi^\emptyset$ is a fractional ideal in the coefficient field $K$ of $\phi$. Is my formulation ok, without having to change to $\lambda \mid \ell$ a prime in $K$?}
    %\com{T: Should we also mention that they prove $R_{\ov{\rho}_{\phi}}=\bfT_\fm$, which should directly give $R_{\ov{\rho}_{\phi}}=\Oo$?\\E.g. something like: We note that if one additionally assumes that $\ov{\rho}_\phi$ is absolutely irreducible when restricted to $\Gal(\ov{\bfQ}/\bfQ(\sqrt{(-1)^{(\ell-1)/2}\ell})$ then \cite{DiamondFlachGuoAnnScEcole04}   Theorem 3.7 (see also \cite{HidaMFG} Theorem 5.20) relates (via an $R_{\ov{\rho}_{\phi}}=\bfT_\fm$ theorem) $H^1_f(\bfQ, {\rm ad}^0 \rho_\phi \otimes E/\Oo)$ to a  congruence ideal $\eta_\phi^\emptyset$.}
    %\com{T: I think this also uses Proposition \ref{prop5.8} to deduce the vanishing of the residual Selmer group from that of the divisible one.\\T:  Proposition \ref{prop5.8} doesn't applies since $$H^0(G_\bfQ, \ad^0 \rho_{\phi} \otimes E/\Oo)=0$$- the only $G_\bfQ$-invariants for $\ad \rho_\phi$ come from Schur's lemma and they will not be contained in $\ad^0$. Need to also use that $\ov{\rho}_\phi$ is irreducible.}
\end{rem}

  %\com{K: replaced $\ov{\rho}_{f, \Lambda}$ with $\ov{\rho}_{f}$\\K: BUt maybe keep $\ov{\rho}_{f}$ or at least $\ov{\rho}$?\\T: I replaced all the $\rho$ by $\ov{\rho}_{f}$} for some  lattice  as in Proposition \ref{Ribet1}.
 % \com{J: Do we still want $g$ here or is there a reason we go back to $f$?\\
 % K: The reason is thaty now we are fixing our residual representation, so we want $f$.}

\begin{lemma} \label{scalar centralizer}
Let $G$ be a group and $F$ be a field. For $i\in \{1,2\}$, let $n_i \in \bfZ_+$ and $\rho_i: G \to \GL_{n_i}(F)$ be an irreducible representation with $\rho_1\not\cong \rho_2$. Let $\rho: G \to \GL_{n_1+n_2}(F)$ be a representation such that $$\rho = \bmat \rho_1 & a \\ & \rho_2 \emat\not\cong \rho_1 \oplus \rho_2.$$ Then $\rho$ has scalar centralizer.
\end{lemma}

\begin{proof}
    This is a simple consequence of Schur's Lemma and the fact that $\tilde{a}: g \to \rho_2(g)^{-1}a(g)$ defines a cocycle from $G$ to $\Hom(\rho_2, \rho_1)$ which is not a coboundary. %The fact that $\rho$ is non-semi-simple implies that $\tilde{a}$ is not a coboundary. For $i,j\in \{1,2\}$ let $A_{i,j} \in {\rm Mat}_{n_1, n_2}(F)$ be such that $\bmat A_{1,1}& A_{1,2}\\ A_{2,1} & A_{2,2}\emat$ centralizes $\rho$. Using the fact that $\rho_1, \rho_2$ are irreducible and non-isomorphic one gets that $A_{2,1}=0$. Then Schur's Lemma gives us that $A_{1,1}=\alpha I_{n_1}$ and $A_{2,2}=\delta I_{n_2}$ for scalars $\alpha, \delta$. Finally one gets \be \label{sc1} \alpha a+ A_{1,2}\rho_2=\rho_1A_{1,2}+a \delta.\ee If $\alpha \neq \delta$, then one gets $$\tilde{a}=\frac{1}{\alpha-\delta}(\rho_1 A_{1,2} \rho_2^{-1}-A_{1,2}),$$ contradicting the fact that it is not a coboundary. Thus $\alpha=\delta$ and \eqref{sc1} now gives that $A_{1,2}=0$.
\end{proof}
 Fix a lattice in the space of $\rho_f$ as in Lemma \ref{Ribet1}, i.e. such that $\ov{\rho}_{f}=\bmat \ov{\rho}_{\phi} &* \\ &\ov{\rho}_{\phi}(k-2)\emat: G_{\{\ell\}}\to \GL_4(\bfF)$ is non-semisimple. For simplicity, we will  write $R$ for the universal deformation ring $R_{\ov{\rho}_{f}}$  of $\ov{\rho}_{f}$ and $\rho^{\rm univ}: G_{\{\ell\}}\to \GL_4(R)$ for the universal deformation. Note that the deformation problem is representable because $\ov{\rho}_{f}$ is non-semisimple with irreducible, mutually non-isomorphic Jordan-Holder factors, hence by Lemma \ref{scalar centralizer}  the centralizer of $\ov{\rho}_f$ consists of only scalar matrices.  We say that a deformation $\tilde\rho$ is \emph{upper-triangular} if $\tilde\rho$ is strictly equivalent to a deformation of $\ov{\rho}_{f}$ of the form $\bmat *&*\\ 0&*\emat$ with the stars representing $2\times 2$ blocks.

\begin{lemma} \label{infi} There do not exist any non-trivial deformations of $\ov{\rho}_{f}$ into $\GL_4(\bfF[X]/X^2)$ that are \emph{upper-triangular}.
\end{lemma}

\begin{proof} We use Proposition 7.2 in \cite{BergerKlosin13} noting that Assumption 6.1(i) in [loc.cit.] is satisfied because we impose the current Assumption \ref{assumption section 6}(ii).
%that $\#H^1_f(\bfQ, {\rm ad}^0\rho_{\phi, \lambda}(2-k)\otimes_{\Oo}E/\Oo)\leq \#\Oo/\lambda$.
On the other hand, Assumption 6.1(ii) in [loc.cit.] is satisfied because of Lemma \ref{diag1}. %\com{K: Sort of unfortunately in the current paper we also have Assumptions 6.1(i) and (ii)}
\end{proof}

\begin{definition}
The smallest ideal $I$ of $R$ such that $\tr \rho^{\rm univ}$ is the sum of two pseudocharacters mod $I$ will be called the \emph{reducibility ideal} of $R$. We will denote this ideal by $I_{\rm re}$.
\end{definition}

\begin{prop}\label{universality} Let $I\subset R$ be an ideal such that $R/I$ is an Artin ring. Then $I\supset I_{\rm re}$ if and only if $\rho^{\rm univ}$ (mod $I$) is \emph{upper-triangular}.
\end{prop}

\begin{proof}
This is proved as Corollary 7.8 in \cite{BergerKlosin13}.
\end{proof}

\begin{corollary} \label{Osurj}  The structure map $\Oo\to R/I_{\rm re}$ is surjective  and descends to an isomorphism $\Oo/\lambda^s \to R/I_{\rm re}$  for some $s\in \bfZ_{\geq 0}\cup\{\infty\}$. In fact, one has $$R/I_{\rm re} \cong\Oo/\lambda.$$
\end{corollary}

%\com{T: I think we could change this result to $R^{\rm red}$ and work with this ring from now on.}

\begin{proof}

By Theorem 7.16 in \cite{Eisenbud} we know that any local complete Noetherian $\Oo$-algebra with residue field $\bfF$ is a quotient of $\Oo[[X_1, \dots, X_n]]$ for some positive integer $n$.
% \com{K: OK, after re-examining Theorem 7.16 I think it works because $R/I_{\rm re}$ is an object in our category, so has residue field $\bfF$. The theorem requires that (see part (b)) for the map $\Oo[[X_1, \dots, X_n]]\to R/I_{\rm re}$ to be surjective, the induced map $\Oo[[X_1, \dots, X_n]]\to (R/I_{\rm re})/\fn$ to be surjective. Here $\fn$ is the ideal with respect to which $R/I_{\rm re}$ is complete. For us it is always the maximal ideal, so $(R/I_{\rm re})/\fn=\bfF$. As $\bfF$ is the residue field of $\Oo$, any $\Oo$-algebra map into $\bfF$ must be surjective. I will insert this into the main text (and add ``with residue field $\bfF$''), but I first wanted to check with you if you agree. What prompted me to look at this is the situation where $R/I_{\rm re}=$some extension of $\bfF$ (not possible for us - as the residue field would then not be $\bfF$). Then Theorem 7.16 I think does not give a surjection.\\
% J: I agree with this.  We use this result in the proof of Lemma 5.2 as well, so maybe the citation should be moved there. \\K: I added `with residue field $\bfF$' here and in the proof of Lemma 5.2 together with a citation - feel free to remove this comment if you are ok with it}
Hence $S:=R/(I_{\rm re}+\lambda R)\cong\bfF[[X_1, \dots, X_n]]/J$ for some ideal $J$. Suppose first that $J$ is not maximal. Then by Lemma  \ref{commalg1} we know that $S$ admits a surjection $\varphi$ to $\bfF[T]/T^2$. This means that there
%exist at least two distinct elements of $\Hom_{\Oo-{\rm alg}}(R/I_{\rm re},\bfF[T]/T^2)$ - the surjection $\varphi$ and the map $R/I_{\rm re}\twoheadrightarrow \bfF\hookrightarrow \bfF[T]/T^2$. By universality of $R$, this means that there exist two non-equivalent deformations of $\rho$ to $\bfF[T]/T^2$, the trivial one and the one corresponding to $\varphi$. Note that both of them are upper-triangular by Lemma \ref{universality}.
exists a non-trivial (because the image of $\varphi$ is not contained in $\bfF$) deformation of $\rho$ to $\bfF[T]/T^2$ which is upper-triangular (by Proposition \ref{universality}), which contradicts Lemma \ref{infi}.
Thus, indeed, $S=\bfF$.

Hence, the structure map $\Oo\to R/I_{\rm re}$ is surjective by the complete version of Nakayama's Lemma (see the proof of Lemma \ref{diag1}).
%cf. \cite{Eisenbud}, Exercise 7.2 or \cite{Matsumura} Theorem 8.4). Note that this version applies as $ \bigcap_n \lambda^n (R/I_{\rm re}) \subset \bigcap_n\fm^n=0.$
So,  $R/I_{\rm re}\cong \Oo/\lambda^s$ for some $s\in \bfZ_{\geq 0}\cup\{\infty\}$.
%The latter intersection is zero, since $R/I_{\rm re}$ is complete, so separated with respect to $\fm$. Hence \eqref{sep1} implies that $R/I_{\rm re}$ is separated with respect to $\lambda (R/I_{\rm re})$ allowing for the application of the complete version of Nakayama's Lemma.
%\com{K: I'm afraid we don't know a priori that $R$ is a finitely generated $\Oo$-module. So here we maybe indeed need the complete version?}
%This proves that  the structure map $\Oo\to R/I_{\rm re}$ is surjective

The composition of $\rho^{\rm univ}$ with the map $R\to R/I_{\rm re}$ gives rise to a deformation $\rho_{\rm re}: G_{\{\ell\}}\to \GL_4(R/I_{\rm re})=\GL_4(\Oo/\lambda^s)$.   By Proposition \ref{universality}, this deformation is upper triangular, i.e., one has
$\rho_{\rm re}=\bmat *_1&*_2\\ &*_3\emat.$
%\com{K: The next sentence is OK, but should say only ``we see that $*_1$ and $*_3$ are FL'' without mentioning $\rho_{\rm re}$. If the latter is a deformation as we claimed two sentences ago (and it is by universality), then it is FL}
As the property of being Fontaine-Laffaille is preserved by subobjects and quotients, we see that  $*_1$ and $*_3$ are Fontaine-Laffaille representations with values in $\GL_2(R/I_{\rm re})=\GL_2(\Oo/\lambda^s)$. Thus by Lemma \ref{diag1} we can conclude that $*_1=\rho_{\phi}$, $*_3=\rho_{\phi}(k-2)$ mod $\lambda^s$. Hence by \eqref{trivial invariant}  and Proposition \ref{prop5.8}  $*_2$ gives rise to a class in $H_f^1(\bfQ,{\rm ad}^0\rho_{\phi}(2-k)\otimes_{\Oo}E/\Oo)$ as $\rho_{\rm re}$ is Fontaine-Laffaille. As $\rho$ is non-semi-simple, we conclude that $*_2$ is not annihilated by $\lambda^{s-1}$, i.e., the class of $*_2$ gives rise to a subgroup of $H^1_f(\bfQ,{\rm ad}^0\rho_{\phi}(2-k)\otimes_{\Oo}E/\Oo)$ isomorphic to $\Oo/\lambda^s$. Thus $s\leq 1$ as  $\#H^1_f(\bfQ, {\rm ad}^0\rho_{\phi}(2-k)\otimes_{\Oo}E/\Oo)\leq \#\Oo/\lambda$ by Assumption \ref{assumption section 6}(ii). Finally, $s>0$ as $\ov{\rho}_{f}$ itself is reducible.
%\com{T: I think at this point we could observe that $R \to \Oo \to \bfF$ corresponding to $\rho=\ov{\rho}_{f, \Lambda}$ factors through $R \to R^{\rm red} \to R^{\rm red}/I_{\rm re}$}
This concludes the proof.
\end{proof}

%\begin{prop} \label{bound1} Suppose that $\#H^1_f(\bfQ, W)\leq \#\Oo/\lambda$. The one has $$R/I_{\rm re} \cong\Oo/\lambda.$$\end{prop}

%\begin{proof} By Corollary \ref{Osurj} the composition of $\rho^{\rm univ}$ with the map $R\to R/I_{\rm re}$ gives rise to a deformation $\rho_{\rm re}: G\to \GL_4(R/I_{\rm re})=\GL_4(\Oo/\lambda^s)$.   By Proposition \ref{universality}, this deformation is upper triangular, i.e., one has $$\rho_{\rm re}=\bmat *_1&*_2\\ &*_3\emat.$$ By Lemma \ref{diag1} we can conclude that $$*_1=\rho_{\phi}, \quad *_3=\rho_{\phi}(k-2)\pmod{\lambda^s}.$$ Thus $*_2$ gives rise to a class in $H_f^1(\bfQ,W)$ as $\rho^{\rm univ}$ and hence also $\rho_{\rm re}$ is Fontaine-Laffaille. As $\rho$ is non-semi-simple, we conclude that $*_2$ is not annihilated by $\lambda^{s-1}$, i.e., the class of $*_2$ gives rise to a subgroup of $H^1_f(\bfQ,W)$ isomorphic to $\Oo/\lambda^s$. Thus $s\leq 1$ as  $\#H^1_f(\bfQ, W)\leq \#\Oo/\lambda$.\end{proof}

%\begin{rem}\label{eq1} Note that as $\rho$ is a reducible deformation of itself, $I_{\rm re}$ is a proper ideal of $R$, so in fact Proposition \ref{bound1} gives us that $R/I_{\rm re}\cong \Oo/\lambda$.\end{rem}

%\begin{prop} \label{gen by traces}The ring $R$ is topologically generated as an $\Oo$-algebra by $\{\tr \rho^{\rm univ}(\Frob_{\ell})\mid \ell\neq p\}$\end{prop}

The following Proposition does not use Assumption \ref{assumption section 6}(ii).
\begin{prop} \label{princ of I}
Assume that $\dim H^1_f(\bfQ, \ad\ov{\rho}_{\phi}(k-2))\leq 1$. Then the ideal $I_{\rm re}$ is a principal ideal.
\end{prop}

% \com{J: we do want the twist to be $k-2$ and not $2-k$ here correct?\\
% K: Yes. For principality we want the opposite twist to the one coming from congruence bounds.\\
% Also, why do we switch to dimension here instead of order?\\K: Sorry, no reason. We have to think what is the best formulation for the reader.\\
% J: How strong of an assumption is the assumption on the dimension?  Is it expected this is true often?  I don't feel like I have a feel for how strong this assumption and the one that $\#H_{f}^1(\bfQ,W) \leq \#\Oo/\lambda$ are.\\
% K: Whether an assumption is strong is a matter of opinion. Certainly it should be the case that it is satisfied by most modules, but this may still mean that it is strong. For what it's worth I published a few papers with this kind of assumptions in place. The assumption on dimension is also (a lot?) weaker that the one on the order. This is so because it only assumes that the divisible Selmer group is cyclic, but places no restriction on the order.}

\begin{proof}
Since $\rho^{\rm univ}$ is a trace representation in the sense of  Section 1.3.3 of \cite{BellaicheChenevierbook} Lemma 1.3.7 in [loc.cit.] tells us that we can conjugate $\rho^{\rm univ}$ by a matrix $P \in \GL_2(R)$ (here we use that every finite type projective $R$-module is free since $R$ is local) to get $\rho^{\rm univ}$ adapted to a data of GMA idempotents for $R[G_{\{\ell\}}]/\ker \rho^{\rm univ}$. %\com{K: Should we say somewhere that we mean the kernel of the algebra homomorphism and not the group homomorphism?}
By \cite{BellaicheChenevierbook} Lemma 1.3.8 we then get an isomorphism of $R$-modules $$R[G_{\{\ell\}}]/\ker \rho^{\rm univ}\cong\bmat \Mat_2(R)&\Mat_2(B)\\\Mat_2(C)&\Mat_2(R) \emat$$ for ideals $B, C \subset R$. By \cite{BellaicheChenevierbook} Proposition 1.5.1 we further  know that $I_{\rm re}=BC$.

\cite{BellaicheChenevierbook} Theorem 1.5.5 proves that there are injections $\Hom_R(B, \bfF) \hookrightarrow H^1(G_{\{\ell\}}, \ad\ov{\rho}_{\phi}(2-k))$
and
 $\Hom_R(C, \bfF) \hookrightarrow H^1(G_{\{\ell\}}, \ad\ov{\rho}_{\phi}(k-2)).$
 Arguing as in \cite{Akers24} Proposition 4.2 (see also \cite{WWEForum} Theorem 4.3.5 and Remark 4.3.6) one sees that the images are contained in the Selmer groups $H^1_f(\bfQ, \ad\ov{\rho}_{\phi}(2-k))$ and $H^1_f(\bfQ, \ad\ov{\rho}_{\phi}(k-2))$, respectively.
From Assumption \ref{assumption section 6} (ii) and Proposition \ref{prop5.8} we see that $H^1(\bfQ, \ad\ov{\rho}_{\phi}(2-k)) \cong \bfF$. Together with the assumption $\dim H^1_f(\bfQ, \ad\ov{\rho}_{\phi}(k-2))\leq 1$ we deduce by Nakayama's Lemma that both $B$ and $C$, and therefore also $I_{\rm re}$ are principal ideals of $R$. Note that Nakayama's Lemma applies since $B$ and $C$ are ideals in $R$, which is Noetherian, hence they are finitely generated over $R$.
\end{proof}

\begin{rem}
%Note that there is a natural anti-involution on $R[G_{\{\ell\}}]$ given by $g\mapsto \epsilon^{2k-3}(g)g^{-1}$, however it swaps $\ov{\rho}_{\phi}$ with $\ov{\rho}_{\phi}(k-2)$, so the results of section 2 of \cite{BergerKlosin13} guaranteeing the principality of the reducibility ideal do not apply in this case.

\cite{Akers24} Proposition 3.10 proves the principality of the reducibility ideal of the reduced Fontaine-Laffaille deformation ring $R^{\rm red}$ for any residual representations with two Jordan-H\"older factors.
Our argument (whilst relying on  \cite{Akers24} Proposition 4.2) is slightly more general as it allows us to treat the case of non-reduced deformation rings.
\end{rem}

% \com{K: I am not sure if the following remark is a problem for us. It may be the case, I think, that $\dim_{\bfF}H^1_f(\bfQ, \bfF(k-2))$=1, in which case our assumption requires that $H^1_f(\bfQ, \ad^0\ov{\rho}_{\phi}(k-2))=0$. Does this sound like it would be a problem?\\
% J: I don't have a good feel for when such a thing would be a problem, so to me, no... but that shouldn't count for much.. maybe ask Tobias if it seems like it would be a problem?\\ J: I wrote Romyar about this.  I'll let you know when/if I hear anything back.\\J: After I clarified I have not heard anything back.  Not real comfortable prodding him again.  Maybe we have to just leave it as is?\\K: Yes. Let's see what the referee says. Could also send the preprint to Skinner once it is done and ask about that remark.\\
% J: Sounds good!}
\begin{rem}
    By \eqref{dirsum 1} we have
    $$H^1_f(\bfQ, \ad\ov{\rho}_{\phi}(k-2))=H^1_f(\bfQ, \ad^0\ov{\rho}_{\phi}(k-2))\oplus H^1_f(\bfQ, \bfF(k-2)).$$ However, as opposed to the case of the $(2-k)$-twist of the trivial representation (cf. proof of Lemma \ref{vanishing 1}), there is no simple relation between $H^1_f(\bfQ, \bfF(k-2))$ and part of a class group except for the case $k=2$ by Proposition \ref{lem4.18}.
%    \com{T: took out: ``(cf. Proposition 1.6.4 in \cite{RubinEulerSystems2000})." since I think this result only says something for $k=3$, as in \cite{BergerKlosin20} Proposition 2.9. As we are only looking at even $k$ here this is probably not worth mentioning here.\\I think there is a typo in the discussion just before \cite{BergerKlosin19} Proposition 6.4- it should say $k=2$ there.} \com{T: took out: ``See also \cite[pg. 8060]{BergerKlosin19} for a detailed discussion regarding a bound on this group." and replaced it with a more precise reference below.}
By the same Proposition  for $2<k\leq \ell$ the group $H^1_f(\bfQ, \bfF(k-2))$ requires no ramification condition at $\ell$, so equals $H^1(G_{\{\ell\}}, \bfF(k-2))$. %\com{T: I think it would be good to write something like $G_{\{p\}}$ or $G(p)$ instead of just $G$ to indicate that it allows ramification at $p$}
    %We refer the reader to \cite{BergerKlosin19} Proposition 6.5 that shows that $\dim_\bfF H^1_f(\bfQ, \bfF(k-2)) \leq 1$ for $k>2$ if one has $\ell \nmid \#\Cl_{\bfQ(\zeta_{\ell})^{+}}^{\omega_{\ell}^{k-2}}$ (which would be true under Vandiver's conjecture).
\end{rem}
We have the following results about $H^1(G_{\{\ell\}}, \bfF(n))$ for $n>0$:
\begin{prop}[\cite{BergerKlosin19} Proposition 6.5]
    Suppose $n \in \bfZ_{>0}$ and $n \not \equiv 1 \mod{\ell-1}$. Assume that $\ell \nmid \#\Cl_{\bfQ(\zeta_{\ell})}^{\ov{\epsilon}^{n}}$. Then $\dim H^1(G_{\{\ell\}}, \bfF(n)) \leq 1$.
\end{prop}

\begin{prop} \label{prop6.10}
    Let $n>0$ be an even integer. Assume $\ell \nmid B_n$ (the $n$-th Bernoulli number) and $n \not \equiv 0\mod{\ell-1}$. Then $H^1(G_{\{\ell\}}, \bfF(n))=0$.
\end{prop}

\begin{proof}
    Since $n$ is even and $H^0(G_{\{\ell\}}, \bfF(n))=0$ as $n \not \equiv 0\mod{\ell-1}$ we know $\dim_\bfF H^1(G_{\{\ell\}}, \bfF(n))=\dim_\bfF H^2(G_{\{\ell\}}, \bfF(n))$ by \cite{NSW} Corollary 8.7.5 (Euler Poincare characteristic). \cite{Assim95} Proposition 1.3 (condition $(ii, \beta)$) proves that $H^2(G_{\{\ell\}}, \bfF(n))=0$ if $n \not \equiv 1 \mod{\ell-1}$ (which is automatically satisfied for even $n$) and $\ell \nmid \#\Cl_{\bfQ(\zeta_{\ell})}^{\ov{\epsilon}^{1-n}}$.
  By Herbrand's Theorem  (see e.g. \cite{WashingtonCyclotomic} Theorem 6.17) the latter follows from our assumption that $\ell \nmid B_n$ (here we use again $n \not \equiv 0\mod{\ell-1}$). %to satisfy that $\ell-n \not \equiv 1 \mod{\ell-1}$).
\end{proof}

\begin{rem}
    Note that the  assumption $\ell \nmid B_{n}$ is stronger than  $\ell \nmid \#\Cl_{\bfQ(\zeta_{\ell})}^{\ov{\epsilon}^{n}}$ in \cite{BergerKlosin19} Proposition 6.5. As noted in the proof of Proposition \ref{prop6.10}  $\ell \nmid B_{n}$ implies $\ell \nmid \#\Cl_{\bfQ(\zeta_{\ell})}^{\ov{\epsilon}^{\ell-n}}$ by Herbrand's Theorem.  By the ``reflection theorem" \cite{WashingtonCyclotomic} Theorem 10.9 this means that also $\ell \nmid \Cl_{\bfQ(\zeta_{\ell})}^{\ov{\epsilon}^{n}}$.

\end{rem}
    %\com{T:\cite{Niziol93} Table 2 in Section 9.4 calculates $H^1_f(\bfQ, \bfF(n))$ for even $n$, but her definition of the global Selmer group is different to ours. Her $H^1_f(\bfQ, \bfF(n))$ vanishes if the crystalline cohomology group (defined in \cite{Niziol93} Section 5) $H^2_f(\bfQ, \bfZ_\ell(n))$ has no $\ell$-torsion.  \\I think there is a typo on p. 790, and then this group would vanish if $\ell \nmid \zeta(1-n)$, which would be compatible with the result above.}

This allows us to prove the following modularity theorem.
\begin{theorem} \label{main} Recall that we impose Assumptions \ref{admis} and \ref{assumption section 6}. Furthermore, assume    that $\dim H^1_f(\bfQ, {\rm ad}\ov{\rho}_{\phi}(k-2))\leq 1$.
Then
    the structure map $\iota:\Oo\to R$ is an isomorphism.  In particular, if $\tau:G_{\bfQ}\to \GL_4(E)$ is any continuous irreducible homomorphism unramified outside $\ell$, crystalline at $\ell$ with Hodge-Tate weights in $[3-2k,2k-3]$
    % \com{K: unnecessary: ``with determinant $\epsilon^{2k-3}$'' - its reduction has the correct determinant because of the equation below and FL-condition guarantees that the determinant cannot be anything else. In fact this assumption is never used but instead the fact that $\tau$ will have such a determinant follows from the claim $\tau\cong \rho_f$. This is really a consequence of FL assumption because this assumption is what gives us Lemma \ref{diag1} - should we explain all this in a remark?\\J: Yes, it can't hurt.  Referee can always tell us to remove it if s/he doesn't think it is necessary.\\K: So , I added a short remark but refrained from explaining how this works. It works because of various assumptions that we make, that first make Lemma 5.2 work and then ensure that $I_{\rm re}$ is principal. If they ask us to explain, we will. At the moment I left it as "in many deformation problems" without including examples. I'd say let's leave it like this and see what the referee says.}
    and such that $$\ov{\tau}^{\rm ss}=\ov{\rho}_{\phi} \oplus \ov{\rho}_{\phi}(k-2),$$ then $\tau\cong \rho^{\rm univ}\cong \rho_f$, i.e., in particular $\tau$ is modular.
\end{theorem}

%\com{J: Does this mean we obtain a Siegel modular form of full level that has $\tau$ as its Galois representation?\\K: Yes. But in fact, we get that $\tau\cong \rho_f$ - should we state it this way? Under our strict assumptions you don't get any more forms other than the one we started with.}

\begin{proof}
It follows from Corollary \ref{Osurj} that $I_{\rm re}$ is a maximal ideal of $R$. %\com{T: could shorten proof by using that maximal ideal is also principal, rather than using the commutative algebra criterion}
As the deformation $\rho_f$ induces a surjective map $j: R\to \Oo$, we get the following commutative diagram of $\Oo$-algebra maps
\be\label{comm alg crit}\xymatrix{\Oo\ar[r]_{\iota}\ar@/^1pc/[rr]^{\rm id}\ar[d]&R\ar[r]_{j}\ar[d]&\Oo\ar[d]\\
\Oo/\lambda\ar[r]^{\ov{\iota}}\ar@/_1pc/[rr]_{\rm id}&R/I_{\rm re}\ar[r]^{\ov{j}}&\Oo/\lambda}\ee
As $\ov{\iota}$ is an isomorphism, we get that so is $\ov{j}$. So, using the fact that $I_{\rm re}$ is principal (Proposition \ref{princ of I}), we can now apply Theorem 6.9 in \cite{BergerKlosin11} to the right square to conclude that $j$ is an isomorphism.

Now, let $\tau$ be as in the statement of the Theorem. Then $\tau$ factors through a representation of $G_{\{\ell\}}$. Using that $\tau$ is irreducible,  Theorem 4.1 in \cite{BergerKlosin20} allows us to find a lattice in the space of $\tau$ such that with respect to that lattice one has $$\ov{\tau}=\bmat \ov{\rho}_{\phi} & * \\ & \ov{\rho}_{\phi}(k-2)\emat$$ that is non-semi-simple. Using Remark \ref{rem6.4} we see that this lattice is Fontaine-Laffaille, so the star gives rise to a non-zero element in $H^1_f(\bfQ, {\rm ad}^0\rho_{\phi\textbf{}}(2-k)\otimes_{\Oo}E/\Oo)$. As the latter group has order $\#\Oo/\lambda$ by Assumption \ref{assumption section 6}(ii), we conclude that $\ov{\tau}\cong \rho$. In particular, $\tau$ is a deformation of $\rho$. Hence $\tau$ gives rise to an $\Oo$-algebra map $R\to \Oo$, which must equal $j$ by the first part of the theorem.
\end{proof}

\begin{rem}
    We return to Example \ref{ex 1} and note that Assumption \ref{assumption section 6} (i) holds, as discussed earlier. Since $\ell=163$ or $187273$ do not divide $(2k-1)(2k-3)k!$ %\com{T: looking at the proof of \cite{DiamondFlachGuoAnnScEcole04} Lemma 2.5 and using that $\phi$ is ordinary at $\ell$ we don't need to worry about $2k-3$ but maybe not worth discussing this.}
    for $k=26$ and $\ov{\rho}_\phi$ is irreducible, \cite{DiamondFlachGuoAnnScEcole04} Lemma 2.5 proves that  $\ov{\rho}_\phi$ stays irreducible when restricted to $\Gal(\ov{\bfQ}/\bfQ(\sqrt{(-1)^{(\ell-1)/2}\ell}))$. Via Remark \ref{rem 6.5} we can therefore check that $H^1_f(\bfQ, \ad^0\ov{\rho}_{\phi})=0$ as $\phi$ is the only cusp form of weight 26 and level 1, so in particular, $\phi$ is not congruent mod $\ell$ to other forms. Since in addition $L_{\rm alg}(50, {\rm Sym}^2 \phi)$ has $\ell$-valuation 1 for both $\ell=163$ and $187273$ the Bloch-Kato conjecture for $\#H^1_f(\bfQ, {\rm ad}^0 \rho_\phi(2-k) \otimes E/\Oo)=\#\Oo/\lambda$
    (see \cite{DummiganSymmSquareII} Conjecture (5.2) and (5)) would imply that Assumption (ii) holds.

    We do not know how to check $\dim H^1_f(\bfQ, {\rm ad}\ov{\rho}_{\phi}(k-2))\leq 1$, as the corresponding divisible Selmer group is not critical (in the sense of Deligne). Note that $\dim H^1_f(\bfQ, {\rm ad}\ov{\rho}_{\phi}(k-2))=\dim H^1_f(\bfQ, {\rm ad}^0\ov{\rho}_{\phi}(k-2))$ by Proposition \ref{prop6.10} since neither prime $\ell$ divides $B_{24}$.
\end{rem}

%\com{J: For the conclusion that $\ov{\tau} \cong \rho$ because the latter group has order $\#\Oo/\lambda$, is this because $\#\Oo/\lambda$ is a field?  I'm a little confused by this and I am thinking it is probably clear.\\K: Yes. If $k$ is a field the isomorphism classes of non-semi-simple representations $\bmat \sigma_1 &* \\ 0 & \sigma_2\emat$ (where $\sigma_1, \sigma_2$ are irreducible representations valued in $\GL_{n_i}(k)$) are in one-to-one correspondence with lines in $H^1(G, \Hom(\sigma_2, \sigma_1))$. If we insist on counting those that are Fontaine-Laffaille, then $H^1$ gets replaced with $H^1_f$. }

\section{(Non-)principality of Eisenstein ideals}

In this section we formulate conditions when the Eisenstein ideal of the local Hecke algebra  acting on $S_k(\Gamma_2)$ is non-principal and $\dim_{\bfF}H^1_f(\bfQ, \ad^0\ov{\rho}_{\phi}(k-2))>1.$ In particular, in that case $R\not \cong \Oo$.

Let $\bfT'$ be as in Section \ref{sec:intro}.
Let $\bfT$ denote the $\Oo$-subalgebra of $\bfT'\otimes_{\bfZ}\Oo$ generated by the operators $T^{(2)}(p)$ and $T_1^{(2)}(p^2)$ for all primes $p\nmid \ell$.
 Since strong multiplicity one holds in the level one case, we can choose an orthogonal basis $\mN'$ of $S_{k}(\Gamma_{2})$ consisting of eigenforms for all the operators in $\bfT$.

%\com{K: deleted `Let $\mN''$ denote the set of all Hecke eigenforms in $S_k(\Gamma_2)$. Each such eigenform $g$ gives rise to an element $\psi_g$ of the finite set $\Hom_{\Oo-{\rm alg}}(\bfT, \Oo)$ where $\psi_g(T)=\lambda_g(T)$, with $\lambda_g(T)$  the eigenvalue of the operator $T$ corresponding to $g$.'}

 Each $g\in \mN'$ gives rise to $\psi_g\in\Hom_{\Oo-{\rm alg}}(\bfT, \Oo)$ where $\psi_g(T)=\lambda_g(T)$, with $\lambda_g(T)$  the eigenvalue of the operator $T$ corresponding to $g$.
Thus we get a map
$\Psi: \mN'\to \Hom_{\Oo-{\rm alg}}(\bfT, \Oo)$ given by $g\mapsto \lambda_g$, which by strong multiplicity one is an injection.
%\com{K: Perhaps we should include some sort of explanation like this (but should first clarify which of these are actually known in our case: If $\bfT\otimes \ov{\bfQ}_p$ is Gorenstein, then this map is surjective, but we do not make this assumption. On the other hand, if strong multiplicity one holds, then its fibers are one-dimensional.\\J: I think this is fine.  We do have strong multiplicity one in the level one case. }
%\com{K: So, I am not claiming that the map is surjective (because we don't know if we have the duality). I think it does not cause any problems later, but I would appreciate if you could chaeck this as well - i.e., that if there were some Homs not accounted for by eigenforms, this does not screw up any of the arguments below. We know that at least one of the fibers is non-empty because we have $f$ and this seems to be all that we need.\\J: I do not see any issues that arise from this.  I will read it again later and make sure I don't see anything after a second and third reading just to be safe. \\J: have gone through it again, I still think it looks fine.}
%\com{K: deleted `We write $\mN'$ for a complete set of representatives of fibers of this map. '}

\begin{lemma} \label{l7.1}
The natural $\Oo$-algebra map \be \label{nat1} \bfT\to \prod_{g\in \mN'}\Oo \quad \textup{given by} \quad  T\mapsto (\psi_g(T))_{g}\ee  is injective and has finite cokernel, i.e. $\bfT$ can be viewed as a  lattice in $\prod_{g\in \mN'}\Oo$.

\end{lemma}

\begin{proof}
The injectivity follows from the fact that the elements of $\mN'$ form a basis.
    %Any $t \in \bfT$ in the kernel of this map kills every $g \in \mN'$. As the elements of $\mN'$ form a basis of $S_k(\Gamma_2)$, the operator $t$ is  the zero endomorphism. This proves injectivity.

%\com{K: added to explain the finite cokernel business:}
We will now show that the map has finite cokernel. Note that the (set) map $\Psi \otimes \ov{\bfQ}_\ell: \mN'\to \Hom_{\ov{\bfQ}_{\ell}-{\rm alg}}(\bfT\otimes \ov{\bfQ}_{\ell},\ov{\bfQ}_{\ell})\hookrightarrow \Hom_{\ov{\bfQ}_{\ell}}(\bfT\otimes \ov{\bfQ}_{\ell}, \ov{\bfQ}_{\ell})$ given by $g \mapsto \lambda_g \otimes \ov{\bfQ}_\ell$ is injective (because $\Psi$ is injective), and strong multiplicity one implies that no non-trivial linear relation $\sum_{g\in \mN'} c_g \lambda_{g}=0$ can hold. %Consider the form $g_0=\sum_{g\in \mN'} c_g g\in S_k(\Gamma_2)$. It is clear that $g_0$ is an eigenform for all the operators in $\bfT\otimes \ov{\bfQ}_{\ell}$ with all eigenvalues zero. By strong multiplicity one we conclude that $g_0=0$. As the elements of $\mN'$ (being an orthogonal basis of $S_k(\Gamma_2)$) form a linearly independent set, we get that $c_g=0$ for all $g\in \mN'$.
Thus the set $\{\lambda_g\mid g\in \mN'\}$ is a linearly independent subset of $\Hom_{\ov{\bfQ}_{\ell}}(\bfT\otimes \ov{\bfQ}_{\ell}, \ov{\bfQ}_{\ell})$. Hence \be \label{dim1} \dim_{\ov{\bfQ}_{\ell}}\bfT\otimes \ov{\bfQ}_{\ell}=\dim_{\ov{\bfQ}_{\ell}}\Hom_{\ov{\bfQ}_{\ell}}(\bfT\otimes \ov{\bfQ}_{\ell}, \ov{\bfQ}_{\ell})\geq \# \mN'.\ee

%\com{T: do we need Gorensteinness for the first equality?}

 Tensoring the map \eqref{nat1} with $\ov{\bfQ}_{\ell}$ we get a corresponding map  $\bfT\otimes\ov{\bfQ}_{\ell}\to \prod_{g\in \mN'}\ov{\bfQ}_{\ell}$, which is injective as \eqref{nat1} is.  Thus it must be surjective by \eqref{dim1}. Hence the map \eqref{nat1} has finite cokernel.
\end{proof}

We now identify $\bfT$ with the image of the map \eqref{nat1} and note that %$\bfT$ is a semi-local, complete, reduced $\Oo$-algebra and one has the following decomposition
$\bfT=\prod_{\fm \in {\rm MaxSpec}\bfT}\bfT_{\fm},$ where $\bfT_{\fm}$ is the localization of $\bfT$ at the maximal ideal $\fm$.
Let $\mN$ be the subset of $\mN'$ consisting of all the $g\in \mN'$ which satisfy $$\psi_g(T)\equiv \lambda_{E_{\phi}^{1,2}}(T) \pmod{\lambda}\quad \textup{for all $T\in \bfT$}.$$

We write $\fm$ for the corresponding maximal ideal. %Then there is a maximal ideal $\fm\in {\rm MaxSpec}\bfT$ such that the map $\bfT\to \prod_{g\in \mN'}\Oo\to \prod_{g\in \mN}\Oo$ factors through $\bfT_{\fm}$. We fix this $\fm$ from now on.
Set $J\subset \bfT$ to be the Eisenstein ideal, i.e., $J$ is the ideal of $\bfT$ generated by the set
$\{T^{(2)}(p)-(\tr \rho_{\phi}(\Frob_p)+\tr\rho_{\phi}(k-2)(\Frob_p))\mid p\neq \ell\}.$
%\com{J: changed $T_p$ to $T^{(2)}(p)$ in definition}
Write $J_{\fm}$ to be the image of $J$ under the canonical map $\bfT\to \bfT_{\fm}$.

Recall that  we fixed in Section \ref{s5.2} the weight $k \geq 12$ even and prime $\ell > 4k-5$ and imposed Assumption \ref{admis} on the field $E/\bfQ_\ell$. We also fixed the Fontaine-Laffaille interval $I=[3-2k, 2k-3]$.
Let $\phi\in S_k(\Gamma_1)$ be a newform such that $\ov{\rho}_{\phi}$ is irreducible.

For the rest of this section we also impose Assumption \ref{assumption section 6} and fix the corresponding $f \in S_k(\Gamma_2)$.  Then $f\in \mN$, i.e., $\bfT_{\fm}/J_{\fm}\neq 0$.
Let $R=R_{\ov{\rho}_f}$ be the universal deformation ring defined in Section \ref{Modularity}.

\begin{theorem} \label{R=T thm}
Recall that we impose Assumptions \ref{admis} and \ref{assumption section 6}.  Then
there exists a surjective $\Oo$-algebra map $\varphi: R\to \bfT_{\fm}$ such that $\varphi(I_{\rm re})=J_{\fm}$ and $J_{\fm}$ is a maximal ideal of $\bfT_{\fm}$. If, in addition $\dim_{\bfF} H^1_f(\bfQ, {\rm ad}\ov{\rho}_{\phi}(k-2))\leq 1$, then all of the following are true:
\begin{itemize}
    \item the map $\varphi$ is an isomorphism;
    \item the Hecke ring $\bfT_{\fm}$ is isomorphic to $\Oo$;
    \item the Eisenstein ideal $J_{\fm}$ is principal.
\end{itemize}

\end{theorem}

\begin{proof} Let $g\in \mN$.
Then by Lemma \ref{Ribet1} there exists a $G_{\bfQ}$-stable lattice with respect to which one has $\ov{\rho}_{g}=\bmat \ov{\rho}_{\phi}&* \\  & \ov{\rho}_{\phi}(k-2)\emat$
and is not semi-simple. Hence the $*$ gives rise to an element in $H^1_f(\bfQ, W[\lambda])$, where  $W={\rm ad}^0\rho_{\phi}(2-k)\otimes_{\Oo}E/\Oo$.

By \eqref{trivial invariant}  and Proposition \ref{prop5.8} we get $H^1_f(\bfQ, W[\lambda])=H^1_f(\bfQ, W)[\lambda]$. The latter group is cyclic by Assumption \ref{assumption section 6} (ii), so we must have that $\ov{\rho}_{g}\cong \ov{\rho}_f$, and so after adjusting the basis if necessary we get that $\rho_{g}$ is a deformation of $\ov{\rho}_f$.

This implies that for every $g\in \mN$ we get an $\Oo$-algebra (hence continuous) map $\varphi_{g}:R\to \Oo$ with the property that $\tr\rho^{\rm univ}(\Frob_p) \mapsto \lambda_{g}(T^{(2)}(p))$. This property completely determines $\varphi_{g}$ because $R$ is topologically generated by the set $\{\tr\rho^{\rm univ}(\Frob_p)\mid p\neq \ell\}$ by Proposition 7.13 in \cite{BergerKlosin13}.  Putting these maps together we get an $\Oo$-algebra map $\varphi: R\to \prod_{g\in \mN}\Oo$ whose image is an $\Oo$-subalgebra of $\prod_{g\in \mN}\Oo$ generated by $\{T^{(2)}(p)\mid p\neq \ell\}$. %We now claim that this subalgebra equals $\bfT_{\fm}$. Indeed, one clearly has
Note that $\varphi(R)\subset \bfT_{\fm}$. To see the opposite inclusion consider the characteristic polynomial $f_p(X)\in R[X]$ of $\rho^{\rm univ}(\Frob_p)$ for $p\neq \ell$.
Combining Theorem \ref{Weissauer} with the definition of $L_p(X, f; {\rm spin})$ we see that the coefficient at $X^2$ is mapped by $\varphi$ to $T^{(2)}(p)^2-T^{(2)}(p^2)-p^{2k-4}\in \prod_{g\in \mN}\Oo$. As $T^{(2)}(p)$ and $p^{2k-4}$ both belong to  $\varphi(R)$, so therefore must $T^{(2)}(p^2)$. We now use the fact (\cite[3.3.38]{Andrianov}, \cite[pg. 547]{JohnsonLeung}) that
\begin{equation*}
p T_1^{(2)}(p^2) = T^{(2)}(p)^2 - T^{(2)}(p^2) - p(p^2+p+1) T(\diag(p,p,p,p))
\end{equation*}
to conclude that $T_1^{(2)}(p^2) \in \varphi(R)$.  Hence $\varphi(R)$ contains all the Hecke operators away from $\ell$, i.e., $\varphi(R)=\bfT_{\fm}$.
 We denote the resulting $\Oo$-algebra epimorphism $R\to \bfT_{\fm}$ again by $\varphi$.
We claim that $\varphi(I_{\rm re})\subset J_{\fm}$.

Indeed, %as $\varphi(\tr\rho^{\rm univ}(\Frob_p))=T^{(2)}(p)$ and $$T^{(2)}(p)-(\tr \rho_{\phi}(\Frob_p)+\tr \rho_{\phi}(k-2)(\Frob_p))\in J_{\fm} $$ for all primes $p\neq \ell$, one has $$\tr\rho^{\rm univ}(\Frob_p)-(\tr \rho_{\phi}(\Frob_p)+\tr \rho_{\phi}(k-2)(\Frob_p))\in \varphi^{-1}(J_{\fm}).$$
using the Chebotarev Density Theorem, one sees that
$$\tr \rho^{\rm univ}\equiv \tr \rho_{\phi}+\tr \rho_{\phi}(k-2)\pmod{\varphi^{-1}(J_{\fm})},$$ so $I_{\rm re}\subset \varphi^{-1}(J_{\fm})$.
As $\varphi$ is a surjection, this implies that $\varphi(I_{\rm re})\subset J_{\fm}$.
Hence $\varphi$ gives rise to a sequence of $\Oo$-algebra surjections $R/I_{\rm re}\to \bfT_{\fm}/\varphi(I_{\rm re})\to \bfT_{\fm}/J_{\fm}$. As $R/I_{\rm re}=\bfF$ by Corollary \ref{Osurj} we conclude that all these surjections are isomorphisms (note that $\bfT_{\fm}/J_{\fm}\neq 0$), hence $\varphi(I_{\rm re})=J_{\fm}$ and $J_{\fm}$ is maximal. This proves the first claim.

Now  assume in addition that $\dim H^1_f(\bfQ, {\rm ad}\ov{\rho}_{\phi}(k-2))\leq 1$. Then
 Theorem \ref{main} gives us that $R=\Oo$, so we get that $\varphi$ is an isomorphism, and so $R\cong \bfT_{\fm}\cong \Oo$. Hence $J_{\fm}$ is a principal ideal.
 %Identifying $R$ with $\bfT_{\fm}$ via $\varphi$, we see that for all primes $p\neq \ell$ one has \begin{multline}\tr\rho^{\rm univ}(\Frob_p)-(\tr \rho_{\phi}(\Frob_p)+\tr \rho_{\phi}(k-2)(\Frob_p))\\=T_p-(\tr \rho_{\phi}(\Frob_p)+\tr \rho_{\phi}(k-2)(\Frob_p))\in J_{\fm},\end{multline} hence $$\tr \rho^{\rm univ}(\Frob_p)\equiv \tr \rho_{\phi}(\Frob_p)+\tr \rho_{\phi}(k-2)(\Frob_p) \pmod {J_{\fm}},$$ which implies $R\supsetneq J_{\fm}\supset I_{\rm re}$ (the first inclusion is strict because $\bfT_{\fm}/J_{\fm}\neq 0)$. The fact that the map $\ov{j}$ in diagram \eqref{comm alg crit} is an isomorphism implies that the ideal $I_{\rm re}$ is maximal. Thus we conclude that $J_{\fm}=I_{\rm re}$, so $J_{\fm}$ is maximal.
\end{proof}

\begin{corollary} \label{not cyclic}
If $J_{\fm}$ is not principal, then $\dim_{\bfF}H^1_f(\bfQ, \ad\ov{\rho}_{\phi}(k-2))>1.$ If in addition $\ell \nmid B_{k-2}$ then
$\dim_{\bfF}H^1_f(\bfQ, \ad^0\ov{\rho}_{\phi}(k-2))>1.$
\end{corollary}
\begin{proof}
    The first inequality is just a restatement of one of the claims of Theorem \ref{R=T thm}.
 The second follows from the first one and Proposition \ref{prop6.10}.
\end{proof}

%The latter can potentially be done by counting the depth of Eisenstein congruences as the following result shows.

\begin{prop} \label{BKK}
 For each $g\in \mN$ write $m_{g}$ for the largest positive integer $m$ such that $g\equiv E_{2,1}^{\phi}$ mod $\lambda^{m}$. If \be \label{BKK1} \val_{\ell}(\# \bfT_{\fm}/J_{\fm})<[\bfF:\bfF_{\ell}] \cdot \sum_{g\in \mN}m_{g}\ee then $J_{\fm}$ is not principal.

\end{prop}

\begin{proof}
    Set $A = \prod_{g\in \mN}A_{g}$, where $A_{g}=\Oo$ for all $g\in \mN$. Let $\phi_{g} : A \to A_{g}$ be
the canonical projection. Since by Lemma \ref{l7.1} $\bfT$ is a full rank $\Oo$-submodule of $\prod_{g \in \mN'} \Oo$  we conclude that the local complete $\Oo$-subalgebra $\bfT_{\fm} \subset A$ is of full
rank as an $\Oo$-submodule and $J_{\fm} \subset \bfT_{\fm}$ is an ideal of finite index. Set $T_{g} = \phi_{g}(\bfT_{\fm})=A_{g}=\Oo$
and $J_{g} = \phi_{g}(J_{\fm})=\lambda^{m_{g}}\Oo$.  Hence we are in the setup of Section 2 of \cite{BergerKlosinKramer13}. Assume $J_{\fm}$ is principal. Then Proposition 2.3 in \cite{BergerKlosinKramer13} gives us that \be \label{BKK2} \#\bfT_{\fm}/J_{\fm}=\prod_{g\in \mN} \#T_{g}/J_{g}.\ee
Note that one has \be \label{eq7.5}\val_{\ell}\left(\prod_{g\in \mN} \#T_{g}/J_{g}\right)= [\bfF:\bfF_{\ell}] \cdot \sum_{g\in \mN}m_{g}.\ee
This equality together with \eqref{BKK2} contradicts the inequality \eqref{BKK1}. \end{proof}

\begin{corollary} \label{not cyclic 2}
Let $m_{g}$ be defined as in Proposition \ref{BKK}.  If $ \sum_{g\in \mN}m_{g}>1$ then $J_{\fm}$ is not principal and  $\dim_{\bfF}H^1_f(\bfQ, \ad\ov{\rho}_{\phi}(k-2))>1.$ If in addition $\ell \nmid B_{k-2}$ then
$\dim_{\bfF}H^1_f(\bfQ, \ad^0\ov{\rho}_{\phi}(k-2))>1.$
\end{corollary}

\begin{proof}
Note that from the proof of Theorem \ref{R=T thm} we get that $\bfT_{\fm}/J_{\fm}=\bfF$, even without assuming $\dim_{\bfF}H^1_f(\bfQ, \ad\ov{\rho}_{\phi}(k-2))\leq 1$.
Assume that $J_\fm$ is principal. Then from \eqref{BKK2} and \eqref{eq7.5} we conclude that $\sum_{g\in \mN}m_{g}=1$, which contradicts our assumption. Hence $J_\fm$ is not principal.
The Selmer group inequalities now follow from Corollary \ref{not cyclic}.
\end{proof}

%\com{T: should we add a remark about the example, i.e. that in that case the total depth of congruences is just one, which means that $\dim_{\bfF}H^1_f(\bfQ, \ad\ov{\rho}_{\phi}(k-2))\leq 1$ is possible?}

\begin{rem}
    Corollary \ref{not cyclic} directly ties the cyclicity of the non-critical Selmer group $H^1_f(\bfQ, \ad\ov{\rho}_{\phi}(k-2))$ with the principality of the Eisenstein ideal $J_\fm$. We note that Assumption \ref{assumption section 6}(ii) implies the equality $\bfT_\fm/J_\fm=\bfF$. Contrary to what one might think, the existence of several forms $g\equiv E_{2,1}^{\phi}$ mod $\lambda$   does not preclude this equality. For example, if there are exactly two linearly independent eigenforms $g_1, g_2 \in \mN$ with $m_{g_1}=m_{g_2}=1$ such that $g_1 \not \equiv g_2 \mod{\lambda^2}$ then $\bfT_\fm \cong \Oo \times_\bfF \Oo=\{(a,b) \in \Oo \times \Oo \mid a \equiv b \mod{\lambda}\}$ and in this case  $J_\fm$ is the maximal ideal, i.e. $\bfT_\fm/J_\fm=\bfF$, so Corollary \ref{not cyclic 2} applies and $\dim_{\bfF}H^1_f(\bfQ, \ad\ov{\rho}_{\phi}(k-2))>1$.

\end{rem}

\bibliographystyle{apalike}
\bibliography{new_bib.bib}

\begin{thebibliography}{}

\bibitem[Akers, 2025]{Akers24}
Akers, G. (2025).
\newblock Galois deformation rings and modularity in the residually reducible
  case.
\newblock {\em International Journal of Number Theory}, 21(02):449--471.

\bibitem[Andrianov, 1987]{Andrianov}
Andrianov, A. (1987).
\newblock {\em Quadratic forms and {H}ecke operators}, volume 286 of {\em
  Grundlehren der mathematischen Wissenschaften [Fundamental Principles of
  Mathematical Sciences]}.
\newblock Springer-Verlag, Berlin.

\bibitem[Assim, 1995]{Assim95}
Assim, J. (1995).
\newblock Codescente en {$K$}-th\'{e}orie \'{e}tale et corps de nombres.
\newblock {\em Manuscripta Math.}, 86(4):499--518.

\bibitem[Bella{\"{\i}}che and Chenevier, 2009]{BellaicheChenevierbook}
Bella{\"{\i}}che, J. and Chenevier, G. (2009).
\newblock $p$-adic families of {G}alois representations and higher rank
  {S}elmer groups.
\newblock {\em Ast\'erisque}, (324).

\bibitem[Berger and Klosin, 2011]{BergerKlosin11}
Berger, T. and Klosin, K. (2011).
\newblock {R=T} theorem for imaginary quadratic fields.
\newblock {\em Math. Ann.}, 349(3):675--703.

\bibitem[Berger and Klosin, 2013]{BergerKlosin13}
Berger, T. and Klosin, K. (2013).
\newblock On deformation rings of residually reducible {G}alois representations
  and ${R}={T}$ theorems.
\newblock {\em Math. Ann.}, 355(2):481--518.

\bibitem[Berger and Klosin, 2015]{BergerKlosin15}
Berger, T. and Klosin, K. (2015).
\newblock On lifting and modularity of reducible residual {G}alois
  representations over imaginary quadratic fields.
\newblock {\em Int. Math. Res. Not. IMRN}, (20):10525--10562.

\bibitem[Berger and Klosin, 2019]{BergerKlosin19}
Berger, T. and Klosin, K. (2019).
\newblock Modularity of residual {G}alois extensions and the {E}isenstein
  ideal.
\newblock {\em Trans. Amer. Math. Soc.}, 372(11):8043--8065.

\bibitem[Berger and Klosin, 2020]{BergerKlosin20}
Berger, T. and Klosin, K. (2020).
\newblock Deformations of {S}aito-{K}urokawa type and the paramodular
  conjecture.
\newblock {\em Amer. J. Math.}, 142(6):1821--1875.
\newblock With and appendix by Chris Poor, Jerry Shurman, and David S. Yuen.

\bibitem[Berger and Klosin, 2023]{BergerKlosin23}
Berger, T. and Klosin, K. (2023).
\newblock {$R=T$} theorems for weight one modular forms.
\newblock {\em Trans. Amer. Math. Soc.}, 376(11):8095--8128.

\bibitem[Berger et~al., 2014]{BergerKlosinKramer13}
Berger, T., Klosin, K., and Kramer, K. (2014).
\newblock On higher congruences between automorphic forms.
\newblock {\em Math. Res. Lett.}, 21(1):71--82.

\bibitem[Bloch and Kato, 1990]{BlochKato}
Bloch, S. and Kato, K. (1990).
\newblock {$L$}-functions and {T}amagawa numbers of motives.
\newblock In {\em The {G}rothendieck {F}estschrift, {V}ol. {I}}, volume~86 of
  {\em Progr. Math.}, pages 333--400. Birkh\"{a}user Boston, Boston, MA.

\bibitem[Booher, 2019]{Booher19}
Booher, J. (2019).
\newblock Producing geometric deformations of orthogonal and symplectic
  {G}alois representations.
\newblock {\em J. Number Theory}, 195:115--158.

\bibitem[Breuil, 1998]{Breuil1998}
Breuil, C. (1998).
\newblock Cohomologie \'{e}tale de {$p$}-torsion et cohomologie cristalline en
  r\'{e}duction semi-stable.
\newblock {\em Duke Math. J.}, 95(3):523--620.

\bibitem[Breuil, 2001]{BreuilBarcelona}
Breuil, C. (2001).
\newblock $p$-adic {H}odge theory, deformations and local {L}anglands.
\newblock
  https://www.imo.universite-paris-saclay.fr/~christophe.breuil/PUBLICATIONS/Barcelone.pdf.

\bibitem[Brown, 2007]{BrownCompMath07}
Brown, J. (2007).
\newblock Saito-{K}urokawa lifts and applications to the {B}loch-{K}ato
  conjecture.
\newblock {\em Compos. Math.}, 143(2):290--322.

\bibitem[Calegari, 2006]{Calegari06}
Calegari, F. (2006).
\newblock Eisenstein deformation rings.
\newblock {\em Compos. Math.}, 142(1):63--83.

\bibitem[Clozel et~al., 2008]{ClozelHarrisTaylor08}
Clozel, L., Harris, M., and Taylor, R. (2008).
\newblock Automorphy for some {$l$}-adic lifts of automorphic mod {$l$}
  {G}alois representations.
\newblock {\em Publ. Math. Inst. Hautes \'{E}tudes Sci.}, (108):1--181.
\newblock With Appendix A, summarizing unpublished work of Russ Mann, and
  Appendix B by Marie-France Vign\'{e}ras.

\bibitem[Cornell et~al., 1997]{MazurDefTheory}
Cornell, G., Silverman, J., and Stevens, G., editors (1997).
\newblock {\em Modular forms and {F}ermat's last theorem}.
\newblock Springer-Verlag, New York.
\newblock Papers from the Instructional Conference on Number Theory and
  Arithmetic Geometry held at Boston University, Boston, MA, August 9--18,
  1995.

\bibitem[Diamond et~al., 2004]{DiamondFlachGuoAnnScEcole04}
Diamond, F., Flach, M., and Guo, L. (2004).
\newblock The {T}amagawa number conjecture of adjoint motives of modular forms.
\newblock {\em Ann. Sci. \'Ecole Norm. Sup. (4)}, 37(5):663--727.

\bibitem[Dummigan, 2001]{DummiganSymmSquare}
Dummigan, N. (2001).
\newblock Symmetric square {$L$}-functions and {S}hafarevich-{T}ate groups.
\newblock {\em Experiment. Math.}, 10(3):383--400.

\bibitem[Dummigan, 2009]{DummiganSymmSquareII}
Dummigan, N. (2009).
\newblock Symmetric square {$L$}-functions and {S}hafarevich-{T}ate groups.
  {II}.
\newblock {\em Int. J. Number Theory}, 5(7):1321--1345.

\bibitem[Eisenbud, 1995]{Eisenbud}
Eisenbud, D. (1995).
\newblock {\em Commutative algebra}, volume 150 of {\em Graduate Texts in
  Mathematics}.
\newblock Springer-Verlag, New York.
\newblock With a view toward algebraic geometry.

\bibitem[Fontaine, 1982]{FontaineAnnMath82}
Fontaine, J.-M. (1982).
\newblock Sur certains types de repr\'esentations {$p$}-adiques du groupe de
  {G}alois d'un corps local;\ construction d'un anneau de {B}arsotti-{T}ate.
\newblock {\em Ann. of Math. (2)}, 115(3):529--577.

\bibitem[Fontaine and Laffaille, 1982]{FoLa}
Fontaine, J.-M. and Laffaille, G. (1982).
\newblock Construction de repr\'{e}sentations {$p$}-adiques.
\newblock {\em Ann. Sci. \'{E}cole Norm. Sup. (4)}, 15(4):547--608.

\bibitem[Fontaine and Ouyang, 2022]{Fontainebook}
Fontaine, J.-M. and Ouyang, Y. (2022).
\newblock Theory of p-adic {G}alois {R}epresentations.
\newblock http://staff.ustc.edu.cn/~yiouyang/galoisrep.pdf.

\bibitem[Hattori, 2019]{Hattori19}
Hattori, S. ([2019] \copyright 2019).
\newblock Integral {$p$}-adic {H}odge theory and ramification of crystalline
  representations.
\newblock In {\em An excursion into {$p$}-adic {H}odge theory: from foundations
  to recent trends}, volume~54 of {\em Panor. Synth\`eses}, pages 159--203.
  Soc. Math. France, Paris.

\bibitem[Hida, 2000]{HidaMFG}
Hida, H. (2000).
\newblock {\em Modular forms and {G}alois cohomology}, volume~69 of {\em
  Cambridge Studies in Advanced Mathematics}.
\newblock Cambridge University Press, Cambridge.

\bibitem[Huang, pear]{Huang24}
Huang, X. (to appear).
\newblock On the universal deformation ring of a residual galois representation
  with three jordan holder factors and modularity.
\newblock {\em Kyoto Journal of Mathematics}.

\bibitem[Johnson-Leung and Roberts, 2012]{JohnsonLeung}
Johnson-Leung, J. and Roberts, B. (2012).
\newblock Siegel modular forms of degree two attached to {H}ilbert modular
  forms.
\newblock {\em J. Number Theory}, 132(4):543--564.

\bibitem[Kalloniatis, 2019]{Kalloniatis19}
Kalloniatis, T. (2019).
\newblock On flagged framed deformation problems of local crystalline {G}alois
  representations.
\newblock {\em J. Number Theory}, 199:229--250.

\bibitem[Katsurada and Mizumoto, 2012]{KatsuradaMizumoto}
Katsurada, H. and Mizumoto, S. (2012).
\newblock Congruences for {H}ecke eigenvalues of {S}iegel modular forms.
\newblock {\em Abh. Math. Semin. Univ. Hambg.}, 82(2):129--152.

\bibitem[Klingen, 1990]{klingen}
Klingen, H. (1990).
\newblock {\em Introductory lectures on {S}iegel modular forms}, volume~20 of
  {\em Cambridge Studies in Advanced Mathematics}.
\newblock Cambridge University Press, Cambridge.

\bibitem[Klosin, 2009]{KlosinAnnInstFourier2009}
Klosin, K. (2009).
\newblock Congruences among modular forms on {${\rm U}(2,2)$} and the
  {B}loch-{K}ato conjecture.
\newblock {\em Ann. Inst. Fourier (Grenoble)}, 59(1):81--166.

\bibitem[Kurokawa, 1979]{Kurokawa1}
Kurokawa, N. (1979).
\newblock Congruences between {S}iegel modular forms of degree two.
\newblock {\em Proc. Japan Acad. Ser. A Math. Sci.}, 55(10):417--422.

\bibitem[Kurokawa, 1981]{Kurokawa2}
Kurokawa, N. (1981).
\newblock Congruences between {S}iegel modular forms of degree two. {II}.
\newblock {\em Proc. Japan Acad. Ser. A Math. Sci.}, 57(2):140--145.

\bibitem[Matsumura, 1989]{Matsumura}
Matsumura, H. (1989).
\newblock {\em Commutative ring theory}, volume~8 of {\em Cambridge Studies in
  Advanced Mathematics}.
\newblock Cambridge University Press, Cambridge, second edition.
\newblock Translated from the Japanese by M. Reid.

\bibitem[Mizumoto, 1984]{MizumotoKodai84}
Mizumoto, S. (1984).
\newblock Fourier coefficients of generalized {E}isenstein series of degree
  two. {II}.
\newblock {\em Kodai Math. J.}, 7(1):86--110.

\bibitem[Mizumoto, 1986]{MizumotoCongruences}
Mizumoto, S. (1986).
\newblock Congruences for eigenvalues of {H}ecke operators on {S}iegel modular
  forms of degree two.
\newblock {\em Math. Ann.}, 275(1):149--161.

\bibitem[Neukirch et~al., 2008]{NSW}
Neukirch, J., Schmidt, A., and Wingberg, K. (2008).
\newblock {\em Cohomology of number fields}, volume 323 of {\em Grundlehren der
  mathematischen Wissenschaften [Fundamental Principles of Mathematical
  Sciences]}.
\newblock Springer-Verlag, Berlin, second edition.

\bibitem[Niziol, 1993]{Niziol93}
Niziol, W. (1993).
\newblock Cohomology of crystalline representations.
\newblock {\em Duke Math. J.}, 71(3):747--791.

\bibitem[Pitale and Schmidt, 2009]{PitaleSchmidtRamanujan}
Pitale, A. and Schmidt, R. (2009).
\newblock Ramanujan-type results for {S}iegel cusp forms of degree 2.
\newblock {\em J. Ramanujan Math. Soc.}, 24(1):87--111.

\bibitem[Ramakrishna, 1993]{Ramakrishna}
Ramakrishna, R. (1993).
\newblock On a variation of {M}azur's deformation functor.
\newblock {\em Compositio Math.}, 87(3):269--286.

\bibitem[Ramakrishna, 2002]{Ramakrishna02}
Ramakrishna, R. (2002).
\newblock Deformations of certain reducible {G}alois representations.
\newblock {\em J. Ramanujan Math. Soc.}, 17(1):51--63.

\bibitem[Ribet, 1976]{RibetInvent76}
Ribet, K. (1976).
\newblock A modular construction of unramified {$p$}-extensions of{$Q(\mu
  _{p})$}.
\newblock {\em Invent. Math.}, 34(3):151--162.

\bibitem[Shimura, 1975]{ShimuraFC}
Shimura, G. (1975).
\newblock On the {F}ourier coefficients of modular forms of several variables.
\newblock {\em Nachr. Akad. Wiss. G\"ottingen Math.-Phys. Kl. II},
  (17):261--268.

\bibitem[Shimura, 1976]{ShimuraCommPureAppliedMath1976}
Shimura, G. (1976).
\newblock The special values of the zeta functions associated with cusp forms.
\newblock {\em Comm. Pure Appl. Math.}, 29(6):783--804.

\bibitem[Skinner and Urban, 2006]{SkinnerUrbanJussieu06}
Skinner, C. and Urban, E. (2006).
\newblock Sur les d\'eformations {$p$}-adiques de certaines repr\'esentations
  automorphes.
\newblock {\em J. Inst. Math. Jussieu}, 5(4):629--698.

\bibitem[Skinner and Urban, 2014]{SkinnerUrbanMC}
Skinner, C. and Urban, E. (2014).
\newblock The {I}wasawa main conjectures for {$\rm GL_2$}.
\newblock {\em Invent. Math.}, 195(1):1--277.

\bibitem[Skinner and Wiles, 1997]{SkinnerWiles97}
Skinner, C.~M. and Wiles, A.~J. (1997).
\newblock Ordinary representations and modular forms.
\newblock {\em Proc. Nat. Acad. Sci. U.S.A.}, 94(20):10520--10527.

\bibitem[Sturm, 1980]{SturmAJM1980}
Sturm, J. (1980).
\newblock Special values of zeta functions, and {E}isenstein series of half
  integral weight.
\newblock {\em Amer. J. Math.}, 102(2):219--240.

\bibitem[Takeda, 2025]{Takeda}
Takeda, N. (2025).
\newblock Kurokawa-{M}izumoto congruence and differential operators on
  automorphic forms.
\newblock {\em J. Number Theory}, 266:98--130.

\bibitem[Urban, 2001]{Urban2001}
Urban, E. (2001).
\newblock Selmer groups and the {E}isenstein-{K}lingen ideal.
\newblock {\em Duke Math. J.}, 106(3):485--525.

\bibitem[Wake, 2023]{Wake23}
Wake, P. (2023).
\newblock The {E}isenstein ideal for weight {$k$} and a {B}loch-{K}ato
  conjecture for tame families.
\newblock {\em J. Eur. Math. Soc. (JEMS)}, 25(7):2815--2861.

\bibitem[Wake and Wang-Erickson, 2019]{WWEForum}
Wake, P. and Wang-Erickson, C. (2019).
\newblock Deformation conditions for pseudorepresentations.
\newblock {\em Forum Math. Sigma}, 7:Paper No. e20, 44.

\bibitem[Wake and Wang-Erickson, 2020]{WWE20}
Wake, P. and Wang-Erickson, C. (2020).
\newblock The rank of {M}azur's {E}isenstein ideal.
\newblock {\em Duke Math. J.}, 169(1):31--115.

\bibitem[Washington, 1982]{WashingtonCyclotomic}
Washington, L.~C. (1982).
\newblock {\em Introduction to cyclotomic fields}, volume~83 of {\em Graduate
  Texts in Mathematics}.
\newblock Springer-Verlag, New York.

\bibitem[Washington, 1997]{WashingtonBoston}
Washington, L.~C. (1997).
\newblock Galois cohomology.
\newblock In {\em Modular forms and {F}ermat's last theorem ({B}oston, {MA},
  1995)}, pages 101--120. Springer, New York.

\bibitem[Weissauer, 2005]{WeissauerAsterique05}
Weissauer, R. (2005).
\newblock Four dimensional {G}alois representations.
\newblock {\em Ast\'erisque}, (302):67--150.
\newblock Formes automorphes. II. Le cas du groupe ${\rm{G}}Sp(4)$.

\bibitem[Wiles, 1990]{Wiles90}
Wiles, A. (1990).
\newblock The {I}wasawa conjecture for totally real fields.
\newblock {\em Ann. of Math. (2)}, 131(3):493--540.

\bibitem[Yamauchi, 2021]{Yamauchi}
Yamauchi, T. (2021).
\newblock Congruences of {S}iegel {E}isenstein series of degree two.
\newblock {\em Manuscripta Math.}, 166(3-4):589--603.

\bibitem[Zagier, 1977]{ZagierSpringer77}
Zagier, D. (1977).
\newblock {\em Modular forms whose {F}ourier coefficients involve
  zeta-functions of quadratic fields}.
\newblock Springer, Berlin.

\end{thebibliography}

% \begin{thebibliography}{9}
% \bibitem{bib1}
% A. Ananin and A. Mironov,  The moduli space of $2$-dimensional algebras, \textit{Comm. Algebra}, 28 (2000), 9,  {4481}--{4488}.

% \bibitem{bib2}
% C. Bai and D. Meng,  The classification of Novikov algebras in low dimension,  \textit{J. Phys. A: Math. Gen}., 34 (2001), {1581}--{1594}.

% \bibitem{bib3}
% E. Ca\~{n}ete and A. Khudoyberdiyev,  The classification of $4$-dimensional Leibniz algebras,  \textit{Linear Algebra and its Applications}, 439 (2013), 1, {273}--{288}.

% \bibitem{bib4}
% M. Goze and E. Remm,  $2$-dimensional algebras,  \textit{Afr. J. Math. Phys}., 10 (2011), 1,  {81}--{91}.

% \bibitem{bib5}
% H. Petersson,  The classification of two-dimensional nonassociative algebras,  \textit{Results Math}., 37 (2000), no. 1-2,  {120}--{154}.

% \bibitem{bib6}
% L. Snobl and P. Winternitz, Classification and Identification of Lie algebras,  \textit{CRM Monograph Series}, 33  (2014).

% \end{thebibliography}

%\printaddress

%\end{Backmatter}

\end{document}